%% file: main.tex
\documentclass[onefignum,onetabnum]{siamart250211}

\input{preamble}
\usepackage[skip=4pt]{caption}
\usepackage{subcaption}

\ifpdf
\hypersetup{
  pdftitle={Algebraic Multigrid with Filtering: An Efficient Preconditioner for Interior Point Methods in Large-Scale Contact Mechanics Optimization},
  pdfauthor={Socratis Petrides, Tucker Hartland, Tzanio Kolev, Chak Shing Lee, Michael Puso, Jerome Solberg, Eric B. Chin, Jingyi Wang, Cosmin Petra}
}
\fi

\headers{AMG with Filtering for Contact Optimization}{Petrides, Hartland, Kolev, Lee, Puso, Solberg, Chin, Wang, Petra}

\title{Algebraic Multigrid with Filtering: An Efficient Preconditioner for Interior Point Methods in Large-Scale Contact Mechanics Optimization
\thanks{Submitted to the editors October 13, 2025.}
\funding{This work was performed under the auspices of the U.S. Department of Energy by Lawrence Livermore National Laboratory under contract DE–AC52–07NA27344 (LLNL–JRNL–2006438). The research was supported by the LLNL-LDRD Program under project 23–ERD–017.}}

\author{
  Socratis Petrides\thanks{\protect Lawrence Livermore National Laboratory, Livermore, CA, 94551, USA
  (\email{petrides1@llnl.gov}, \email{hartland1@llnl.gov}, \email{kolev1@llnl.gov}, \email{lee1029@llnl.gov},
    \email{puso1@llnl.gov},
   \email{solberg2@llnl.gov}, \email{chin23@llnl.gov}, \email{wang125@llnl.gov}, \email{petra1@llnl.gov})}
   \and Tucker Hartland\footnotemark[2]
  \and Tzanio Kolev\footnotemark[2]
  \and Chak Shing Lee\footnotemark[2]
\and Michael Puso\footnotemark[2]
\and Jerome Solberg\footnotemark[2]
\and Eric B. Chin\footnotemark[2]
\and Jingyi Wang\footnotemark[2]
\and Cosmin Petra\footnotemark[2]
}

\newcounter{todocounter}

\newif\ifshowrevised
\showrevisedfalse  

\ifshowrevised
  \colorlet{revised}{darkblue}
\else
  \colorlet{revised}{.} 
\fi

\begin{document}

\maketitle
\begin{abstract}
Large-scale contact mechanics simulations are crucial in many engineering fields such as structural design and manufacturing. In the frictionless case, contact can be modeled by minimizing an energy functional; however, these problems are often nonlinear, nonconvex, and increasingly difficult to solve as mesh resolution increases. In this work, we employ a Newton-based interior-point (IP) filter line-search method, an effective approach for large-scale constrained optimization. While this method converges rapidly, each iteration requires solving a large saddle-point linear system that becomes ill-conditioned as the optimization process converges, largely due to IP treatment of the contact constraints.
Such ill-conditioning can hinder solver scalability and increase iteration counts with mesh refinement. To address this, we introduce a novel preconditioner, algebraic multigrid with filtering (AMGF), tailored to the Schur complement of the saddle-point system. Building on the classical AMG solver, commonly used for elasticity, we augment it with a specialized subspace correction that filters near null space components introduced by contact interface constraints. Through theoretical analysis and numerical experiments on a range of linear and nonlinear contact problems, we demonstrate that the proposed solver achieves mesh independent convergence and maintains robustness against the ill-conditioning that notoriously plagues IP methods.
These results indicate that AMGF makes contact mechanics simulations more tractable and broadens the applicability of Newton-based IP methods in challenging engineering scenarios. More broadly, AMGF is well suited for problems, optimization or otherwise, where solver performance is limited by a low-dimensional subspace, such as those arising from localized constraints, interface conditions or model heterogeneities. This makes the method widely applicable beyond contact mechanics and constrained optimization.
\end{abstract}



\input{intro}

\input{prelim}
\input{ipsolver}
\input{linsystem}
\input{amgf}
\input{results}
\input{cost}
\input{conclusions}
\section*{Acknowledgments} We thank Steven Wopschall for all his help with the Tribol Contact Interface Physics Library. 
\section*{Disclaimer} This manuscript has been authored by Lawrence Livermore National Security, LLC under Contract No. DE-AC52-07NA27344 with the US. Department of Energy. The United States Government retains, and the publisher, by accepting the article for publication, acknowledges that the United States Government retains a nonexclusive, paid-up, irrevocable, world-wide license to publish or reproduce the published form of this manuscript, or allow others to do so, for United States Government purposes.

\bibliographystyle{siamplain}
\bibliography{ref}
\end{document}

%% file: preamble.tex
\usepackage{amsfonts}
\usepackage{amssymb}
\usepackage{graphicx}
\usepackage{epstopdf}
\usepackage[algo2e, ruled, vlined, noend, linesnumbered]{algorithm2e}

\usepackage{arydshln}  
\usepackage{bm}
\usepackage{enumitem}
\usepackage[numbers]{natbib}
\ifpdf
  \DeclareGraphicsExtensions{.eps,.pdf,.png,.jpg}
\else
  \DeclareGraphicsExtensions{.eps}
\fi

\newsiamremark{remark}{Remark}
\newsiamremark{hypothesis}{Hypothesis}
\crefname{hypothesis}{Hypothesis}{Hypotheses}
\newsiamthm{claim}{Claim}
\newsiamthm{assumption}{Assumption}
\crefname{assumption}{Assumption}{Assumptions}

\usepackage{amsopn}
\DeclareMathOperator{\diag}{diag}

\usepackage{cancel}
\usepackage[colorinlistoftodos,textsize=tiny]{todonotes}

\makeatletter%
\@mparswitchfalse%
\makeatother%
\normalmarginpar

\newcounter{todoListItems}

\newcommand{\sfb}{\mathsf{b}}
\newcommand{\sfc}{\mathsf{c}}

\newcommand{\sff}{\mathsf{f}}
\newcommand{\sfg}{\mathsf{g}}

\newcommand{\sfr}{\mathsf{r}}
\newcommand{\sfs}{\mathsf{s}}

\newcommand{\sfu}{\mathsf{u}}
\newcommand{\sfv}{\mathsf{v}}
\newcommand{\sfw}{\mathsf{w}}

\newcommand{\sfz}{\mathsf{z}}
\newcommand{\sfA}{\mathsf{A}}
\newcommand{\sfB}{\mathsf{B}}

\newcommand{\sfD}{\mathsf{D}}
\newcommand{\sfE}{\mathsf{E}}

\newcommand{\sfI}{\mathsf{I}}
\newcommand{\sfJ}{\mathsf{J}}
\newcommand{\sfK}{\mathsf{K}}

\newcommand{\sfM}{\mathsf{M}}

\newcommand{\sfP}{\mathsf{P}}

\newcommand{\R}{\mathbb{R}} 

\renewcommand{\figurename}{Figure}

%% file: intro.tex
\section{Introduction}
Contact between multiple bodies occurs frequently in scientific and engineering applications and can significantly influence the performance of many critical engineering systems such as part assemblies, metal forming, and crash-worthiness. Despite its importance, contact introduces complex constraints that pose major challenges for large-scale simulations and optimization workflows \cite{wang2024designoptimizationunilateralcontact}. In this paper, we introduce a novel preconditioner, algebraic multigrid (AMG) with filtering (AMGF), specifically designed to address the ill-conditioning associated with contact interface constraints. Integrated within a Newton-based interior-point (IP) framework, AMGF explicitly filters undesirable modes that otherwise hinder solver performance. We establish theoretical bounds for the preconditioner and demonstrate that the IP-based approach equipped with AMGF achieves robust, scalable performance and efficiently handles challenging large-scale contact mechanics problems.

Numerical simulations of systems that include contact are significantly more challenging than those that neglect it. The additional numerical challenges stem from the associated nonlinearities due to the unilateral, nonsmooth nature of gaps opening and closing, sliding, friction, rigid body motions, and the linear system solution for implicitly time integrated approaches. Numerical solutions of contact problems are often computed by using explicitly time integrated finite element methods (FEM) \cite{wriggers2006computational,laursen1992formulation}, particularly for transient problems such as crash-worthiness analyses whose time regimes are on the order of milliseconds. This avoids nonlinear Newton type schemes with associated Newton linear system solves required in an implicit FEM scheme. However, long time events such as thermostructural analysis of mechanical systems, geomechanical problems, metal forming, etc., require implicit schemes. In the absence of contact, nonlinear solid mechanics problems lead to linear systems whose condition number scales as $1/h^2$ where $h$ represents a measure of the typical element size. The introduction of contact can significantly worsen the system conditioning, leading most simulations to rely on direct solvers. However, the excessive computational cost and memory requirements of direct solvers in the large-scale regime limits the often needed numerical fidelity achieved through mesh refinement.
The most commonly used schemes to enforce the complementary conditions associated with contact in FEM-based simulations include penalty \cite{hallquist1985},
augmented Lagrangian \cite{laursen1992formulation}, and active set approaches \cite{huebwohl05}. Penalty methods often require large penalty parameters to get sufficiently accurate results, which makes solving the associated nonlinear systems challenging \cite{nocedalwright2006}. The augmented Lagrangian method, while more robust, tends to converge slowly. Active set approaches involving Lagrange multipliers are also difficult to solve in general due to the indefinite nature of the resulting Schur complement systems. Alternative approaches, such as IPC \cite{10.1145/3386569.3392425,10.1145/3657648}, handle contact through barrier potentials embedded directly in the energy, offering robustness and differentiability. These methods follow a different modeling paradigm than the IP formulation adopted here.
In this paper, we employ a Newton-based IP method as the foundation of our solver framework. The IP method has gained significant popularity in the optimization community for solving large-scale nonconvex optimization problems and, when appropriately designed \cite{petra2023, hartland2024scalableinteriorpointgaussnewtonmethod}, can provide a mesh independent means to solve the associated optimization problem. While IP methods have been explored for contact mechanics problems dating back to \cite{chrisklarb1998}, recent advances in IP techniques motivate their renewed application to challenging large-scale simulations. 

Robust iterative linear solvers are needed to replace the often used direct solvers in order to achieve a fully scalable computational method. Krylov-subspace methods help alleviate the computational burden of solving these linear systems but require a large number of iterations unless an effective preconditioner is available. AMG preconditioners are particularly effective for solving linear systems which arise from the discretization of elliptic partial differential equations, such as elasticity which is a principal component of computational contact mechanics linear systems. While AMG preconditioners have been used for active set approaches with indefinite matrices ~\cite{wiesner2021,frances2022}, their application to IP and related continuation methods in contact mechanics remains unexplored. It is worth noting that despite the robustness of IP methods, they introduce inherent sources of ill-conditioning into the associated IP-Newton linear systems, making preconditioner design more challenging.

The proposed preconditioner, AMGF, is designed to handle scenarios where standard solvers perform well except within certain small subspaces. AMGF can be viewed as an AMG solver augmented with an additional filtering step, formulated as a subspace correction. In the context of contact mechanics, the filtered subspace corresponds to degrees of freedom associated with the contact interface, which typically remains sufficiently small relative to the full problem size as the ratio scales as surface area to volume. Although our focus is on contact problems, the AMGF methodology is broadly applicable to any situation where solver performance deteriorates due to small problematic subspaces. The theoretical foundation of AMGF draws on classical subspace correction theory \cite{xu1992,xu2001,xu2002}. AMGF shares similarities with multigrid methods \cite{adams2004,wiesner2021,falgout2005} that combine smoothing with coarse-grid corrections. Here AMG serves as the smoother and the additional subspace correction is applied to address the challenging near-null space modes. Similarities can also be drawn to deflation techniques~\cite{deflation1998,deflation2000}, which accelerate convergence by separately treating problematic eigenmodes. A key contribution of this paper is a theoretical result establishing that if a preconditioner \(\sfB\) (e.g., AMG) of a symmetric positive definite (SPD) matrix \(\sfA\) yields a condition number \(\kappa_{\sfB\sfA}\) for the linear system excluding the problematic subspace, then the preconditioner \(\sfB\) with filtering (e.g., AMGF) ensures that the condition number of the system on the whole space is approximately twice as large as \(\kappa_{\sfB\sfA}\). Numerical experiments confirm our theoretical estimates and demonstrate that the IP-based approach, combined with the AMGF preconditioner, can efficiently handle large-scale contact mechanics problems.

The remainder of this paper is structured as follows: \Cref{section:problem_formulation} presents the FEM formulation of the contact problem. \Cref{sec:interior_point_method} outlines the IP method used to solve the optimization problem. \Cref{section:Model_problem} introduces a contact example, highlighting the need for an effective preconditioner. \Cref{section:precon} details the proposed AMGF preconditioner, accompanied by condition number estimates. \Cref{section:results} evaluates its performance through numerical experiments {\color{revised} and \cref{section:cost} discusses the computational cost. Finally, \cref{section:conclusion} provides a synopsis and outlines directions for future work.}

%% file: prelim.tex
\section{Contact problem formulation and constraints}
\label{section:problem_formulation}
In computational contact mechanics, interactions between multiple deformable bodies are modeled within the FEM framework. The main challenge lies in ensuring that these bodies do not penetrate each other while maintaining physical accuracy. This requires formulating appropriate contact constraints and incorporating them into the governing equations. In this section, we first describe the mathematical enforcement of the contact constraints, followed by the FEM problem formulation.
\subsection{Contact constraints and enforcement}
In contact mechanics, the contact gap $g$ is a measure of the distance between points on the surfaces $\partial \Omega_1$ and $\partial \Omega_2$, as shown in
~\cref{fig:contact_diagram}. The condition $g = 0$ defines the contact surface $\Gamma^c$, where interaction occurs between the bodies. The gap can be formulated using different approaches, among which the node-on-segment and mortar segment-to-segment methods are commonly used. The node-on-segment collocation approach (\cref{fig:gap1}) defines the gap as the closest point distance from nodes on one body to segments on the other \cite{zavarise2009}. While being easier to implement, this method is prone to inaccuracies in coarse meshes or large deformations. Alternatively, the mortar segment-to-segment approach \cite{belgaetal98,pusolaur04a} (\cref{fig:gap2}), has become popular due to its robustness and accuracy. The mortar method enforces the contact condition weakly by integrating the gap constraint over the contact surface. The discretized gap constraint is expressed as
\begin{equation}
    g_A = \int_{\Gamma^c_A} \varphi^1_A \, \left(\bm{x}^1_h - \bm{x}^2_h\right) \cdot \bm{n}_A \, d\Gamma, 
    \quad g(\bm{x}) = \left\{g_1,g_2\,\ldots,g_n\right\}^{\top} \ge 0
    \label{eq:mortargap}
\end{equation}
Here, $\bm{x}^1_h$ and $\bm{x}^2_h$ denote the discrete representation of the deformed coordinate positions along the contacting surfaces, $\bm{n}_A$ is the segment averaged normal at node $A$ used to compute the contact distances and $\varphi^1_A$ is the finite element shape function associated with the node $A$ on side $1$ providing a weighted value of the gap. The gap constraint $g_A$ is enforced everywhere along $\partial \Omega_1 \cup \partial \Omega_2$, leading to the set of constraint equations in ~\cref{eq:mortargap}.
It should be noted that the constraints from \cref{eq:mortargap} are nonlinear. Here, a small incremental slip version is applied by fixing the contact area $\Gamma^c_A$ and normal $\bm{n}_A$ at the beginning of each time step to recover linearity. This doesn't preclude large sliding but does require more time steps, the more general approach is beyond the scope of this work.
\begin{figure}[H]
    \centering
    \includegraphics[width=0.5\linewidth]{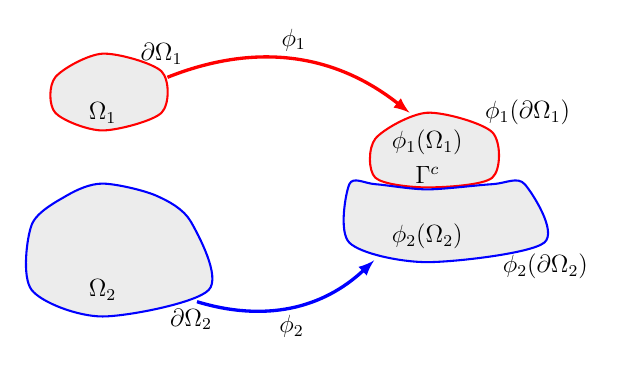}
    \caption{Two-body contact problem. Deformation mappings $\phi_1$ and $\phi_2$
     take undeformed states to contacting states with contact surface \(\Gamma^c := \phi_1(\partial \Omega_1) \cap \phi_2(\partial\Omega_2)\).}
    \label{fig:contact_diagram}
\end{figure}
\begin{figure}[H]
\setlength{\abovecaptionskip}{-2pt} %
    \begin{subfigure}[t]{0.45\textwidth}
    \centering
    \includegraphics[width=0.8\linewidth]{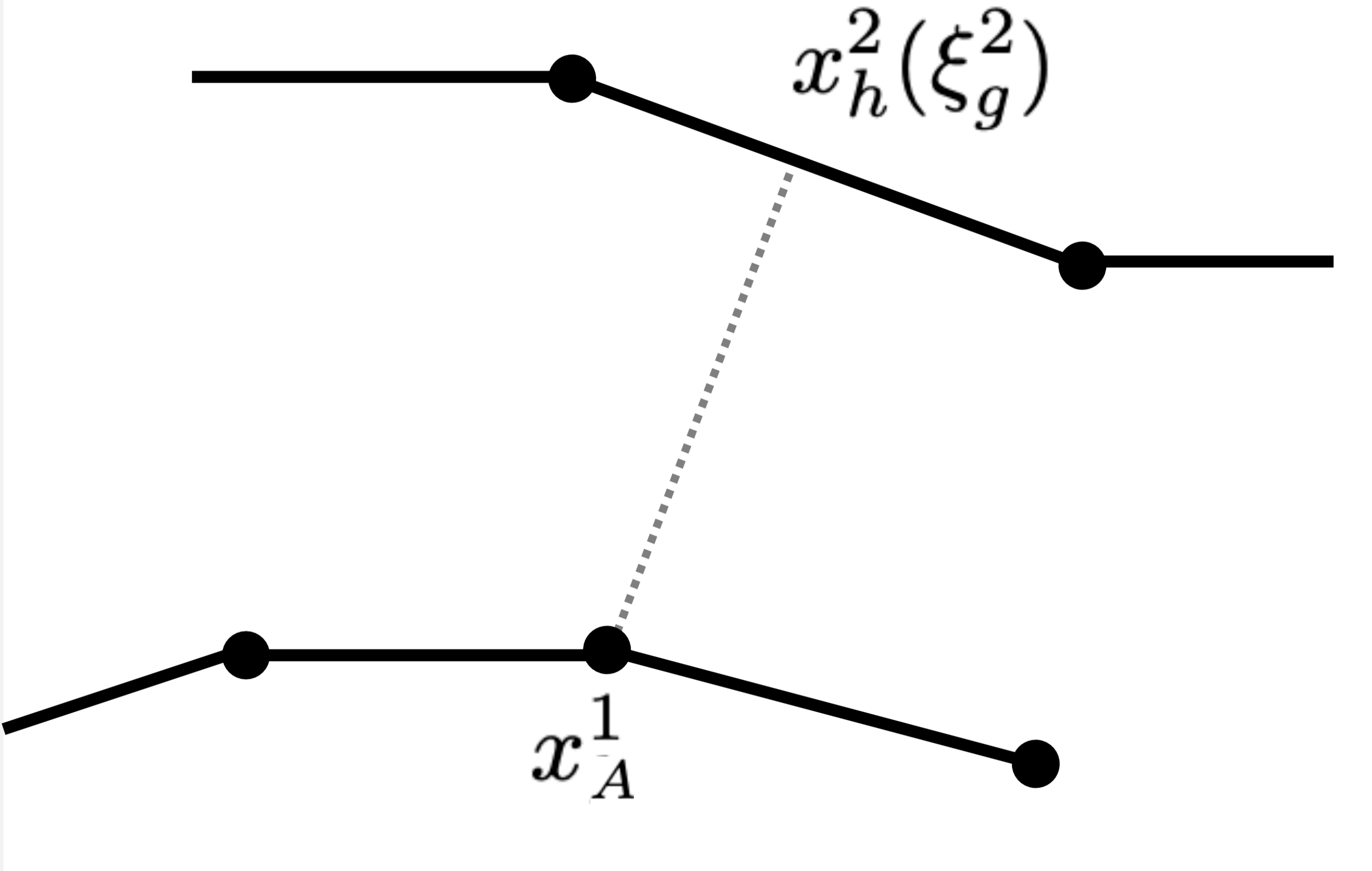}
    \caption{Node-on-surface gap calculation where side $1$ node $A$ is projected
    against side $2$ surface using closest-point.}
    \label{fig:gap1}
    \end{subfigure}\hfill
        \begin{subfigure}[t]{0.45\textwidth}
    \centering
    \includegraphics[width=0.8\linewidth]{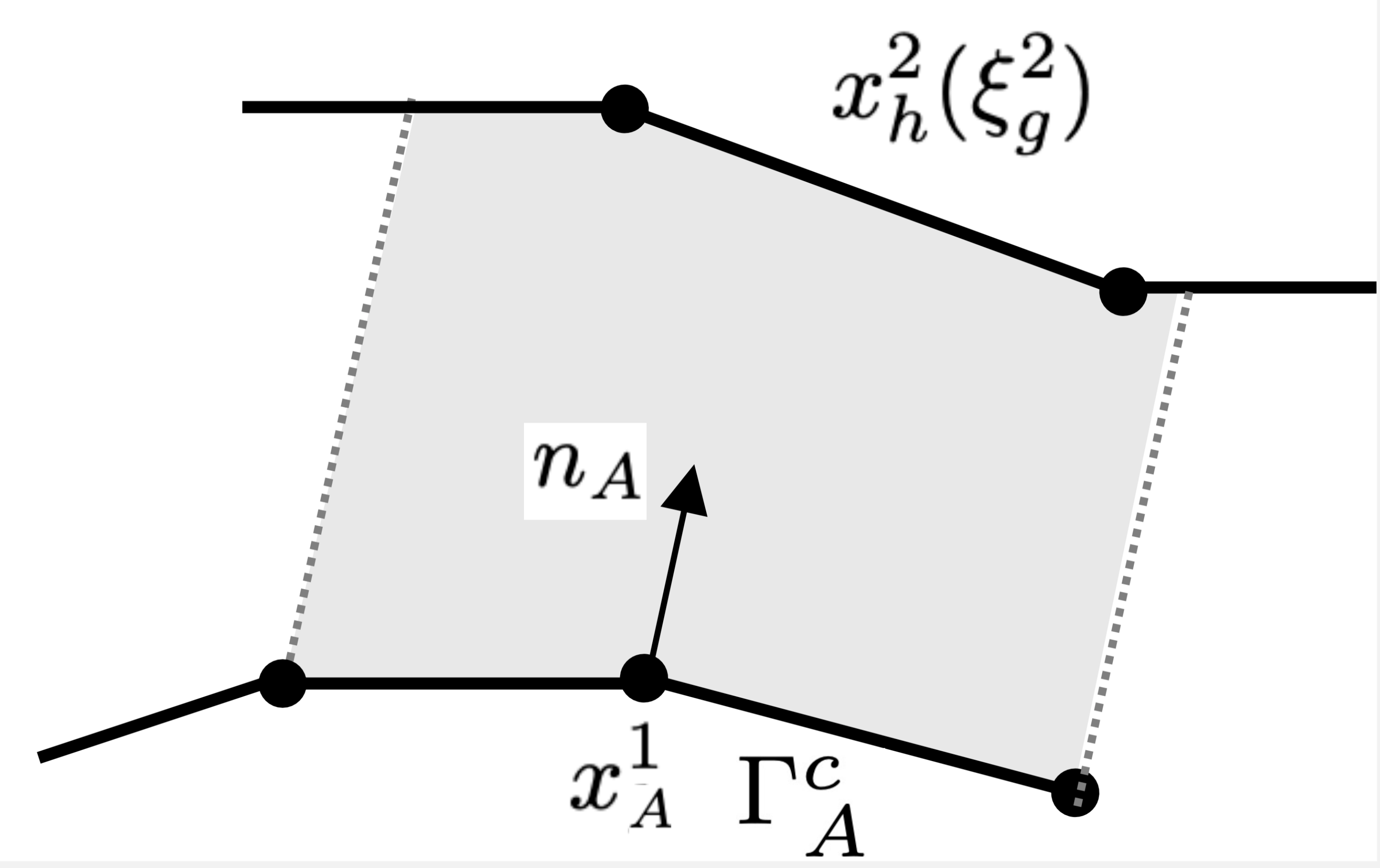}
    \caption{Mortar gap weighted volume approach.
    Here, $\Gamma^c_A$ is the compact domain of the shape function $\varphi^1_A$ and defines the limits of integration
    in \cref{eq:mortargap}.}
    \label{fig:gap2}
    \end{subfigure}
    \caption{Different gap definitions}
    \label{fig:gap_def}
\end{figure}
\subsection{Finite element formulation of the contact problem}
With the contact constraints defined, the problem is now formulated within the finite element framework. We consider the problem of \(N\) linear\footnote{Linear elasticity is used in presentation but extension to nonlinear elasticity poses no restriction.} elastic bodies in contact as illustrated in~\cref{fig:contact_diagram}. The spatial domain occupied by these bodies is defined as the union of their respective subdomains i.e., \(\Omega := \bigcup_{k=1}^N \Omega_k\). The governing equations are formulated in the function space \(\bm{U} := \prod_{k=1}^N \bm{H}^1\left(\Omega_k\right)\) where \(\bm{H}^1\left(\Omega_k\right)\) is the Sobolev space of vector-valued functions:
\begin{equation*}
\begin{aligned}
\bm{H}^1\left(\Omega_k\right) &= \left\{\bm{v}:=(v_1, v_2, v_3):\Omega_k \rightarrow \R^3 |\, v_i \in H^1\left(\Omega_k\right), i=1,2,3\right\}. \\
\end{aligned}
\end{equation*}
The objective is to solve the quasi-static equilibrium problem while enforcing a \emph{nonpenetration} condition. Specifically, we seek the optimal solution to the following constrained minimization problem:
\begin{equation*}
\begin{aligned}
\min_{\bm{u} \in \bm{U}}
\mathcal{E}(\bm{u}):=& \sum_{k=1}^{N}\int_{\Omega_{k}} \left(\frac{1}{2} \bm{\sigma}_{k}(\bm{u}_{k}) : \bm{\epsilon}(\bm{u}_{k}) - \bm{f}_{k}   \cdot \bm{u}_{k} \right) \mathrm{d}\Omega_{k} \\
\text{s.t. }\, g(\bm{u}) \ge&\, 0.
\end{aligned}
\end{equation*}
Here, \(\bm{u}\) represents the displacement field and \(\mathcal{E}(\bm{u})\) is the total potential energy functional. \(\bm{\sigma}_{k}(\bm{u}_{k}):=\lambda_{k} \left(\nabla \cdot \bm{u}_{k}\right) \bm{I} + 2 \mu_{k}\bm{\epsilon}(\bm{u}_{k})\) is the stress tensor of the \(k^{\text{th}}\) body and \(\lambda_{k}\) and \(\mu_{k}\) are the corresponding Lam\'{e} parameters. The small strain tensor is defined as \(\bm{\epsilon}(\bm{u}_{k}) = \frac{1}{2} \left(\nabla \bm{u}_{k} + \nabla \bm{u}_{k}^\top\right)\) and \(\bm{f}_{k} = \bm{f}_{k}(\bm{x},t)\) represents an external body force. The function \(g(\bm{u})\) characterizes the contact gap between the bodies.
Let \(\left\{\varphi_i\right\}_{i=1}^{n}\) denote the basis for the discrete subspace \(\bm{U}_h \subset \bm{U}\).
We then seek the discrete solution \(\bm{u}_h = \sum_{i=1}^{n} \sfu_i \varphi_i\) or equivalently the set of coefficients
\(\sfu = \left[\sfu_i\right]_{i=1}^{n}\) that satisfy the discrete minimization problem
\begin{equation}
\label{eq:linelastmin_discrete}
\begin{aligned}
\sfu = \arg \min_{\sfw\in \R^n} \,\, &\sfE(\sfw):=\left(\frac{1}{2} \sfw^\top \sfK \sfw - \sfw^\top \sff\right) \\
\text{s.t. } \,\, & \sfJ_{\text{\tiny ref}}\left(\sfw - \sfu_{\text{\tiny ref}}\right) + \sfg_{\text{\tiny ref}} \ge 0.
\end{aligned}
\end{equation}
Here, \(\sfK\) is an \(N \times N\) block diagonal matrix, with its blocks representing the linear elasticity matrix on each of the bodies in contact. The vector \(\sff\) represents contributions from external body and surface forces, as well as internal stresses present in the configuration about which the problem is linearized when approximating a nonlinear system through successive linear steps. We linearize the gap function around some reference displacement field represented on the discrete level by the vector \(\sfu_{\text{\tiny ref}}\). For nonlinear material models, we take $\sfE$ to be a second order Taylor series approximation of the elastic energy{\color{revised}, and $\sfK$ represents  the tangent stiffness matrix evaluated at the reference configuration \(\sfu_{\text{\tiny ref}}\)}. The vector \(\sfg_{\text{\tiny ref}}\) and the matrix \(\sfJ_{\text{\tiny ref}}\) represent the discrete approximations of the gap function and its Jacobian, respectively, evaluated at the reference displacement field. With the contact constraints incorporated into the FEM formulation, solving the resulting constrained system efficiently becomes a key challenge. To address this, we employ an IP method, which is well-suited for handling large-scale constrained optimization problems. In the next section, we provide an outline of the IP approach and its application to contact mechanics.

%% file: ipsolver.tex
\section{Interior-point method}
\label{sec:interior_point_method}

In this section, we provide background on the application of the IP method to inequality-constrained optimization problems, such as the targeted frictionless steady-state and quasi-static contact mechanics problems. In particular, we employ a filter line-search IP method that is one of the most robust methods for nonlinear nonconvex optimization. This method possesses best-in-class global and local convergence properties \cite{wachter2005, wachter_05_ipopt2}. In general, we consider
\begin{equation}
\label{eq:optproblem1}
\begin{aligned}
        \min_{\sf{u}\in\mathbb{R}^{n}}&\,{\sfE}(\sf{u}),  \\
        \text{s.t. }& \sf{g}(\sf{u})\geq \sf{0},
\end{aligned}
\end{equation}
defined by the objective \(\sfE:\R^n\rightarrow\R\) and constraint \(\sfg:\R^n\rightarrow\R^m\) functions\footnote{The following presentation follows the incremental slip mentioned in \Cref{section:problem_formulation}, thus coordinates \(x\) can be replaced by
displacements \(u\) as in \cref{eq:mortargap}}.
We transform the inequality-constrained problem into an equality-constrained one by introducing \emph{slack} variables \(\sfs\in \R^m\). Then \cref{eq:optproblem1} is reformulated as:
\begin{equation*}
\begin{aligned}
\min_{\sfu\in\R^n,\, \sfs\in\R^m} \,\, & \sfE(\sfu), \\
\text{s.t. }              & \sfc(\sfu,\sfs)=\sf0, \,\, \text{and } \, \sfs\geq \sf0,
\end{aligned}
\end{equation*}
where \(\sfc(\sfu,\sfs):=\sfg(\sfu)-\sfs\). Generally, an IP method involves solving a sequence of equality constrained optimization problems
\begin{equation}
\label{eq:logbaroptproblem}
\begin{aligned}
\min_{\sfu\in\mathbb{R}^{n},\,\sfs\in\mathbb{R}^{m}}\,\,\varphi(\sfu,\sfs)&:=\sfE(\sfu)-\mu\sum_{i=1}^{m}(\sfM_{\sfs})_{i,i}\log\left(\sfs_{i}\right), \\
\text{s.t. }\sfc(\sfu,\sfs)&\phantom{:}=\sf0.
\end{aligned}
\end{equation}
The \emph{log-barrier parameter} \(\mu>0\) controls the trade-off between feasibility and optimality, forcing the iterates to remain strictly within the feasible region. As \(\mu\rightarrow 0^{+}\), the solution of the log-barrier subproblem converges to the true optimal solution of the constrained problem described by \cref{eq:optproblem1}. The diagonal lumped mass-matrix $\sfM_{\sfs}$ associated to $\sfs$ is introduced as in \cite{petra2023, hartland2024scalableinteriorpointgaussnewtonmethod} so that optimization problems that arise from discretization are well defined upon mesh refinement.
The Lagrangian formalism is then followed wherein a Lagrangian $\mathcal{L}$, with respect to the log-barrier subproblem, is formed with Lagrange multiplier \(\lambda\in\R^m\)
\begin{equation*}
        \mathcal{L}(\sfu,\sfs,\sf\lambda):=
        \varphi(\sfu,\sfs)+\sf\lambda^{\top}\sfc(\sfu,\sfs),
\end{equation*}
in order to express the first-order optimality conditions \cite{nocedalwright2006}
\begin{equation*}
\begin{aligned}
\nabla_{\sfu}\mathcal{L}(\sfu,\sfs,\sf\lambda)=\nabla_{\sfu}\sfE(\sfu)+
\sf{J}^{\top}\sf\lambda&=\sf0
,\\
\nabla_{\sfs}\mathcal{L}(\sfu,\sfs,\sf\lambda)=-\mu\,\sfM_{\sfs}\sf1/\sfs-\sf\lambda&=\sf0,\\
\sfc(\sfu,\sfs)&=\sf0,
\end{aligned}
\end{equation*}
where \(\sf1/\sfs\) denotes the element-wise division of a vector of ones by the slack vector \(\sfs\), and \(\sfJ=\partial\sfc/\partial\sfu\). We then formally introduce the dual variable $\sfz:=\mu\sf1/\sfs,$
 to form the nonlinear system of equations
\begin{equation}
\label{eq:optConditionsPrimalDual}
\begin{aligned}
        \sfr_{\sfu}=\nabla_{\sfu}\sfE(\sfu)+
        \sfJ^{\top}\sf\lambda&=\sf0
        ,\\
        \sfr_{\sfs}=-\sf\lambda-\sfM_{\sfs}\sfz&=\sf0,\\
        \sfr_{\sf\lambda}=\sfc(\sfu,\sfs)&=\sf0,\\
        \sfr_{\sfz}(\mu)=\text{diag}(\sfz)\,\sfs-\mu\cdot\sf1&=\sf0.
\end{aligned}
\end{equation}
Here, $\sfr_{\sfu},\sfr_{\sfs},\sfr_{\sf\lambda}$, and $\sfr_{\sfz}(\mu)$ are nonlinear residuals. Solving the IP system requires an iterative update of the primal and dual variables. Linearizing the first-order optimality conditions \cref{eq:optConditionsPrimalDual} yields the following Newton system, which determines the search direction at each iteration.
\begin{equation}
\label{eq:4x4Newton}
\begin{bmatrix}
        \sfK & \sf0 & \sfJ^{\top} & \sf0 \\
        \sf0 & \sf0 & -\sfI & -\sfM_{\sfs} \\
        \sfJ & -\sfI & \sf0 & \sf0 \\
        \sf0 & \text{diag}(\sfz) & \sf0 & \text{diag}(\sfs)
\end{bmatrix}
\begin{bmatrix}
        \hat{\sfu} \\
        \hat{\sfs} \\
        \hat{\sf\lambda} \\
        \hat{\sfz}
\end{bmatrix}
=
\begin{bmatrix}
        -\sfr_{\sfu} \\
        -\sfr_{\sfs} \\
        -\sfr_{\sf\lambda} \\
        -\sfr_{\sfz}
\end{bmatrix},
\end{equation}
where \(\sfK:=\partial \sfr_{\sfu}/\partial\sfu\) and whose solution \((\hat{\sfu},\hat{\sfs},\hat{\sf\lambda},\hat{\sfz})\) is the so-called Newton search direction. The optimizer estimate is then updated by 
\begin{equation} \label{eq:linesearchupdate}
\begin{aligned}
\begin{bmatrix}
        \sfu^{+} & \sfs^{+} & \sf\lambda^{+}
\end{bmatrix}&=
\begin{bmatrix}
        \sfu & \sfs & \sf\lambda
\end{bmatrix}
+\alpha_{p}
\begin{bmatrix}
        \hat{\sfu} & \hat{\sfs} & \hat{\sf\lambda}
\end{bmatrix},\\
\sfz^{+}&=\sfz+\alpha_d\,\hat{\sfz},
\end{aligned}
\end{equation}
for strictly positive primal and dual step-lengths $\alpha_{p}$ and $\alpha_{d}$ determined by the globalizing filter line-search algorithm \cite{wachter2006}. As is common, we symmetrize \cref{eq:4x4Newton} by algebraically eliminating $\hat{\sfz}$ thereby forming the linear system
\begin{equation*}
\begin{bmatrix}
        \sfK & \sf0 & \sfJ^{\top} \\
        \sf0 & \sfD & -\sfI \\
        \sfJ & -\sfI & \sf0
\end{bmatrix}
\begin{bmatrix}
        \hat{\sfu}\\
        \hat{\sfs} \\
        \hat{\sf\lambda}
\end{bmatrix}
=
\begin{bmatrix}
        \sfb_{\sfu} \\
        \sfb_{\sfs} \\
        \sfb_{\sf\lambda}
\end{bmatrix}:=
\begin{bmatrix}
        -\sfr_{\sfu} \\
        -\left(\sfr_{\sfs}+\sfM_{\sfs}\diag(\sfs)^{-1}\sfr_{\sfz}\right)\\
        -\sfr_{\sf\lambda}
\end{bmatrix},
\end{equation*}
where \(\sfD:=\diag\left(\sfM_{\sfs}\sf z / \sf s\right)\).
We can further reduce the linear system to an equation in \(\hat{\sfu}\) by static condensation of \(\hat{\sfs}\) and \(\hat{\sf\lambda}\):
\begin{equation}
\label{eq:linsystem1}
\sfA \hat{\sfu}=\sfb,
\end{equation}
where \(\sfA:=\sfK+\sfJ^{\top}\sfD\sfJ\), and \(\sfb:=\sfb_{\sfu}+\sfJ^{\top}\left(\sfD\sfb_{\sf\lambda}+\sfb_{\sfs}\right)\).

\begin{algorithm2e}
        \SetAlgoNoLine
        \SetKwInOut{Input}{Input}
        \SetKwInOut{Output}{Output}
        \SetKwRepeat{Do}{do}{while}
        {
                \Input{Optimizer tolerance $\text{tol}_{\text{\tiny IP}}>0$,  initial point $\sf{u}$}
                \Output{Approximate optimizer $\sf{u}^{\star}$}

                Initialize barrier parameter $\mu>0$ and other algorithm parameters\;
                \Repeat{$\operatorname{e}_{\operatorname{opt}}<\operatorname{tol}_{\operatorname{IP}}$}
                {
            \Repeat{$\operatorname{e}_{\operatorname{opt}}^{(\mu)}<\operatorname{tol}_{\operatorname{IP}}^{(\mu)}$}
            {
            $\sf{A}\sf{\hat{u}}=\sf{b}$ 

            Recover $\hat{\sf{s}},\hat{\sf{\lambda}}, \hat{\sf{z}}$ from $\hat{\sf{u}}$ 

            Update $\sf{u}$, $\sf{s}$, $\sf{\lambda}$, and $\sf{z}$ via a line search 
        }

            Decrease $\mu$;
        }
        }
        \caption{IP-Newton optimization algorithm outline}
        \label{alg:outerloop}
\end{algorithm2e}
The main components of the optimization are summarized in Algorithm \ref{alg:outerloop}. The optimality errors $\operatorname{e}_{\operatorname{opt}}$ and $\operatorname{e}_{\operatorname{opt}}^{(\mu)}$ for the optimization problem \cref{eq:optproblem1} and log-barrier subproblem \cref{eq:logbaroptproblem} in Algorithm \ref{alg:outerloop} are defined as
\begin{align*}
\text{e}_{\text{\tiny opt}} &:= \max\left\{
\|\sf{r}_{\sf{u}}\|_{\sfM_{\sfu}^{-1}}, \|\sf{r}_{\sf{s}}\|_{\sfM_{\sfs}^{-1}} ,\|\sf{r}_{\sf{\lambda}}\|_{\infty},\|\sf{r}_{\sf{z}}(0)\|_{\infty} \right\}, \\
\text{e}_{\text{\tiny opt}}^{(\mu)} &:= \max\left\{
 \|\sf{r}_{\sf{u}}\|_{\sfM_{\sfu}^{-1}}, \|\sf{r}_{\sf{s}}\|_{\sfM_{\sfs}^{-1}},\|\sf{r}_{\sf{\lambda}}\|_{\infty},\|\sf{r}_{\sf{z}}(\mu)\|_{\infty} \right\}.
\end{align*}
 The tolerance of the log-barrier subproblem is $\text{tol}_{\text{\tiny IP}}^{(\mu)}=\kappa_{\epsilon}\mu$ and ensures that the subproblems are solved more accurately as the solution of \cref{eq:optproblem1} is approached. The values of the thus far unspecified constants are $\kappa_{\epsilon}=10$ and $s_{\text{max}}=100$. We utilize an inertia-free inertia regularization scheme for nonconvex problems as in \cite{chiang2016} in order that the search direction is of descent when the constraints are nearly satisfied.
\begin{remark}
\label{remark:boundconstr}
In certain cases, such as potentially nonconvex trust-region quadratic programming (QP) subproblems \cite{conn2000}, additional bound constraints are required to ensure that problem \cref{eq:optproblem1} is well defined. For instance, for problems with nonlinear elasticity models, such as the last example in \Cref{section:results}, additional bounds are placed on the optimization variable \(\sfu\), i.e.,
\(\underline{\delta}\leq \sf{u}-\sfu_{\star}\leq \overline{\delta}\), with 
\(\underline{\delta}\) and \(\overline{\delta} \) the upper and lower bounds from a reference state \(\sfu_{\star}\), respectively. For their enforcement, additional slack variables (\(\underline{\sfs}, \overline{\sfs}\)) are introduced
and therefore the constraint function in \cref{eq:logbaroptproblem} becomes:
\[\sfc(\sfu,\sfs)=\left\{ \sfg(\sfu)-\sfs_{\text{\tiny{gap}}},(\sfu-\sfu_{\star})-\underline{\delta}-\underline{\sfs},\overline{\delta}-(\sfu-\sfu_{\star})-\overline{\sfs}\right\}^{\top}.\]
Furthermore, in this context the reduced IP-Newton system matrix transforms to
\[\tilde{\sfA}:=\sfK+\tilde{\sfJ}^{\top}\tilde{\sfD}\tilde{\sfJ},
\text{ with }
\tilde{\sfJ}:=
\begin{bmatrix}
        \sfJ \\
        \sfI \\
        -\sfI
\end{bmatrix}
\text{ and }
\tilde{\sfD}:=
\begin{bmatrix}
        \sfD &  & \\
        & \underline{\sfD} & \\
        & & \overline{\sfD}
\end{bmatrix}.
\]
\end{remark}

The \emph{log-barrier Hessian}, \(\sfD\), is a positive definite diagonal matrix that becomes increasingly ill-conditioned as the IP method converges toward a local minimizer. This deterioration in conditioning makes solving the IP-Newton linear system more challenging and complicates the design of effective preconditioners for Krylov-subspace solvers. While AMG is effective for elasticity problems, it struggles with contact constraints in the IP setting, as shown in \Cref{section:Model_problem}. This motivates the specialized AMGF preconditioner introduced and analyzed in \Cref{section:precon}.

%% file: linsystem.tex
\section{Linear elasticity model problem} 
\label{section:Model_problem}

We consider a simple model problem consisting of two linear elastic bodies in contact: a small cubic block resting on a larger rectangular block (see \cref{fig:testNo4_diagram}). The larger block is subjected to a homogeneous Dirichlet boundary condition (BC) on its bottom face, while the smaller block has a nonhomogeneous Dirichlet BC at its top boundary surface \(\Gamma_1\) given by \(u_{\Gamma_1} := (0,0,-\frac{5}{7})\). In this setup, the top face of the large block acts as the mortar surface, the reference side for enforcing contact, while the bottom face of the smaller block is the nonmortar surface, whose motion is constrained relative to the mortar. All other boundary surfaces are traction-free and no additional external body forces are applied. The Lam\'{e} parameters for the two bodies correspond to the Young moduli  \(E_{\text{small-block}} = 1000 , E_{\text{large-block}} = 1\) and the Poisson's ratios \(\nu_{\text{small-block}} = 0, \nu_{\text{large-block}}=0.499\). To solve the problem, we use hexahedral meshes and discretize using first order Lagrange elements.
\begin{figure}[H]
    \begin{subfigure}[b]{0.53\textwidth}
    \centering
    \includegraphics[width=1.00\linewidth]{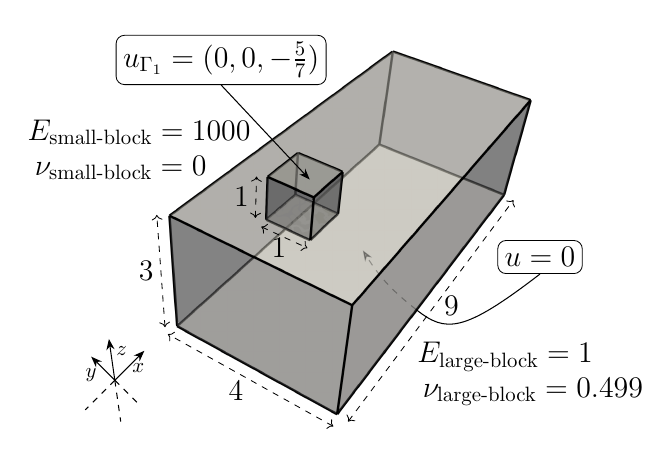}
    \caption{Initial configuration}
    \label{fig:testNo4_diagram}
    \end{subfigure}
    \begin{subfigure}[b]{0.46\textwidth}
        \centering
        \includegraphics[width=0.85\linewidth]{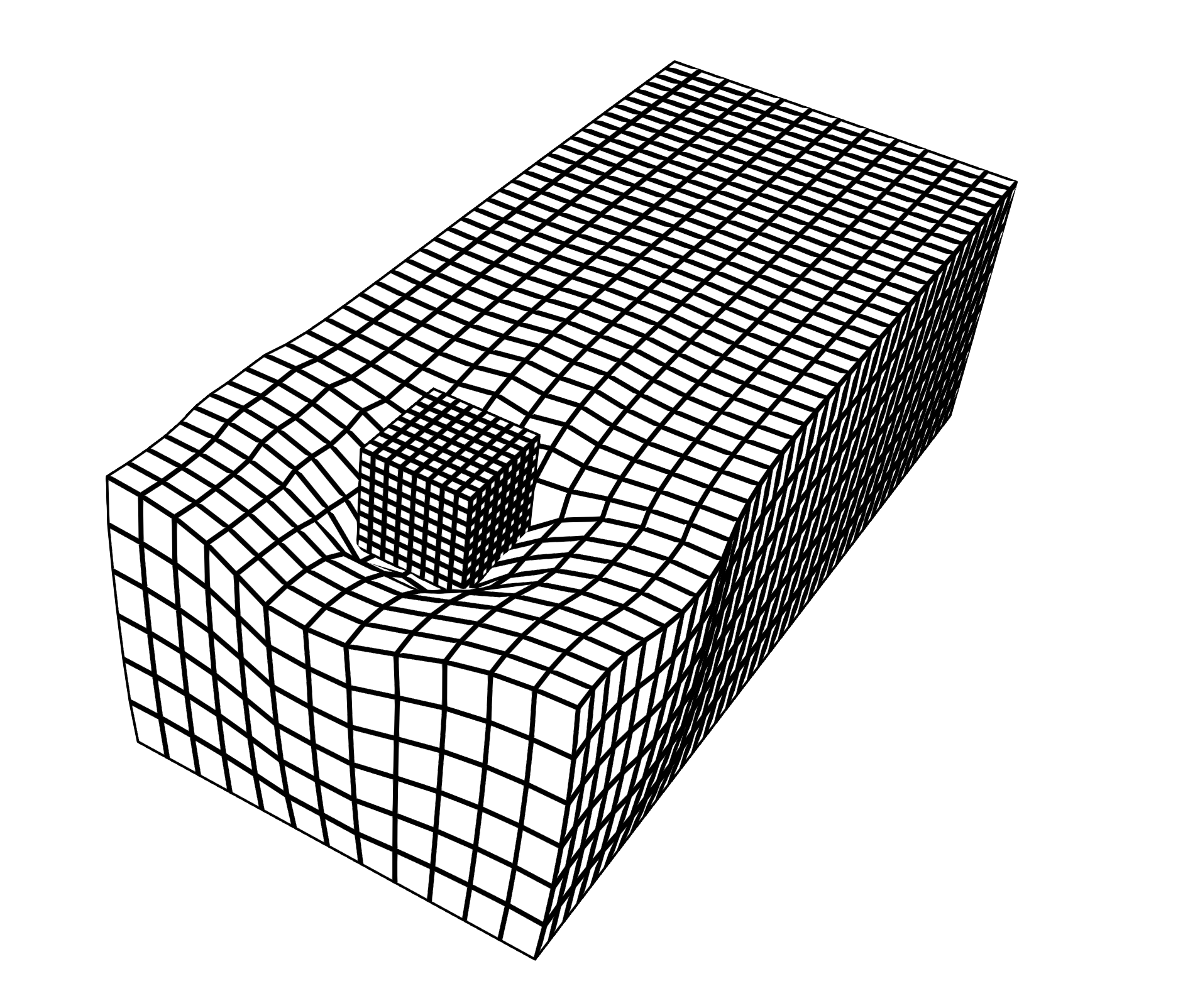}
    \caption{Final configuration for mesh level 1}
    \label{fig:testNo4_diagram-final}
    \end{subfigure}
    \caption{Two-block problem: a small cubic block is in contact with a larger rectangular block. The simulation is driven by a nonhomogeneous Dirichlet BC \(u = (0,0,-\frac{5}{7})\) enforced on the top face of the cubic block. The bottom face of the rectangular block is fixed at \(u=0\) and the rest of the boundary is traction free.}
\end{figure}
Additionally, we introduce a pseudo-time-stepping approach to improve accuracy of the gap function and its Jacobian. The nonhomogeneous Dirichlet BC is applied incrementally over multiple time steps. More specifically, given a time step {\color{revised} size} \(\delta t \) and final time \(T\), the BC at each time step \(t_i\) is defined as:
\begin{equation*}
u^{t_i}_{\Gamma_1} = \frac{i \delta t}{T} u_{\Gamma_1}, \quad i=1,2, \dots, m, \text{ with } \delta t = \frac{T}{m}.
\end{equation*}
At the time step \(t_i\) the minimization problem defined by \cref{eq:linelastmin_discrete} with BC value \(u^{t_i}_{\Gamma_1}\) is solved using the IP solver. The displacement solution is then used to generate an intermediate set of mesh nodes, which are utilized to compute an updated gap and its Jacobian. A new problem is subsequently formulated on the original mesh configuration, incorporating the next BC value and the updated gap and Jacobian.

\overfullrule=0pt
\begin{figure}[H]
    \centering
        \begin{subfigure}{0.4\textwidth}
        \centering
        \includegraphics[width=1.0\linewidth,trim=0 0 75 0, clip]{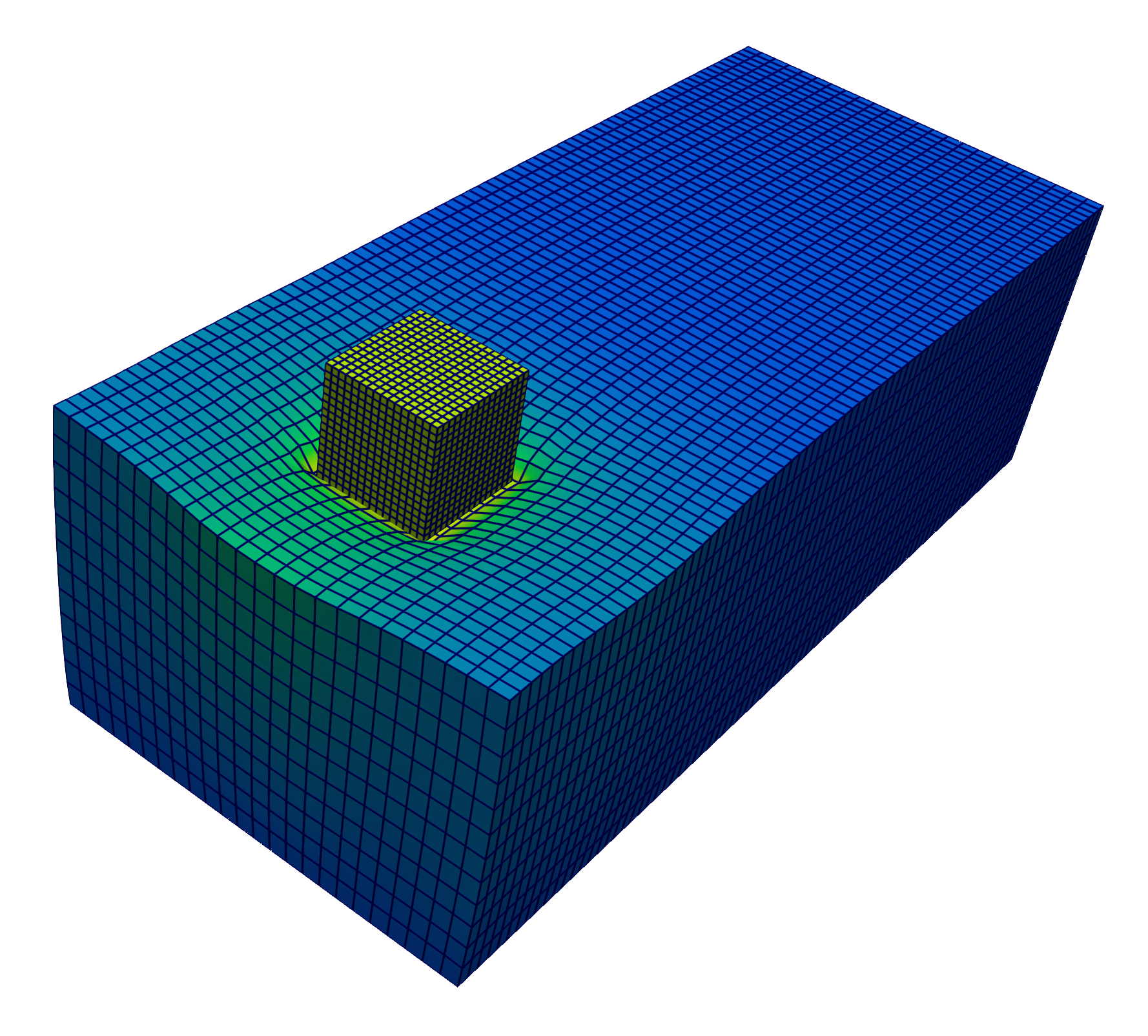}
    \end{subfigure}
        \begin{subfigure}{0.4\textwidth}
        \centering
        \includegraphics[width=1.0\linewidth, trim=0 0 75 0, clip]{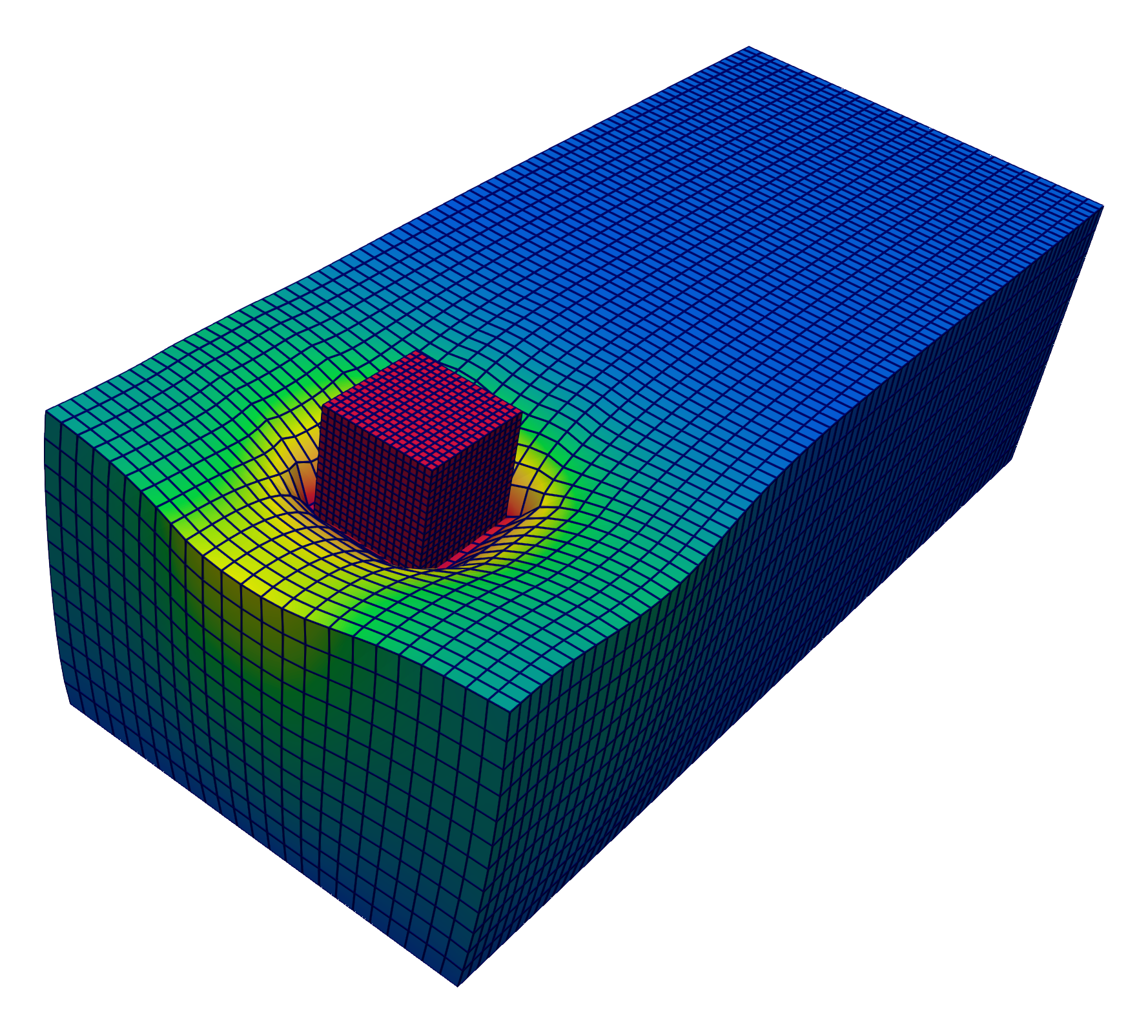}
    \end{subfigure}
    \begin{subfigure}{0.18\textwidth}
        \centering
        \includegraphics[width=1.4\linewidth,trim=150 0 0 0, clip]{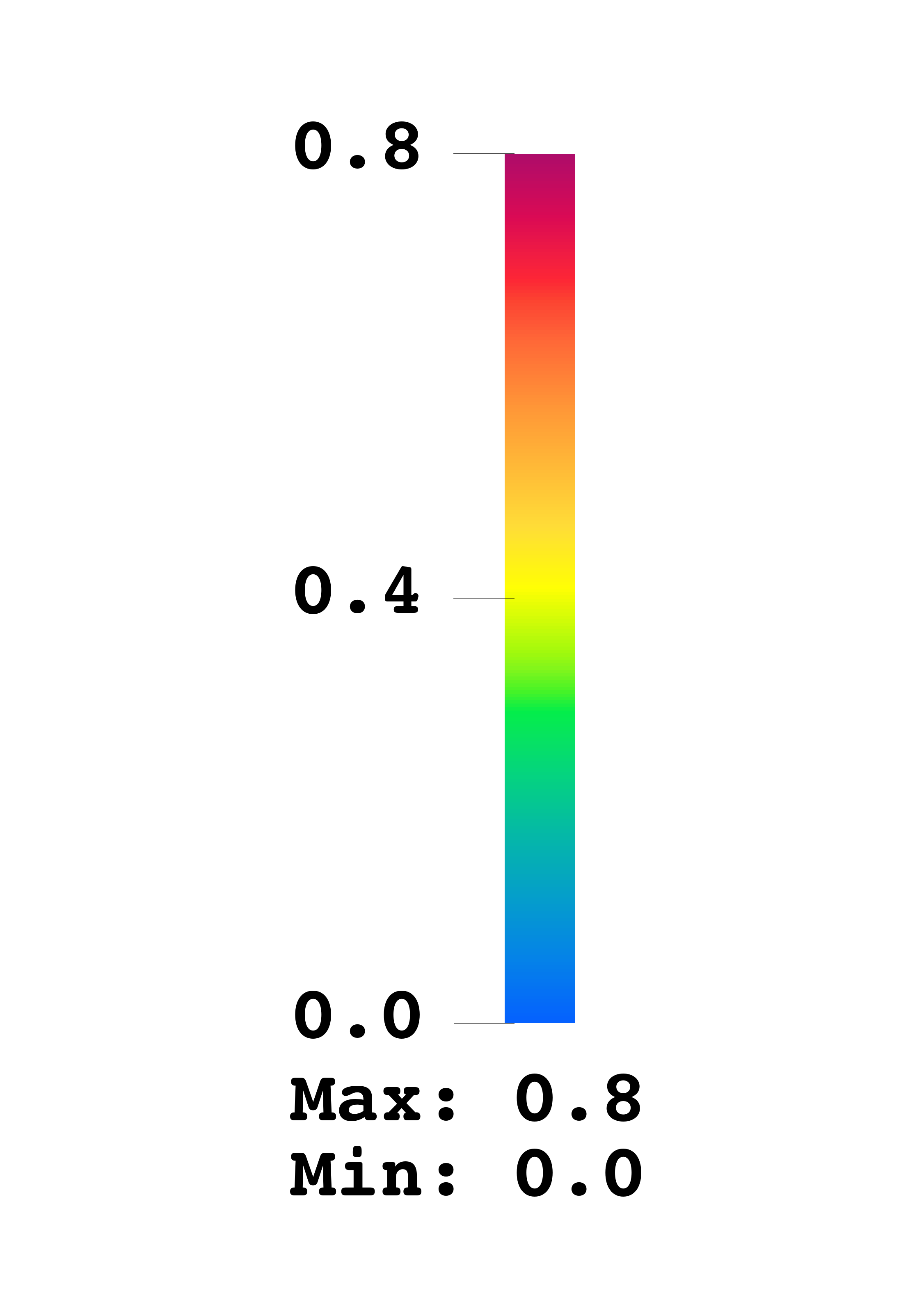}
    \end{subfigure}\hfill
    \caption{Deformed configurations and displacement magnitude at times steps \(t_1,t_2\).}
    \label{fig:ex4_deformation}
\end{figure}
\begin{figure}[H]
    \centering
    \begin{subfigure}[t]{0.49\textwidth}
        \centering
        \includegraphics[width=0.99\linewidth]{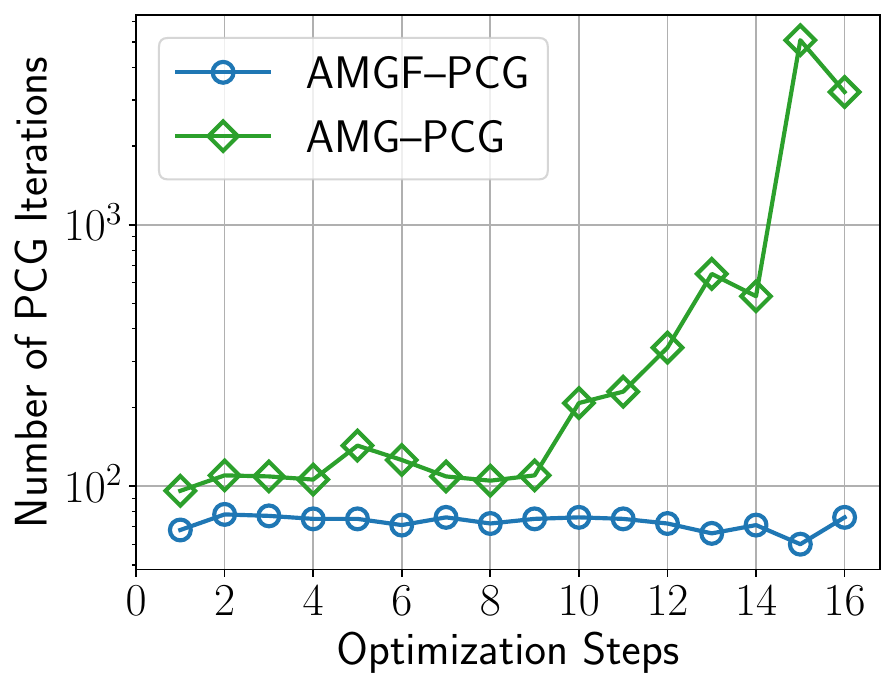}
    \caption{AMGF vs AMG convergence}
    \label{fig:amgvsamgf_subfig}
    \end{subfigure}\hfill
    \begin{subfigure}[t]{0.49\textwidth}
        \centering
        \includegraphics[width=0.99\linewidth]{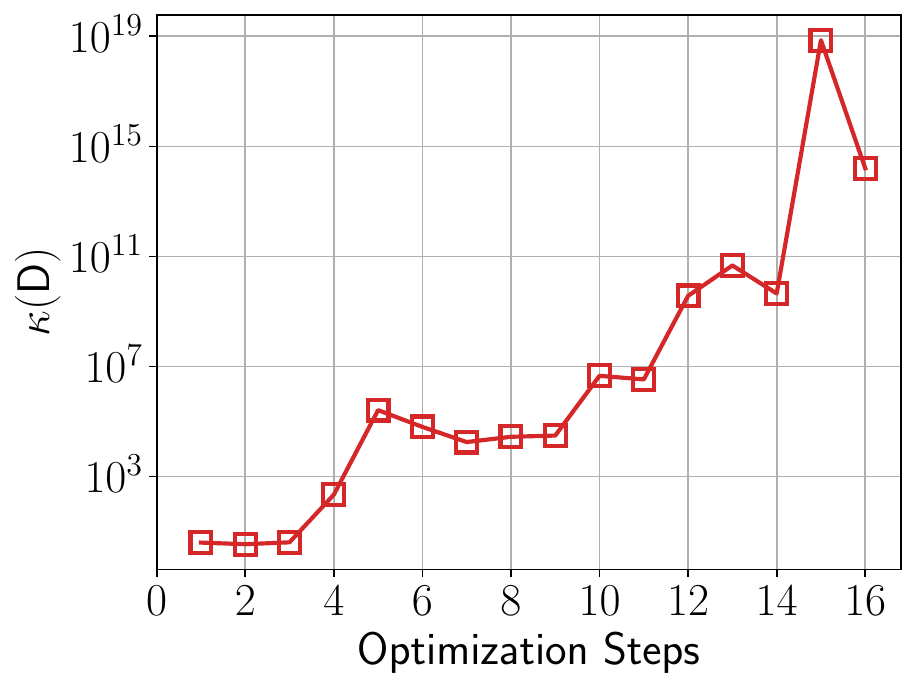}
        \caption{\(\sfD_{\max}/\sfD_{\min}\)}
        \label{fig:cond_d}
    \end{subfigure}\hfill
    \caption{Performance of AMG--PCG vs AMGF--PCG (left) within the IP optimization process for the first time step. The horizontal axis shows the number of IP solver iterations. The figure on the right shows the condition number of the diagonal matrix $\sfD$.}
    \label{fig:amgvsamgf}
\end{figure}
The problem is solved using two time steps. At each time step, the computed displacement solution is applied to the initial configuration, resulting in the deformed configurations shown in \cref{fig:ex4_deformation}. The final configuration is also shown in \cref{fig:testNo4_diagram-final}. The IP inner linear system \cref{eq:linsystem1} is solved using the PCG solver preconditioned with one AMG V-cycle (AMG--PCG). The PCG convergence criterion is met when the relative preconditioned residual norm becomes less than \(\text{tol}_\text{\tiny PCG} = 10^{-10}\). The IP solver exit tolerance is \(\text{tol}_\text{\tiny IP} =10^{-6}\). \cref{fig:amgvsamgf} illustrates the convergence of the AMG--PCG solver for the first time step. As the optimization progresses, we observe a significant increase in the number of iterations, particularly in later steps indicating a degradation in AMG effectiveness. \cref{fig:cond_d} presents the condition number of the diagonal matrix \(\sfD\). It is evident that the condition number of \(\sfD\) grows exponentially as the optimization steps increase. Note that for a fixed time step, both the stiffness matrix \(\sfK\) and the Jacobian matrix \(\sfJ\) remain unchanged throughout the optimization; therefore, the only component of \(\sfA\) that varies is \(\sfD\). Thus, the increasing condition number of \(\sfD\) directly correlates with the worsening performance of AMG preconditioning.

Matrix \(\sfD\) has dimensions \(\tt m \times m\), where \(\tt m\) is the number of contact constraints. The Jacobian matrix \(\sfJ\) has dimensions \(\tt m \times n\), where \(\tt n\) is the total number of elasticity degrees of freedom. However, \(\sfJ\) is highly sparse with at most \(\tt n_c\) nonempty columns, where \(\tt n_c\) corresponds to the number of degrees of freedom involved in contact. Given that \(\sfA = \sfK + \sfJ^{\top} \sfD \sfJ\), and \(\sfJ\) only has \(\tt n_c\) nonzero columns, the term \(\sfJ^{\top} \sfD \sfJ\) affects at most an \(\tt n_c \times n_c\) submatrix of \(\sfA\). The proposed method introduces a subspace correction step that specifically targets the \(\tt n_c\) degrees of freedom that are potentially in contact. Since \(\tt n_c \ll n\), the resulting subspace is small and spatially localized around the contact surface making the subspace correction step inexpensive. In particular, for a 3D problem under uniform mesh refinement, \(\tt n_c\) increases by approximately a factor of four, whereas the total number of degrees of freedom \(\tt n\) increases by a factor of eight.

The construction of the AMGF preconditioner along with theoretical condition number estimates are presented in the next section. As a prelude, in \cref{fig:amgvsamgf_subfig} we also present the PCG iteration count when preconditioned with AMGF (AMGF--PCG). Notably, the iteration count remains nearly constant throughout the optimization process, demonstrating robust convergence with respect to contact enforcement.

%% file: amgf.tex
\section{AMG with filtering preconditioner}
\label{section:precon}
In this section, we introduce and analyze the AMGF preconditioner, a general-purpose preconditioning framework that is particularly effective in scenarios where the performance of a standard solver degrades due to a problematic small subspace, yet remains effective on its orthogonal complement.
While our primary focus is on contact mechanics, the underlying theory builds on classical subspace correction techniques \cite{xu2001,xu1992,xu2002} and extends naturally to a broad class of problems exhibiting similar subspace-induced difficulties. In what follows, we prove that if a preconditioner \(\sfB\) (e.g., AMG) for an SPD matrix \(\sfA\) achieves a condition number \(\kappa_{\sfB\sfA}\) on the orthogonal complement of a small problematic subspace (\cref{assumption2}) and as a solver is convergent over the whole space (\cref{assumption1}),  then applying \(\sfB\) with filtering (e.g., AMGF) ensures that the condition number over the entire space is approximately twice as large as \(\kappa_{\sfB\sfA}\) (\cref{lemma:main,remark:upper_bound_estimate}).
\subsection{AMGF for contact problems}
Consider the two-body contact problem illustrated in \cref{fig:contact_diagram} and Let \(\mathcal{T}^h_1, \mathcal{T}^h_2\) denote the meshes of \(\Omega_1, \Omega_2\), respectively. The finite element (FE) space is given by
\[\bm{U}_h:=\bm{H}^1(\mathcal{T}^h_1) \times \bm{H}^1(\mathcal{T}^h_2).\]
Let \(\left\{\varphi_i\right\}_{i=1}^{\tt n}\) denote a basis for \(\bm{U}_h\).
We define the coefficient space, i.e., the space of degrees of freedom (DOFs), by
\[\mathbb{U} := \left\{ \{\sfu_i\}_{i=1}^{\tt n} \,|\, u_h = \sum_{i=1}^{\tt n} \sfu_i \varphi_i, u_h \in \bm{U}_h \right\} = \R^{\tt n}\]
Define the index set \(\mathcal{I}_c:=\left\{i \, | \, \text{supp}(\varphi_i)\cap \Gamma_c \neq \emptyset \right\}\) with \(|\mathcal{I}_c|={\tt n_c \le n}\). Then we can define the reduced space of functions adjacent to the contact surface by
\[\bm{W}_h := \text{span}\left\{\varphi_i \, | \, i \in \mathcal{I}_c \right\} \] and the corresponding coefficient space by
\[\mathbb{W} := \left\{\{\sfu_i\}_{i\in\mathcal{I}_c} \, | \, (\sfu_1,\dots, \sfu_{\tt n}) \in\mathbb{U}\right\} = \R^{\tt n_c}\]

\begin{definition}[discrete operators]
We define the following discrete operators:
\begin{itemize}
\item \(\sfP \in \R^{\tt n\times n_c}\) is the matrix corresponding to the natural embedding \(\bm{W}_h \hookrightarrow\bm{U}_h\).
\item \(\sfA \in \R^{\tt n\times n}\) and \(\sfA_{\sfw} \in \R^{\tt n_c \times n_c}\) are the matrices corresponding to the FE operators defined on \(\bm{U}_h\)
and \(\bm{W}_h\) respectively. Note that \(\sfA_{\sfw}:=\sfP^{\top} \sfA \sfP\)
%
%
%
\end{itemize}
\end{definition}
\begin{definition}[A-orthogonal complement]
 We define \(\mathbb{V} \subset \mathbb{U}\) to be the \(\sfA\)-orthogonal complement of the range\((\sfP)\), i.e,
 \[\mathbb{V}:=\left\{\sfu \in \mathbb{U}: (\sfu,\sfP\sfw)_{\sfA} = 0, \,\,\, \forall \sfw \in \mathbb{W}\right\}\]

 \emph{Here, and throughout the remainder of this section, \((\cdot, \cdot)_{\sfA} := (\sfA\cdot, \cdot)\) denotes the \(\sfA\)-inner product on \(\mathbb{U}\), and \((\cdot, \cdot)\) denotes the standard \(l_2\)-inner product.}
\end{definition}
\begin{assumption}[convergent colver]
\label{assumption1}
\(\sfB\) is a convergent solver in \(\mathbb{U}\), i.e,
\begin{equation*}
      (\sfA \sfu,\sfu) \le \omega (\sfB^{-1} \sfu, \sfu), \quad \forall \sfu\in \mathbb{U}, \text{ with }\, \omega \in (0,2).
   \end{equation*}
\end{assumption}
\begin{assumption}[spectral equivalence]
\label{assumption2}
 \(\sfB^{-1}\) is spectrally equivalent to \(\sfA\) on the subspace \(\mathbb{V}\), i.e., there exist \(a>0\) and \(b>0\) such that
\[
   \alpha (\sfA \sfv, \sfv) \le (\sfB^{-1} \sfv, \sfv) \le \beta (\sfA \sfv , \sfv), \quad \sfv \in \mathbb{V}.
\]
\end{assumption}
\begin{definition}[AMGF]
\label{def:AMGF}
 We define \(\sfM\), the \emph{multiplicative} AMG with filtering preconditioner, by
\begin{equation*}
\label{eq:M_itmat}
\sfI-\sfM\sfA = (\sfI-\sfB\sfA)(\sfI-\sfP\sfA_{\sfw}^{-1}\sfP^{\top} \sfA)(\sfI-\sfB\sfA)
\end{equation*}
or equivalently:
\begin{equation*}
      \label{eq:M_splitting}
      \sfM:=\begin{pmatrix}
          \sfI & \sfP
       \end{pmatrix}
      \overline{\sfM}^{-1}
       \begin{pmatrix}
        \sfI \\ \sfP^{\top}
       \end{pmatrix},
   \end{equation*}
   where
   \begin{equation*}
      \overline{\sfM} =
      \begin{pmatrix}
         \sfB^{-1} & 0 \\
         \sfP^{\top} \sfA  & \sfI
      \end{pmatrix}
      \begin{pmatrix}
         (2\sfB^{-1}-\sfA)^{-1} & 0 \\
         0  & \sfA_\sfw
      \end{pmatrix}
      \begin{pmatrix}
         \sfB^{-1} & \sfA \sfP \\
         0  & \sfI
      \end{pmatrix}
   \end{equation*}
\end{definition}
\begin{lemma}[condition number estimate]
\label{lemma:main}
Suppose \cref{assumption1,assumption2} hold. Then \(\sfM\), given in \cref{def:AMGF}, is spectrally equivalent to \(\sfA^{-1}\) and the condition number of the preconditioned system \(\sfM\sfA\) satisfies
\[\kappa(\sfM\sfA) \le \frac{2(\beta + 2+\omega)}{2-\omega}\]
\end{lemma}
The proof follows standard subspace correction arguments \cite{xu1992} by establishing a stable decomposition of \(\mathbb{U}\) into \(\mathbb{V}\) and \(\mathbb{W}\) as in \cref{lemma:stability}, and deriving upper and lower bounds on the preconditioned norm \(\|\sfu\|_{\sfM^{-1}}\) with respect to \(\|\sfu\|_{\sfA}\) for all \(\sfu \in \mathbb{U}\). For the lower bound we use the following lemma.
\begin{lemma}[lower bound]
For any \(\sfu \in \mathbb{U}\) the following bound holds
\label{lem:upper_bound}
   \begin{equation*}
      (\sfA \sfu, \sfu) \le (\sfM^{-1} \sfu, \sfu)
   \end{equation*}
\end{lemma}
\begin{proof}
   First we note that the above bound is equivalent to \((\sfA \sfM \sfA \sfu, \sfu) \le (\sfA \sfu, \sfu)\). It suffices to prove that \(\sfA(\sfI- \sfM\sfA)\) is positive semi-definite. \cref{eq:M_itmat} gives
   \begin{equation*} \begin{aligned}
      \sfA - \sfA\sfM\sfA = (\sfI-\sfA\sfB)(\sfA-\sfA\sfP\sfA_{\sfw}^{-1}\sfP^{\top}\sfA)(\sfI-\sfB\sfA)
   \end{aligned}
   \end{equation*}
   This can also be found [\citenum{falgout2005}, Lemma~2.2].  Since \(\sfA\) and \(\sfB\) are symmetric, it is enough to show that \(\sfA-\sfA\sfP\sfA_{\sfw}^{-1}\sfP^{\top}\sfA\) is semi-definite. Indeed, let \(\sfK:=\sfI - \sfP \sfA_{\sfw}^{-1} \sfP^{\top} \sfA\). The result follows by noticing that
   \[\sfK^{\top} \sfA \sfK = \sfA-\sfA\sfP\sfA_{\sfw}^{-1}\sfP^{\top}\sfA\]
   \end{proof}
 The proof of the upper bound is derived from the following two Lemmas. The first one provides a stability estimate and the second is using a key identity from  subspace correction theory.
\begin{lemma}[stability estimate]
\label{lemma:stability}
Consider the space splitting 
   \begin{equation}
\label{eq:splitting}
\sfu = \sfv + \sfP \sfw, \quad \sfu\in \mathbb{U}, \sfv \in \mathbb{V},\sfw \in \mathbb{W}.
\end{equation}
and let \cref{assumption2} hold. Then we have the following stability estimate
   \begin{equation}
   \label{eq:stability}
      2\|\sfv\|^2_{\sfB^{-1}} + (2+\omega)\|\sfP \sfw\|^2_{\sfA} \leq 2(\beta + 2+\omega) \|\sfu\|^2_{\sfA}.
   \end{equation}
\end{lemma}
\begin{proof}
   Let \(\sfv\) be the \(\sfA\)-orthogonal projection of \(\sfu\) onto \(\mathbb{V}\), i.e.,
   \[
   \sfv := (\sfI - \sfP (\sfP^{\top} \sfA \sfP)^{-1} \sfP^{\top} \sfA) \sfu, \text{ and } \sfw = (\sfP^{\top} \sfA \sfP)^{-1} \sfP^{\top} \sfA \sfu.
   \]
   \emph{Bound on the first term:}
   By definition of the \(\sfA\)-orthogonal projection we obtain
   \[\|\sfv\|_{\sfA} \le \|\sfI - \sfP (\sfP^{\top} \sfA \sfP)^{-1} \sfP^{\top} \sfA\|_{\sfA} \|\sfu\|_{\sfA} \le \|\sfu\|_{\sfA}\]
In addition, by \cref{assumption2} on spectral equivalence on \(\mathbb{V}\), we have:
   \begin{equation*}
      \|\sfv\|^2_{\sfB^{-1}} \le \beta \|\sfv\|^2_{\sfA} \le \beta \|\sfu\|^2_{\sfA}.
   \end{equation*}
\emph{Bound on the second term:} By construction, \(\sfP \sfw = \sfu - \sfv\), and by triangle inequality we have
\[
  \|\sfP \sfw\|_{\sfA} \leq \|\sfu\|_A + \|\sfv\|_A \leq 2 \|\sfu\|_A
\]
Combining the above bounds we get the final stability estimate \cref{eq:stability} and this completes the proof.
\end{proof}
\begin{lemma}
\label{lemma:schwarz2}
Consider the decomposition \cref{eq:splitting} and let \cref{assumption1} hold. Then
   \begin{equation*}
      (\sfM^{-1} \sfu, \sfu) \le
       \frac{1}{2-\omega} \inf_{\sfw \in \mathbb{W}}
      \left\{ 2\|\sfu - \sfP \sfw\|^2_{\sfB^{-1}} +
       (2+\omega) \|\sfP \sfw\|^2_{\sfA}  \right\}.
   \end{equation*}
\end{lemma}
\begin{proof}
   The proof relies on the Schwarz lemma [\citenum{xu2002}, Lemma 2.4]. Applying this lemma on $\sfM$ gives
%
   \begin{equation}
   \label{inf_eq}
      (\sfM^{-1}  \sfu, \sfu) = \inf_{\sfw \in \mathbb{W}} \left\{\overline{\sfM}
      \begin{pmatrix}
         \sfu - \sfP \sfw \\
         \sfw
      \end{pmatrix},
      \begin{pmatrix}
         \sfu - \sfP \sfw \\
         \sfw
      \end{pmatrix} \right\} .
   \end{equation}
%
%
After expanding \cref{inf_eq} and noticing that \cref{assumption1} implies that \(\forall \sfu\in \mathbb{U}\)
\begin{equation*}
   \left((2\sfB^{-1} - \sfA)^{-1} \sfu, \sfu\right) \le \frac{1}{2-\omega} \left(\sfB \sfu, \sfu \right) \le \frac{\omega}{2-\omega} \left(\sfA^{-1} \sfu, \sfu\right)
\end{equation*}
we get
\begin{equation*}
\begin{aligned}
 (\sfM^{-1}  \sfu, \sfu)
 &\le \inf_{\sfw \in \mathbb{W}} \left\{ \left(\sfA_{\sfw} \sfw, \sfw\right) + \frac{1}{2-\omega} \left( \sfu-\sfP\sfw + \sfB \sfA \sfP \sfw, \sfB^{-1} \left(\sfu-\sfP\sfw\right) + \sfA\sfP\sfw \right) \right\} \\
 &\le \inf_{\sfw \in \mathbb{W}} \left\{ \|\sfP \sfw\|^2_{\sfA} + \frac{1}{2-\omega} \left( \|\sfu-\sfP\sfw\|^2_{\sfB^{-1}} + \omega \|\sfP \sfw\|^2_{\sfA} + 2\left(\sfu-\sfP\sfw, \sfA\sfP\sfw\right) \right) \right\}
\end{aligned}
\end{equation*}
The result follows by noting that
\begin{equation*}
\begin{aligned}
2(\sfu-\sfP\sfw, \sfA\sfP\sfw)
&= 2\left(\sfB^{-1/2} \left(\sfu-\sfP\sfw\right), \sfB^{1/2} \sfA \sfP \sfw\right) \\
&\le \|\sfB^{-1/2} (\sfu-\sfP\sfw)\|^2 + \|\sfB^{1/2} \sfA \sfP \sfw\|^2 \\
&\le \|\sfu-\sfP\sfw\|^2_{\sfB^{-1}} + \omega \| \sfP \sfw\|^2_{\sfA} 
\end{aligned}
\end{equation*}
\phantom{x}
\end{proof}
Combining \cref{lemma:stability,lemma:schwarz2} gives an estimate of the upper bound and therefore the proof of \cref{lemma:main} is derived by
\begin{equation*}
\begin{aligned}
(\sfA\sfu,\sfu) \le (\sfM^{-1} \sfu,\sfu) &\le \frac{1}{2-\omega} (2\|\sfu - \sfP\sfw\|^2_{\sfB^{-1}} + (2+\omega)\|\sfP\sfw\|^2_{\sfA}) \\
& \le \frac{2(\beta + 2+\omega)}{2-\omega} (\sfA \sfu,\sfu)
\end{aligned}
\end{equation*}
\begin{remark}
We emphasize that this result holds in a more general setting beyond the contact mechanics context described here. Specifically, the same condition number estimate applies when \(\sfP: \mathbb{W} \rightarrow \mathbb{U}\) is any full column rank matrix,  \(\sfA\) is an SPD matrix defined on \(\mathbb{U}\), \(\sfB\) is a convergent solver for \(\sfA\) in \(\mathbb{U}\) and is spectrally equivalent with \(\sfA^{-1}\) on \(\mathbb{V}\),  the \(\sfA\)-orthogonal complement of range(\(\sfP\)). 
\end{remark}

\begin{remark}
\label{remark:upper_bound_estimate}
In our framework \(\sfB := \text{AMG}(\sfA)\) which can be made convergent with \mbox{\(\omega = 1\)} (and \(\alpha = 1\) in \cref{assumption2})  using \(l_1\)-smoothers \cite{Baker2011}. In this case, the condition number upper bound simplifies to
   \begin{equation}
   \kappa(\sfM \sfA) \le 2 (\beta + 3).
   \label{eq:practical_bound}
   \end{equation}
Here \(\beta\) represents the condition number of the elasticity matrix without contact, preconditioned with AMG. In practice, this result implies that the increase in PCG iterations is bounded by approximately a factor of \(\sqrt{2}\).
\end{remark}
%

%
%

%% file: results.tex
\section{Numerical experiments}
\label{section:results}

In this section, we present numerical experiments to evaluate the performance of the proposed AMGF preconditioner for solving large-scale contact mechanics problems using an IP method. The primary objectives are to assess the convergence behavior of the preconditioned solver across different problem setups, examine its robustness with respect to both mesh refinement and the enforcement of contact constraints, which significantly impact system conditioning, and analyze its scalability for both linear and nonlinear contact problems. All the numerical experiments are implemented using the MFEM FE library \cite{mfem1,mfem2}. The gap function and its Jacobian are computed through Tribol \cite{Tribol}. The AMGF preconditioner employs the \emph{BoomerAMG} solver along with the convergent \(l_1\)-Gauss--Seidel smoother \cite{Baker2011} from \emph{hypre} \cite{hypre,falgout2002hypre} and a multi-frontal direct solver from MUMPS \cite{MUMPS:1} for the subspace filtering.

We consider three benchmark problems: a two-block contact problem, an ironing problem, and a beam-sphere problem. Each test is solved using the Newton-based IP method, where the linear system at each optimization iteration is solved by the AMGF--PCG solver. The subsequent numerical results highlight the effectiveness of the AMGF preconditioner in improving solver convergence, keeping iteration counts bounded despite mesh refinement and contact enforcement.
\subsection{Two-block contact problem}
As a first example we consider the two-block contact scenario introduced in \Cref{section:Model_problem}. This setup is used to systematically evaluate how the enforcement of contact constraints affects solver performance under successive mesh refinements. \Cref{tab:problem_size_test4} reports the dimension of the solution space \(\mathbb{U}\), denoted by \(\tt n\),  as well as the maximum dimension of the contact subspace \(\mathbb{W}\), denoted by \(\tt n_{\tt c}^{max}\), encountered throughout the time-stepping process. Note that \(\dim(\mathbb{W})\) varies over time, since contact constraints are re-evaluated at each time step. We also report the maximum number of contact constraints, \(\tt m^{max}\), observed during the simulation. Solver performance is summarized via \(\tt{k}_{IP}^{avg}\), the average number of IP iterations over all time steps and \(\tt{k}_{AMGF}^{avg}\), the average number of AMGF--PCG iterations required across all time and optimization steps. In parentheses next to each \(\tt{k}_{AMGF}^{avg}\) value, we include an upper bound derived from the theoretical estimate \cref{eq:practical_bound}. In practice, the upper bound on the AMGF--PCG iteration count is computed by recording the number of AMG--PCG iterations required for the corresponding contact-free problem, and then multiplying that value by $\sqrt{\frac{2 (\beta + 3)}{\beta}}$ which approaches $\sqrt{2}$ as $\beta$ become large.
\begin{table}[H]
    \small
    \centering
    \begin{tabular}{lrrrrrr}
        \hline
        \hline
                                  & Mesh 1  & Mesh 2 & Mesh 3  & Mesh 4    & Mesh 5      & Mesh 6 \\
        \hline
         \(\tt n\)                & 13,380  & 93,714 & 699,486 & 5,401,014 & 42,440,550  & 336,477,894 \\
         \(\tt n_{\tt c}^{max}\)  & 348     & 1,092  & 4,002   & 15,369    & 60,207      & 238,458 \\
         \(\tt m^{max}\)          & 81      & 289    & 1,089   & 4,225     & 16,641      & 66,049       \\
        \hdashline
        
        \(\tt{k}_{IP}^{avg}\)    & 16      & 20        & 20        & 20        & 22         & 23    \\
         \(\tt{k}_{AMGF}^{avg}\)  & 72 (155)& 72 (158)  & 81 (175)  & 97 (210)  & 124 (268)  & 150 (337)    \\
        \hline
    \end{tabular}
    \caption{{\bf Solver iteration counts for the two-block contact problem across mesh refinement levels.} Starting from the initial coarse mesh (Mesh 1), each subsequent mesh is obtained by uniformly refining the previous one. Here \(\tt n\) denotes the total number of DOFs in the solution space \(\mathbb{U}\), while \(\tt n_{\tt c}^{max}\) and \(\tt m^{max}\) represent the maximum dimension of the contact subspace \(\mathbb{W}\) and the maximum number of contact constraints encountered across all time steps, respectively. We also list \(\tt{k}_{IP}^{avg}\), the average number of IP iterations over all time steps. The last row reports \(\tt{k}_{AMGF}^{avg}\), the average AMGF--PCG iteration count across all optimization and time steps. In parentheses we include an upper bound estimate derived from the theoretical condition number upper bound given by \cref{eq:practical_bound}.}
    \label{tab:problem_size_test4}
\end{table}
\Cref{fig:test4_results} presents the number of AMGF--PCG and AMG--PCG iterations per optimization step across different refinement levels. Each subplot corresponds to a specific refinement level, ranging from Mesh 1 (coarsest) to Mesh 6 (finest), and shows results for two time steps. The estimated upper bound for AMGF--PCG iterations, derived from the condition number estimate \cref{eq:practical_bound} in \Cref{remark:upper_bound_estimate}, is indicated by a horizontal dashed line in each plot. As expected, solver iteration counts increase with mesh refinement. The growth in PCG iterations mirrors the behavior of standard AMG methods, which are sensitive to the number of processors used in parallel. This sensitivity is due in part to the use of the \(l_1\)\nobreakdash-Gauss--Seidel smoother at each AMG level, which exhibits minor performance degradation as core counts increase \cite{Baker2011}. This trend is also reflected in the AMGF--PCG increasing upper bound.

More importantly, the results highlight the limitations of the classical AMG preconditioner: in the presence of contact constraints, AMG--PCG often suffers from dramatic iteration count growth, with some curves reaching the maximum PCG iterations, chosen here as 5000, which indicates failure to converge to the requested tolerance. In sharp contrast, the AMGF--PCG solver maintains robust performance, with iteration counts staying well within theoretical bounds across all refinement levels and time steps. These results confirm that the AMGF preconditioner effectively mitigates the ill-conditioning introduced by contact enforcement, enabling reliable and efficient convergence throughout the optimization process.
\begin{figure}[H]
    \centering
    \begin{subfigure}{0.32\textwidth}
        \centering
        \includegraphics[width=1.0\linewidth]{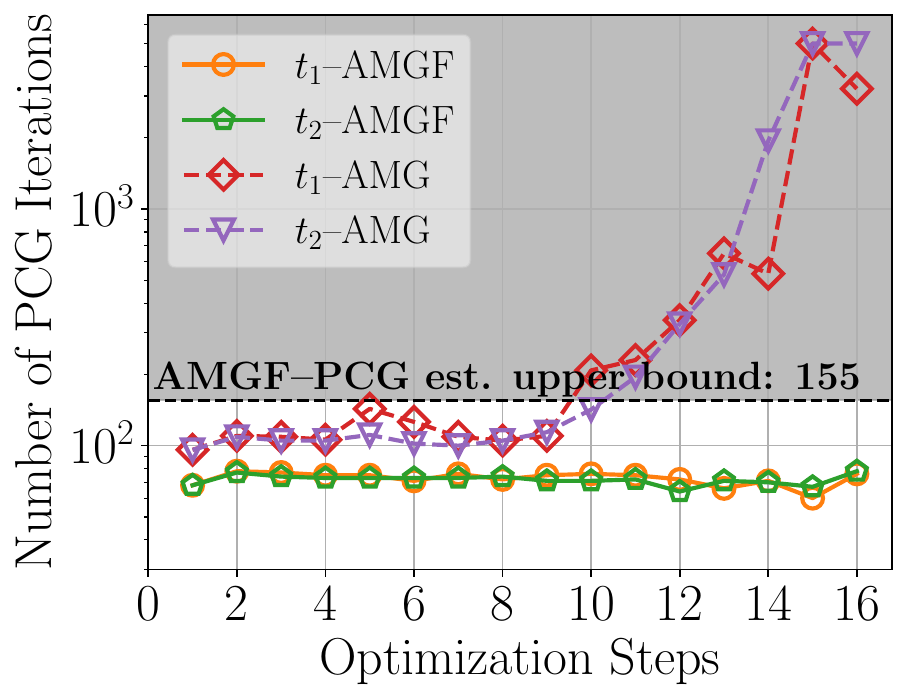}
        \caption{Mesh 1: 13,380 DOFs}
        \label{fig:testNo4_ref1}
    \end{subfigure}
    \begin{subfigure}{0.32\textwidth}
        \centering
        \includegraphics[width=1.0\linewidth]{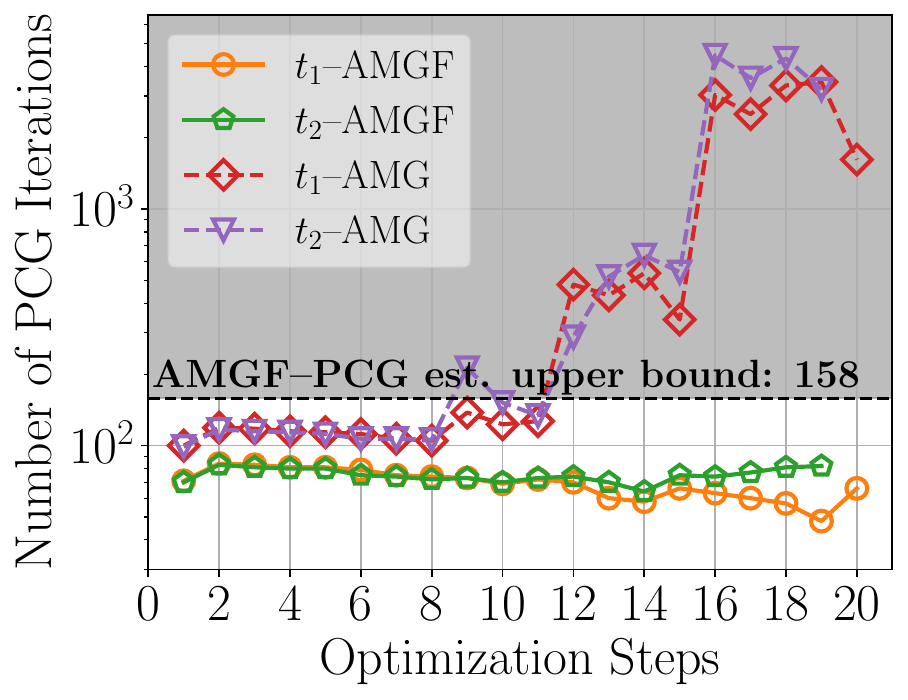}
        \caption{Mesh 2: 93,714 DOFs}
        \label{fig:testNo4_ref2}
    \end{subfigure}
        \begin{subfigure}{0.32\textwidth}
        \centering
        \includegraphics[width=1.0\linewidth]{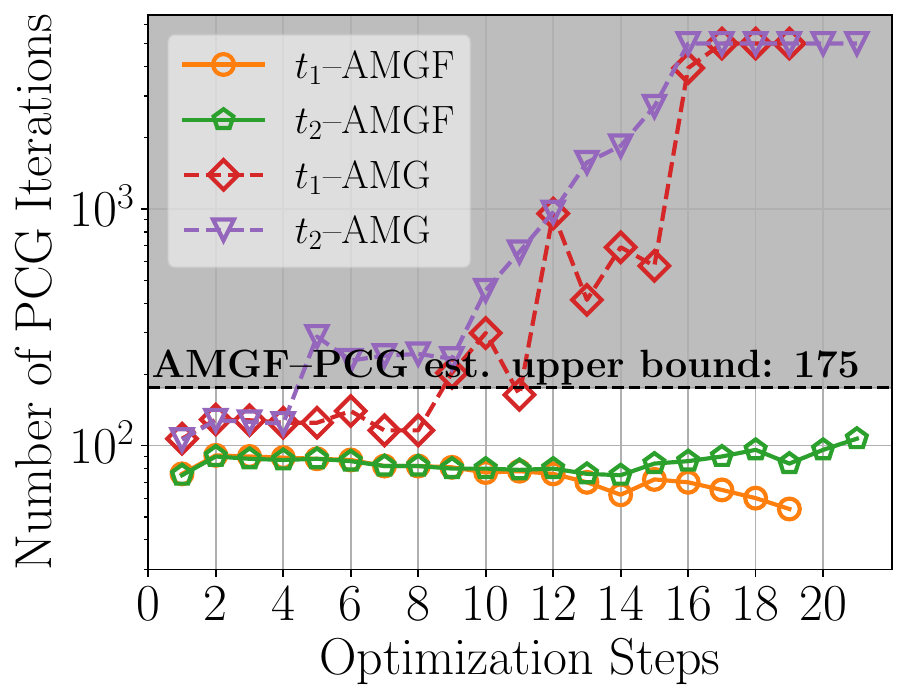}
        \caption{Mesh 3: 699,486 DOFs}
        \label{fig:testNo4_ref3}
    \end{subfigure}

    \begin{subfigure}{0.32\textwidth}
        \centering
        \includegraphics[width=1.0\linewidth]{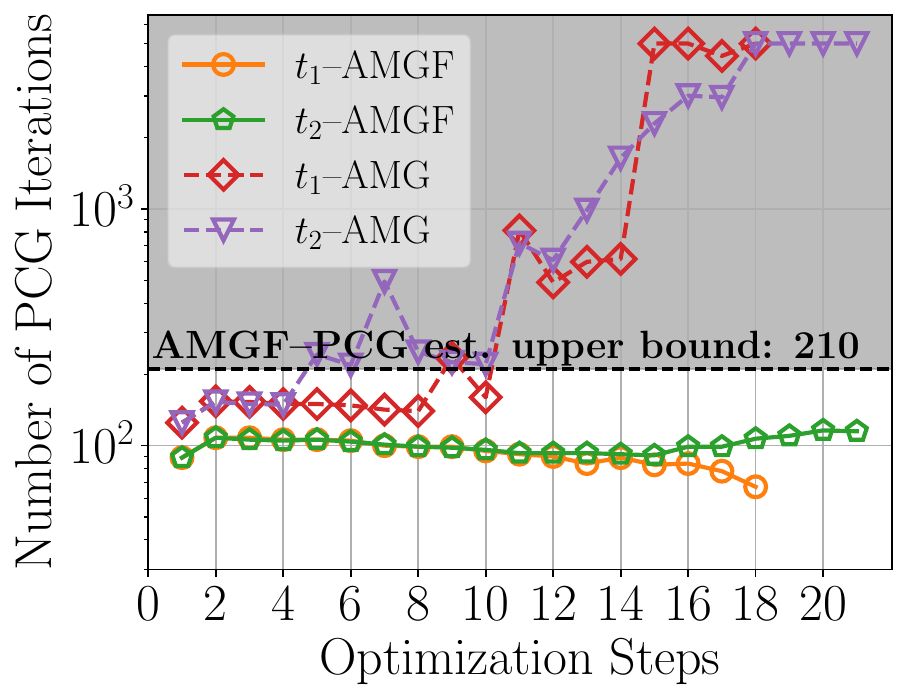}
        \caption{Mesh 4: 5,401,014 DOFs}
        \label{fig:testNo4_ref4}
    \end{subfigure}
    \begin{subfigure}{0.32\textwidth}
        \centering
        \includegraphics[width=1.0\linewidth]{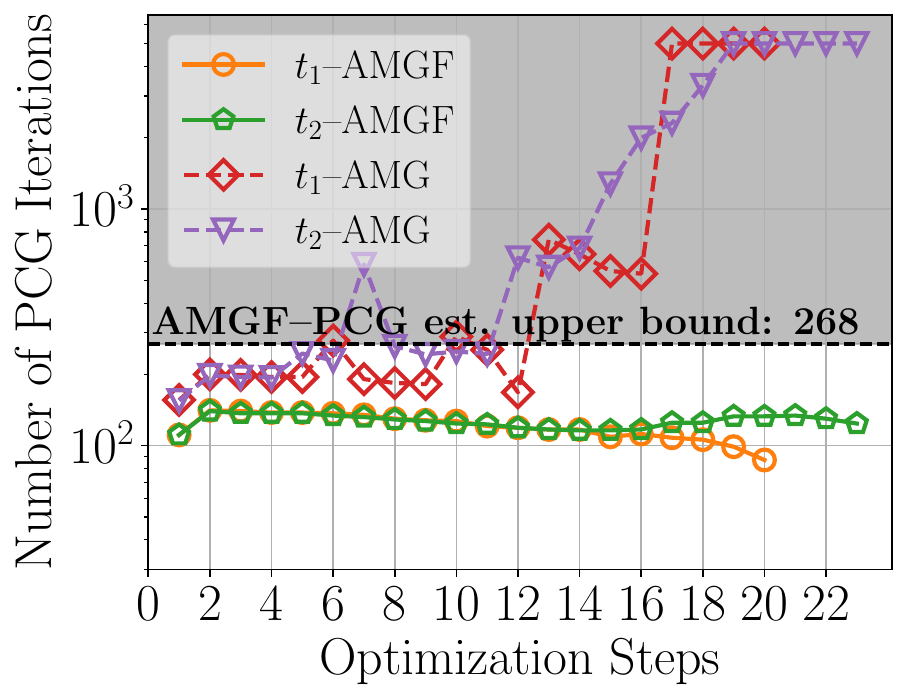}
        \caption{Mesh 5: 42,440,550 DOFs}
        \label{fig:testNo4_ref5}
    \end{subfigure}
        \begin{subfigure}{0.32\textwidth}
        \centering
        \includegraphics[width=1.0\linewidth]{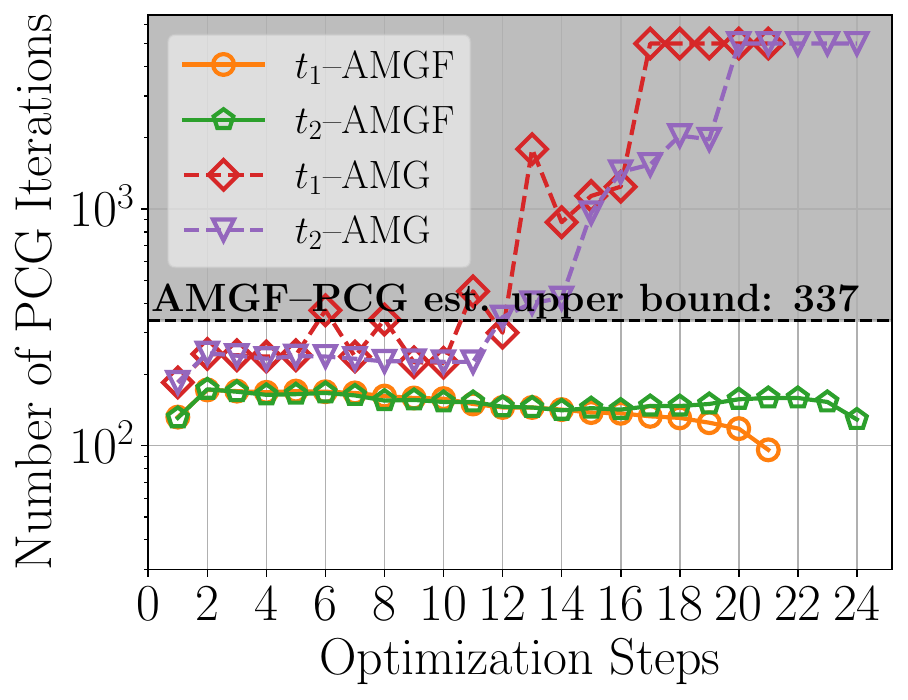}
        \caption{Mesh 6: 336,477,894 DOFs}
        \label{fig:testNo4_ref6}
    \end{subfigure}
    \caption{Comparison of AMGF--PCG and AMG-PCG solvers with respect to iteration count for the two-block problem. Each curve represents the PCG iteration count throughout the IP optimization for time steps $t_i$, $i=1,2$. The horizontal line indicates the estimate of the upper bound derived from the theoretical condition number result \cref{eq:practical_bound}, i.e., the number of iterations of the AMGF--PCG solver when applied to the contact problem is bounded by approximately \(\sqrt{2}\) times the AMG--PCG solver count when applied to the contact-free problem.}
    \label{fig:test4_results}
\end{figure}
\subsection{Ironing problem}
The \emph{ironing} problem is a common benchmark in computational contact mechanics \cite{pusolaur04a}. It consists of two deformable bodies, a cylindrical die resting on an elastic slab (rectangular block) as illustrated in \Cref{fig:testNo5_diagram}. The die has a nonhomogeneous Dirichlet BC at its top face \(\Gamma_1\) given by \(u_{\Gamma_1}:=(\frac{15}{7},0,-\frac{5}{7})\), while the slab has a homogeneous Dirichlet BC at its bottom surface, i.e, the die is pushed towards the negative \(z\) and the positive \(x\) direction. Here, the top face of the large block serves as the mortar surface, while the bottom face of the die is the nonmortar surface, constrained against the mortar surface. All other boundary surfaces are traction free and no other external body forces are applied. As in \Cref{section:Model_problem} the Lam\'{e} parameters for the two bodies correspond to the Young moduli  \(E_{\text{die}} = 1000 , E_{\text{slab}} = 1\) and the Poisson's ratios \(\nu_{\text{die}} = 0, \nu_{\text{slab}}=0.499\).
\begin{figure}[H]
    \begin{subfigure}[b]{0.49\textwidth}
    \centering
    \includegraphics[width=1.0\linewidth]{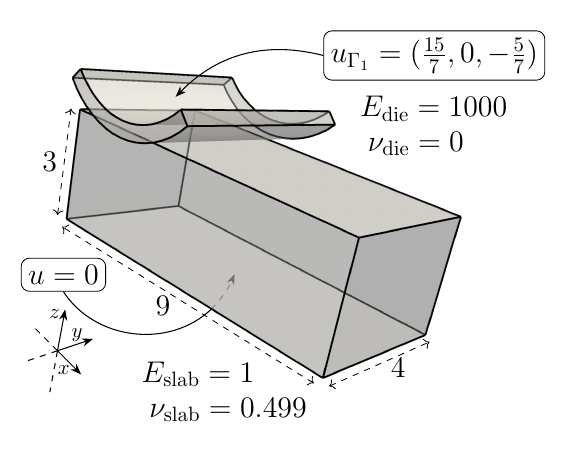}
    \caption{Initial configuration}
    \label{fig:testNo5_diagram}
    \end{subfigure}
    \begin{subfigure}[b]{0.49\textwidth}
        \centering
        \includegraphics[width=0.87\linewidth]{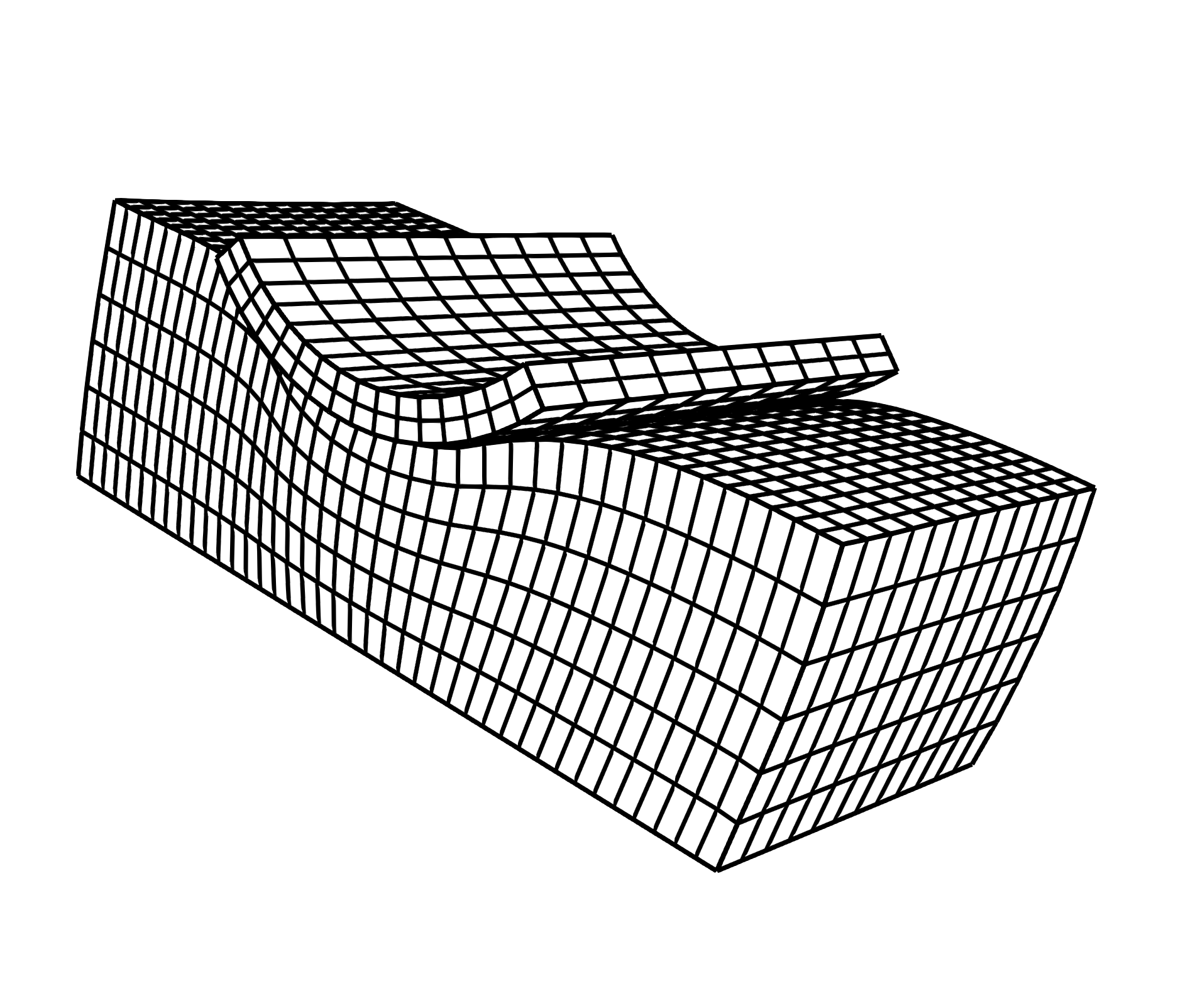}
    \caption{Final configuration for mesh level 1}
    \label{fig:testNo5_diagram-final}
    \end{subfigure}
    \caption{Ironing problem: a small die (top body) is in contact with a rectangular slab (bottom body). The simulation is driven by a nonhomogeneous Dirichlet BC \(u = (\frac{15}{7},0,-\frac{5}{7})\) enforced at the top face of the die, i.e., the die is pushed towards the positive \(x\) and negative \(z\) direction. The bottom face of the slab is fixed at \(u=0\) and the rest of the boundary is traction free.}
\end{figure}
The problem is solved using ten time steps. During the first three time steps the die is incrementally pushed towards the negative \(z\) direction (downwards) and in the remaining seven time steps it's pushed towards the positive \(x\) direction, i.e.,
\[
u^{t_i}_{\Gamma_1} =
\begin{cases}
    (0, 0, -\frac{5}{7} \frac{i}{3}), & i = 1,\ldots,3 \\
    (\frac{15}{7} \frac{i-3}{7}, 0, -\frac{5}{7}),  &  i = 4,\ldots,10
\end{cases}
\]
\begin{figure}[H]
\setlength{\belowcaptionskip}{-1pt} %
    \centering
        \begin{subfigure}{0.295\textwidth}
        \centering
        \includegraphics[width=1.0\linewidth,trim=75 0 150 150, clip]{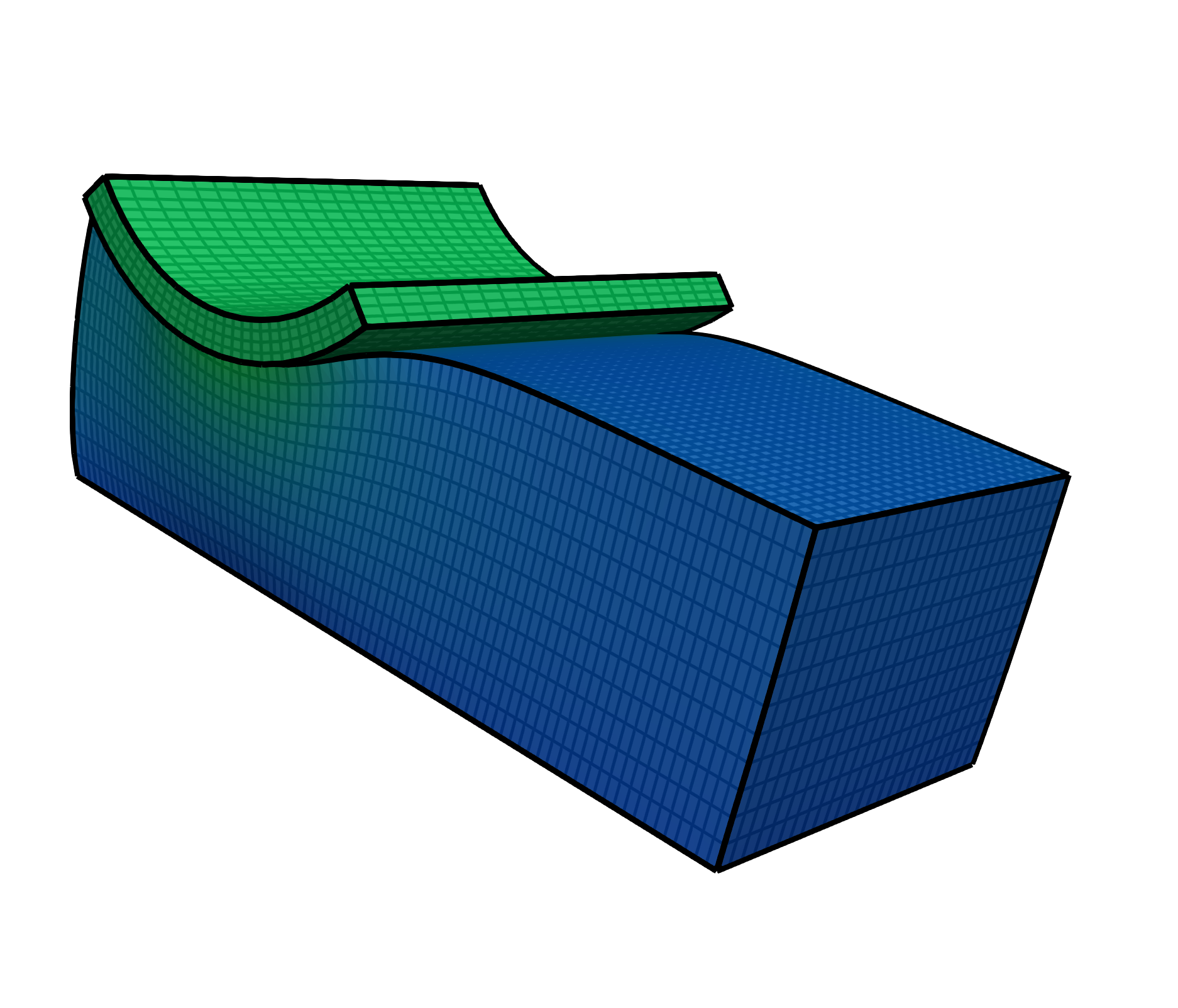}
    \end{subfigure}
        \begin{subfigure}{0.295\textwidth}
        \includegraphics[width=1.0\linewidth,trim=75 0 150 150, clip]{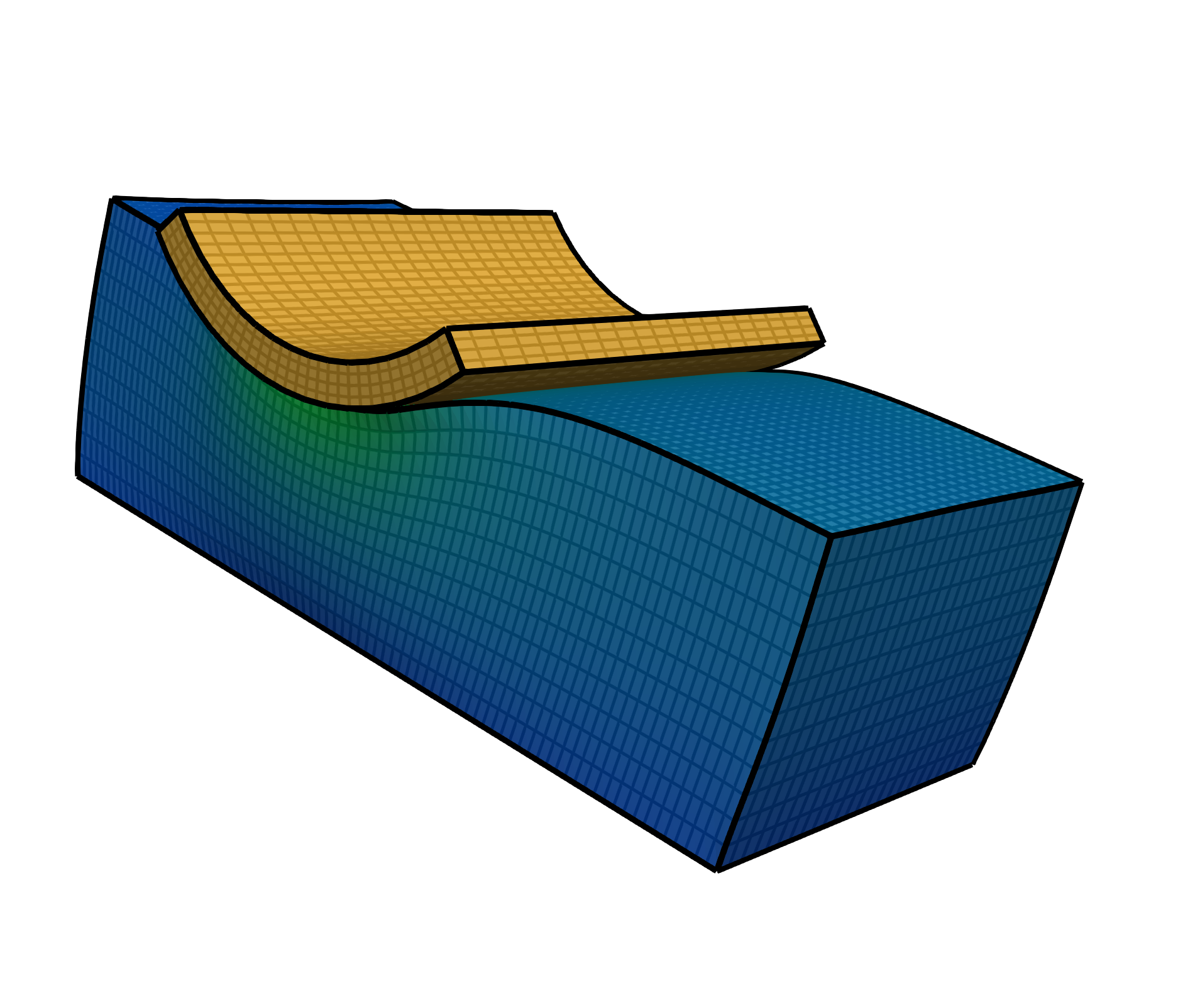}
    \end{subfigure}
    \begin{subfigure}{0.295\textwidth}
        \centering
        \includegraphics[width=1.0\linewidth,trim=75 0 150 150, clip]{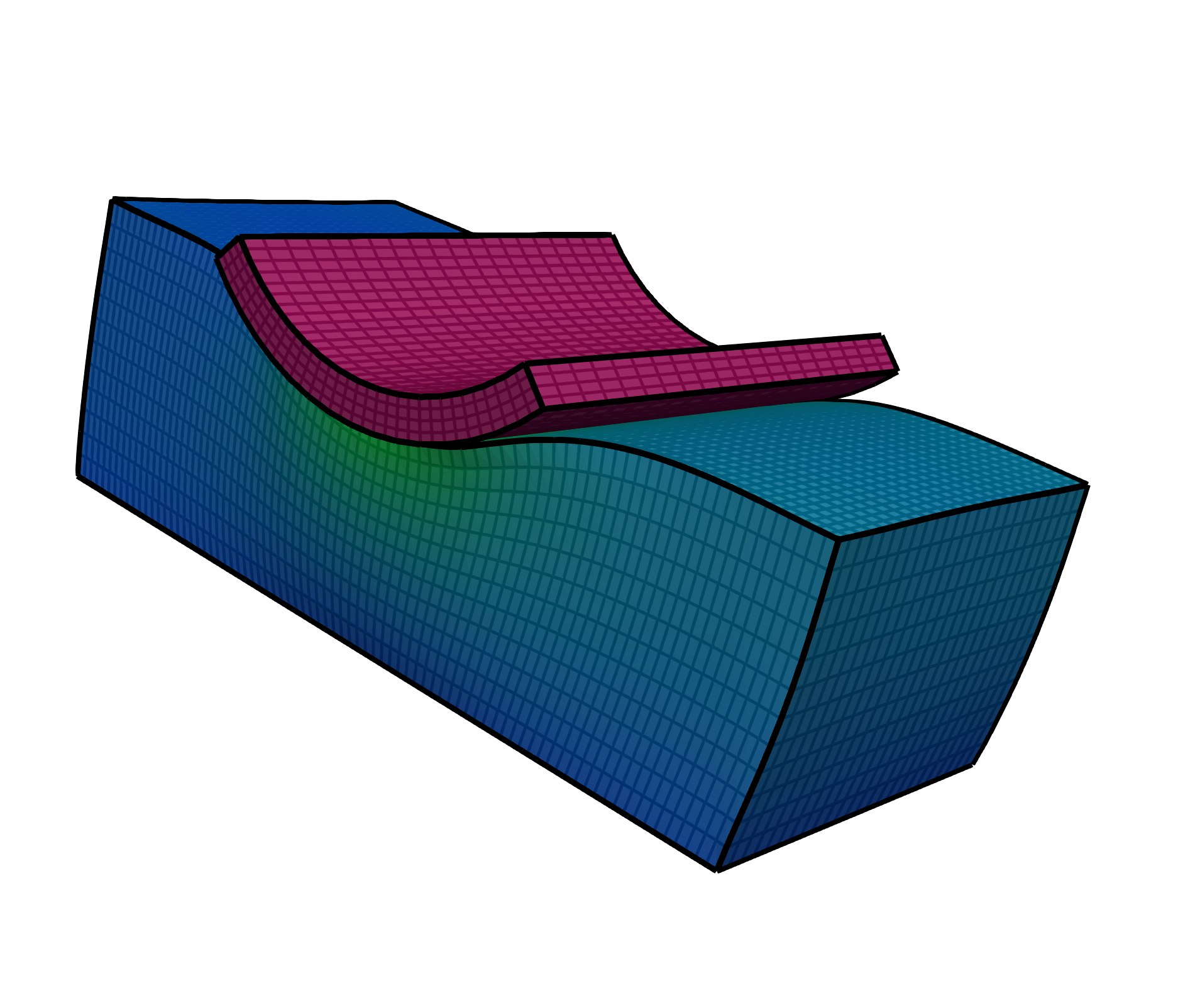}
    \end{subfigure}
    \begin{subfigure}{0.075\textwidth}
        \centering
        \includegraphics[width=1.5\linewidth, trim=600 0 250 300, clip]{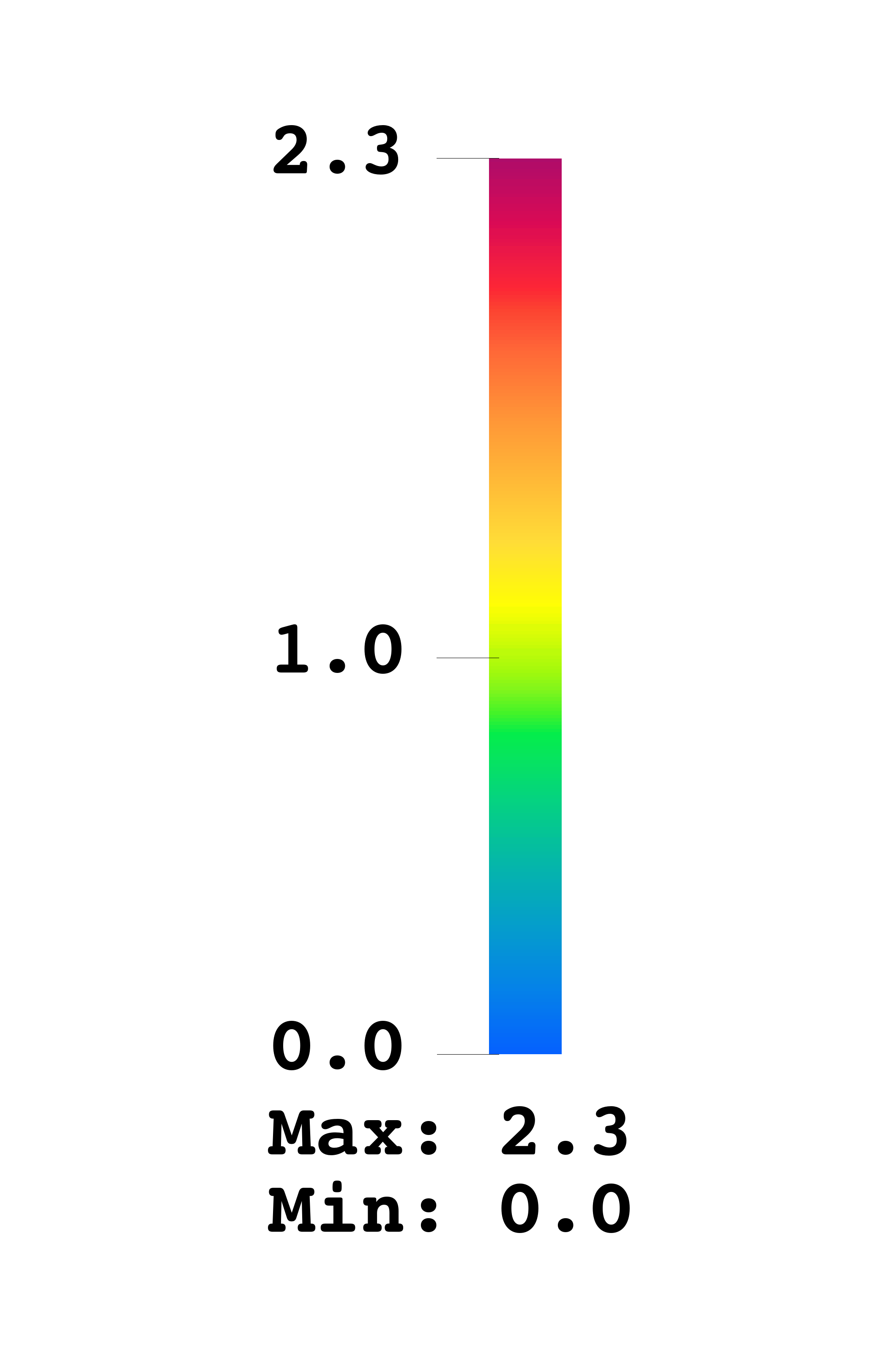}
    \end{subfigure}\hfill
    \caption{Deformed configurations and displacement magnitudes at time steps \(t_3,t_7,t_{10}\).}
    \label{fig:ex51_deformation}
\end{figure}
The computed displacement solutions and the corresponding deformed configurations are shown in \cref{fig:ex51_deformation}. The final configuration is also shown in \cref{fig:testNo5_diagram-final}. The IP inner linear system \cref{eq:linsystem1} is solved using the AMGF--PCG solver. As in the previous example \(\text{tol}_\text{\tiny PCG} = 10^{-10}\) and \(\text{tol}_\text{\tiny IP} =10^{-6}\).  The contact problem is solved across mesh refinement levels, ranging from a very coarse discretization with approximately \(13,000\) DOFs up to a fine mesh with roughly \(300\) million DOFs. In \Cref{tab:problem_size_test5} we report \(\tt n\), the total number of DOFs, as well as \(\tt n_c^{max}\) and \(\tt m^{max}\), the maximum number of DOFs in the contact subspace and the maximum number of contact constraints encountered across all time steps, respectively. We also include \(\tt{k}_{IP}^{avg}\), the average number of IP iterations over all time steps, and \(\tt k_{AMGF}^{avg}\) the average AMGF--PCG iteration number across all time and optimization steps. The numbers in parentheses indicate the corresponding iteration count upper bounds derived by \cref{eq:practical_bound}.
\begin{table}[H]
    \centering
    \scalebox{1}{
    \begin{tabular}{lrrrrrrr}
        \hline
        \hline
                                  && Mesh 1 & Mesh 2 & Mesh 3  & Mesh 4 & Mesh 5 & Mesh 6 \\
        \hline
        \(\tt n\)                 && 12,876 & 89,370 & 663,630 & 5,110,038 & 40,096,806  & 317,665,350 \\
       \(\tt n_{\tt c}^{max}\)    && 1,419  & 5,127  & 19,404  & 75,402          & 297,384     & 1,180,218 \\
       \(\tt m^{max}\)            && 185    & 655    & 2,417   & 9,247          & 36,179     & 143,263  \\
        \hdashline
   \(\tt{k}_{IP}^{avg}\)          && 13     & 15     & 16      & 19          & 20     & 21  \\
 \(\tt k_{AMGF}^{avg}\)           && 73 (168)     & 83 (187)     & 97 (198)      & 124 (238)       & 160 (298)  & 195 (384)  \\
        \hline
    \end{tabular}}
    \caption{{\bf Solver iteration counts for the ironing contact problem across mesh refinement levels.} Starting from the initial coarse mesh (Mesh 1), each subsequent mesh is obtained by uniformly refining the previous one. Here \(\tt n\) denotes the total number of DOFs in the solution space \(\mathbb{U}\), while \(\tt n_c^{max}\) and \(\tt m^{max}\) represent the maximum dimension of the contact subspace \(\mathbb{W}\) and the maximum number of contact constraints encountered across all time steps, respectively. We also list \(\tt{k}_{IP}^{avg}\), the average number of IP iterations over all time steps. The last row reports \(\tt{k}_{AMGF}^{avg}\), the average number of iterations required by AMGF--PCG solver across all optimization and time steps to solve the contact problem. The numbers in parentheses indicate the corresponding computed AMGF-PCG iteration count upper bounds derived by the condition number estimate \cref{eq:practical_bound}.}
    \label{tab:problem_size_test5}
\end{table}
\begin{figure}[H]
    \centering
    \begin{subfigure}{0.32\textwidth}
        \centering
        \includegraphics[width=0.99\linewidth]{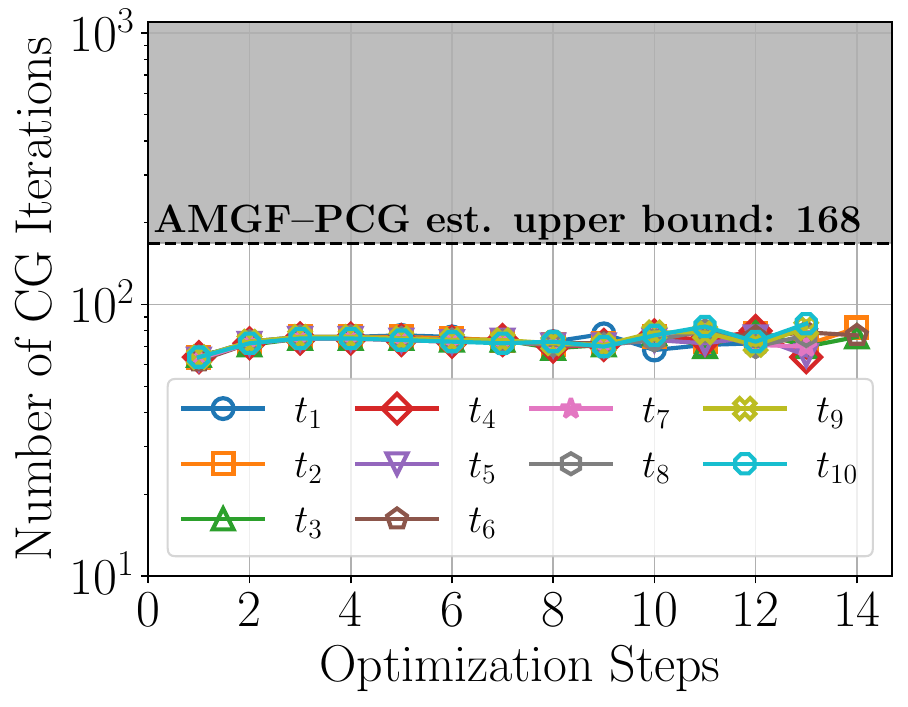}
        \caption{Mesh 1: 12,876 DOFs}
        \label{fig:testNo51_ref1}
    \end{subfigure}
    \begin{subfigure}{0.32\textwidth}
        \centering
        \includegraphics[width=0.99\linewidth]{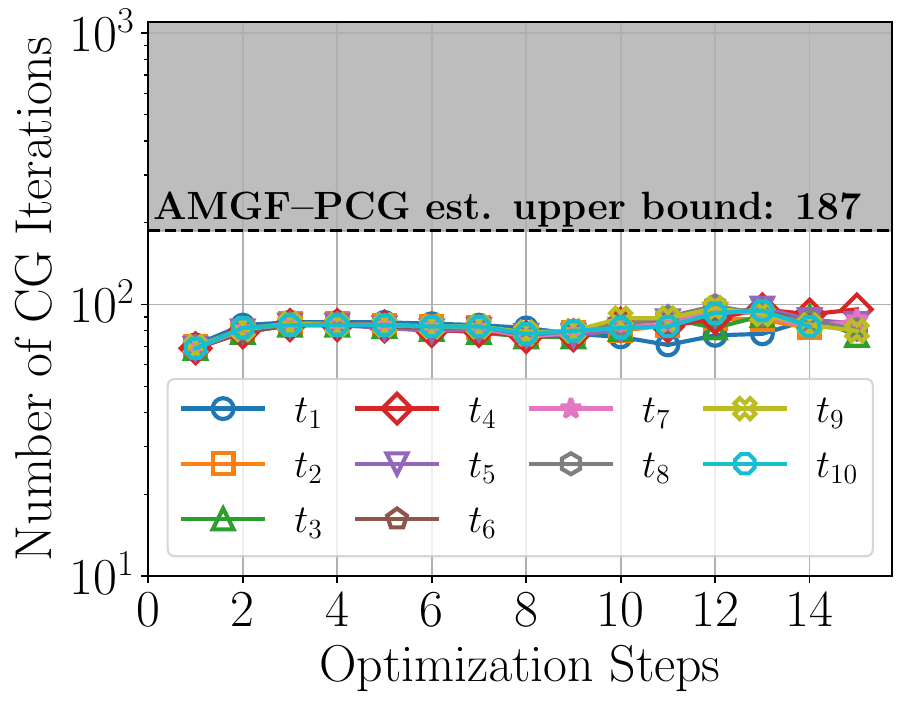}
        \caption{Mesh 2: 89,370 DOFs}
        \label{fig:testNo51_ref2}
    \end{subfigure}
        \begin{subfigure}{0.32\textwidth}
        \centering
        \includegraphics[width=0.99\linewidth]{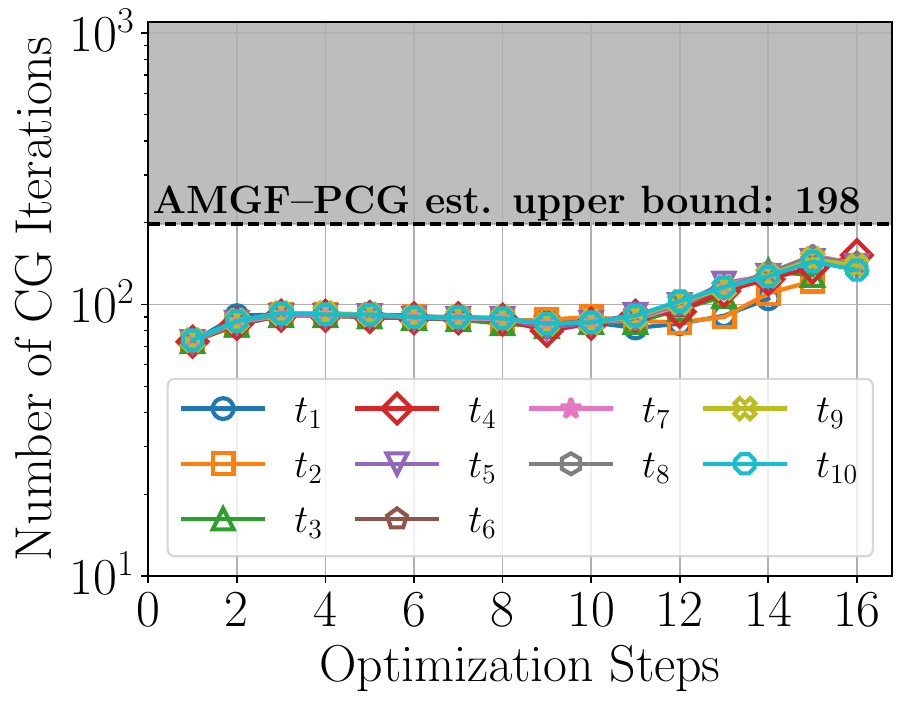}
        \caption{Mesh 3: 663,630 DOFs}
        \label{fig:testNo51_ref3}
    \end{subfigure}

    \begin{subfigure}{0.32\textwidth}
        \centering
        \includegraphics[width=0.99\linewidth]{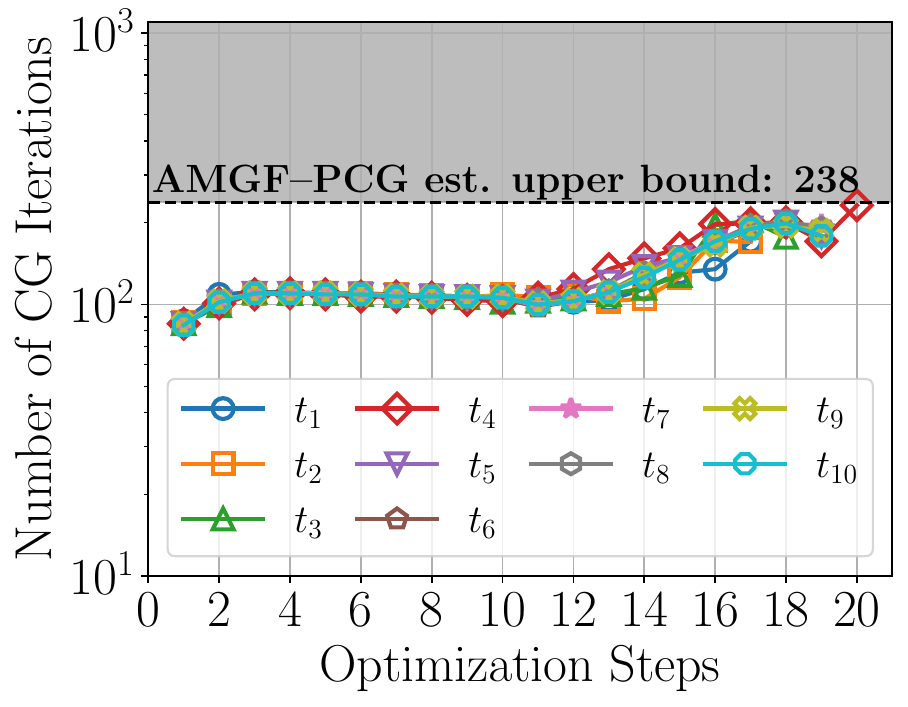}
        \caption{Mesh 4: 5,110,038 DOFs}
        \label{fig:testNo51_ref4}
    \end{subfigure}
    \begin{subfigure}{0.32\textwidth}
        \centering
        \includegraphics[width=0.99\linewidth]{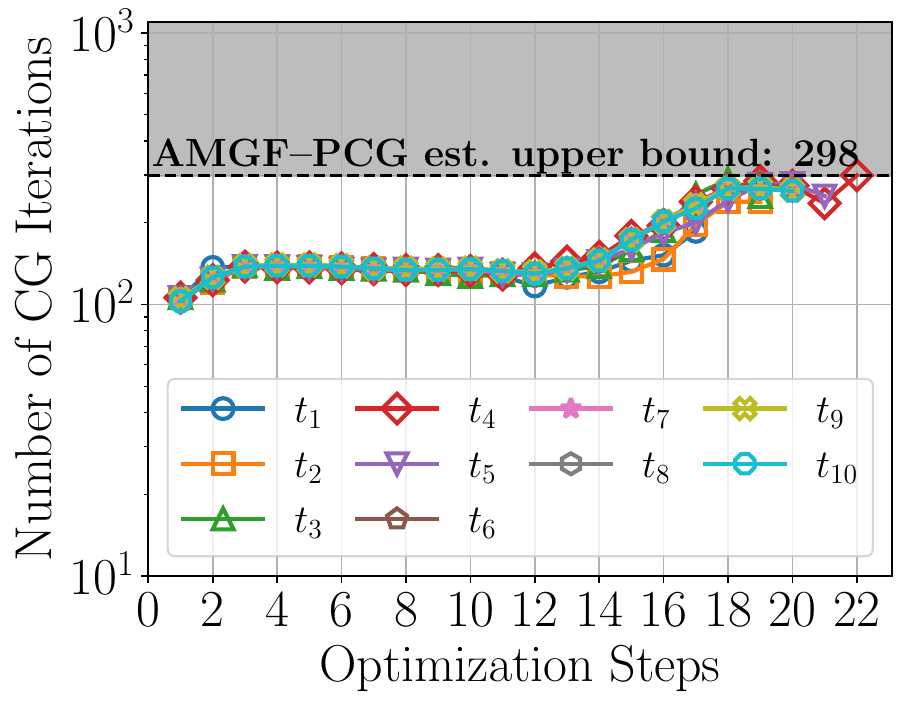}
        \caption{Mesh 5: 40,096,806 DOFs}
        \label{fig:testNo51_ref5}
    \end{subfigure}
        \begin{subfigure}{0.32\textwidth}
        \centering
        \includegraphics[width=0.99\linewidth]{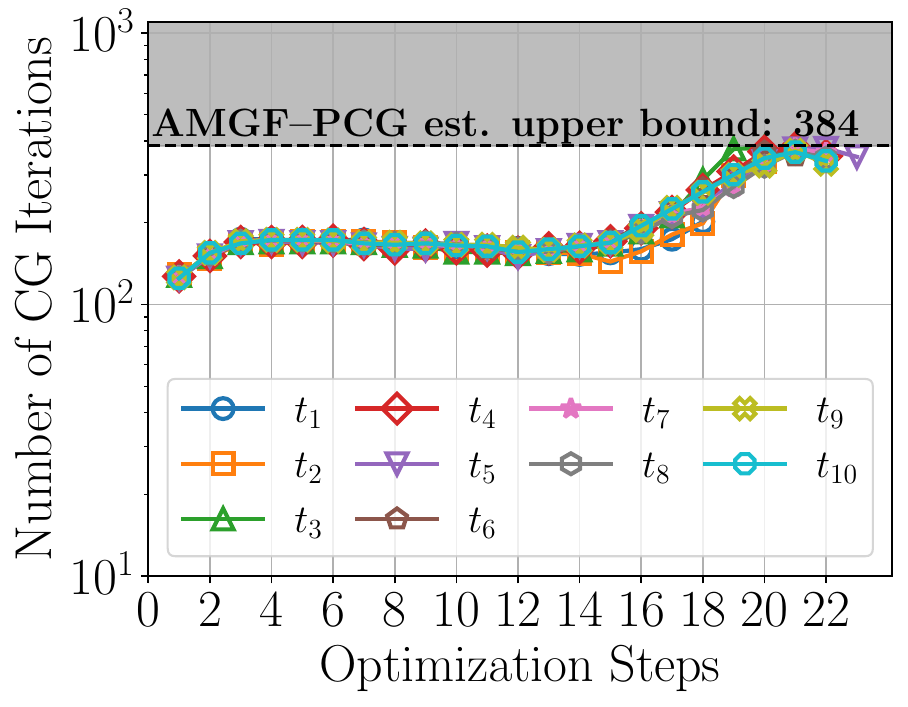}
        \caption{Mesh 6: 317,665,350 DOFs}
        \label{fig:testNo51_ref6}
    \end{subfigure}
    \caption{AMGF--PCG convergence results for the ironing problem. Each curve represents the AMGF--PCG iteration count throughout the IP optimization for time steps $t_i$, $i=1,\dots,10$. The horizontal line indicates the estimate of the upper bound derived from the theoretical condition number result \cref{eq:practical_bound} in \Cref{remark:upper_bound_estimate}, i.e., the number of iterations of the AMGF--PCG solver when applied to the contact problem is bounded by approximately \(\sqrt{2}\) times the AMG--PCG solver count when applied to the contact-free problem.}
    \label{fig:test51_results}
\end{figure}
\Cref{fig:test51_results} illustrates the convergence behavior of the AMGF--PCG solver across all mesh refinement levels for the ironing test. Each subplot shows iteration counts over optimization steps, with separate curves corresponding to different time steps. The observed trends are consistent with those of the two-block problem discussed in \Cref{section:Model_problem}, particularly in how mesh refinement and the activation of contact constraints affect solver performance. Notably, the influence of contact enforcement becomes more pronounced when the die (top body) begins to slide after the fourth time step. This results in a moderate increase in the number of solver iterations as the IP method approaches convergence, suggesting a temporary worsening of the system's condition number. However, the number of iterations never exceeds the estimated upper bound, thereby validating the theoretical condition number estimate from \Cref{section:precon}. Finally, we omit iteration counts for the standard AMG--PCG solver in this figure, as it frequently fails to converge within 5,000 iterations for this problem.

\subsection{Beam-sphere problem}
\label{sub:beam_sphere_problem}
In this final example, we consider a contact problem involving a layered spherical object compressed between two half-elliptical beams.  This is an idealized problem meant to represent many common features of contact analysis, including the transmission of force through layered contact constraints, contact geometries featuring compound curvature, and contact constraints which induce local structural bifurcation. The geometry consists of four elastic bodies: two elongated half-elliptical beams, a hollow spherical shell, and a nested hollow oblate spheroidal shell. The beams are aligned with the z-axis, have elliptical bases and together they extend from \(z=-2.5\) to \(z=2.5\). At \(z=0\) they rest atop a hollow spherical shell, which is centered at the origin and has outer and inner radii of \(1\) and \(0.8\), respectively. Enclosed within this sphere is a hollow oblate spheroidal shell with outer semi-axes \(a = 0.775\), \(b = 0.8\) and \(c = 0.8\) (in the \(x\),\(y\) and \(z\) directions, respectively) and inner semi-axes \(a = 0.625\), \(b = 0.695\) and \(c = 0.68\).
\begin{figure}[H]
    \centering
    \begin{subfigure}[t]{0.32\textwidth}
        \centering
        \includegraphics[width=1.15\linewidth]{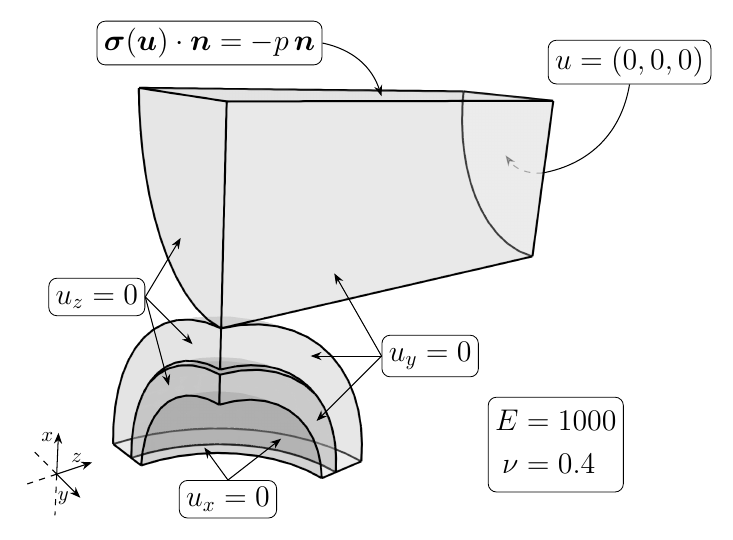}
        \caption{Initial configuration}
        \label{fig:test6_diagram}
    \end{subfigure}
    \begin{subfigure}[t]{0.32\textwidth}
        \centering
        \includegraphics[width=0.98\linewidth]{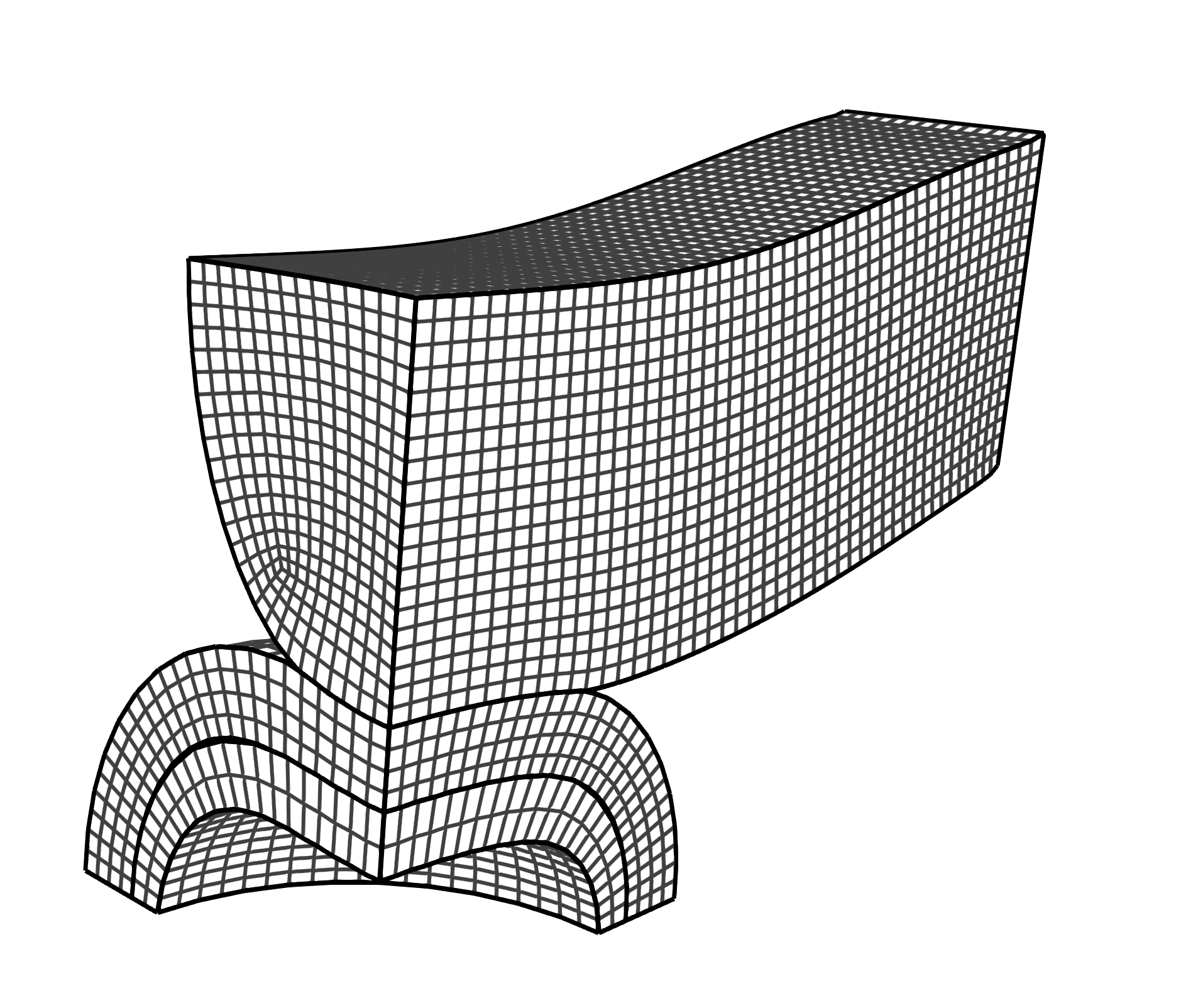}
        \caption{Linear model}
        \label{fig:test6lin_final_mesh}
    \end{subfigure}
    \begin{subfigure}[t]{0.32\textwidth}
    \centering
    \includegraphics[width=0.98\linewidth]{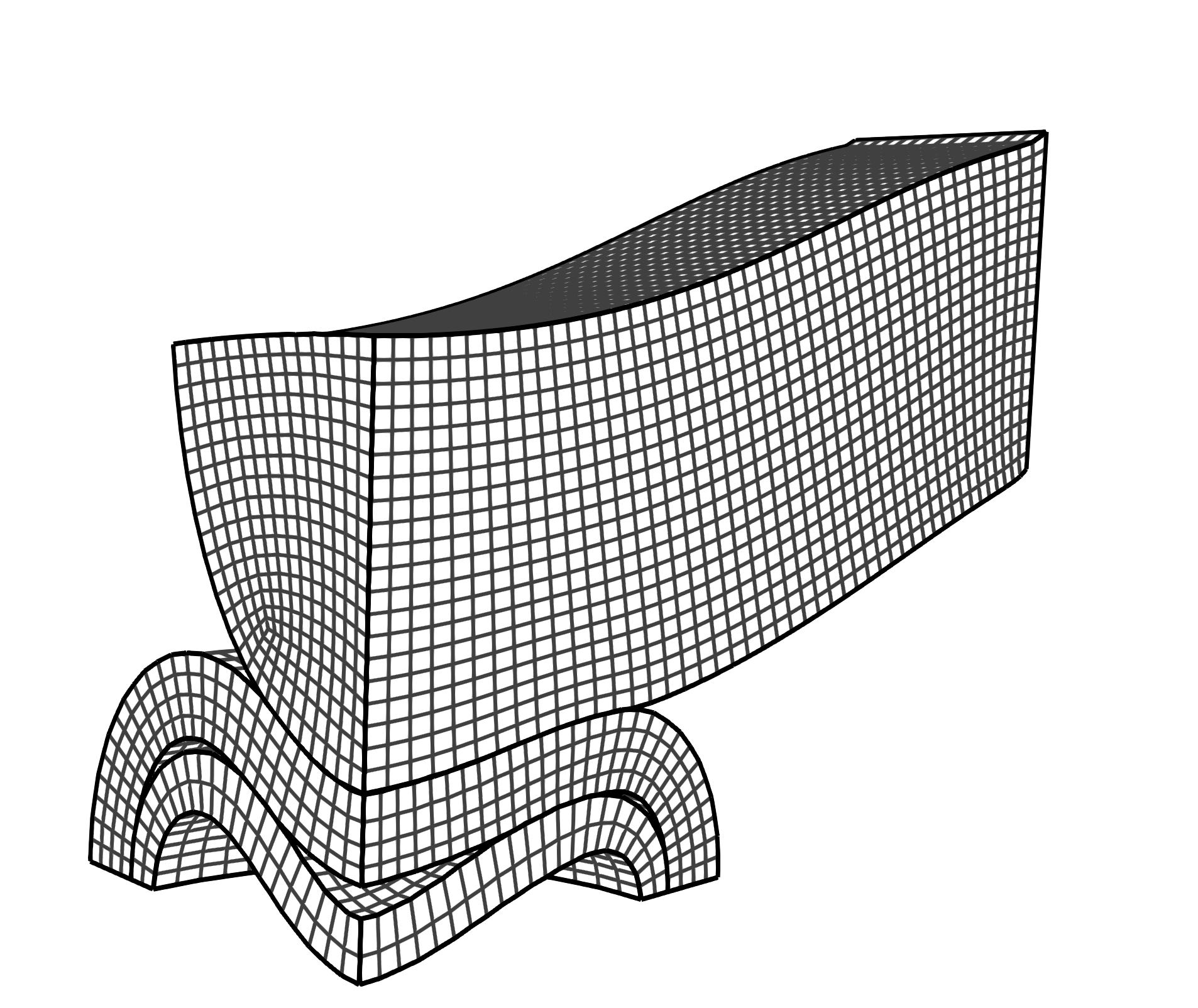}
    \caption{Nonlinear model}
    \label{fig:test6nl_final_mesh}
\end{subfigure}
\caption{Beam-sphere problem: An elongated beam rests on top of a hollow spherical shell. Enclosed within this sphere there is a hollow oblate spheroidal shell. The beam is fixed with a homogeneous Dirichlet BC on one end and symmetry conditions are applied at $x=0, y=0$ and $z=0$. A uniform compressive pressure of magnitude $30$ is applied at the top boundary surface of beam. The figure on the left shows the initial setup. The middle and right figures show the final configurations when using a linear and nonlinear material model, respectively. Note that only the nonlinear model is able to capture the material buckling.}
\end{figure}
In order to reduce computational cost, we exploit symmetry and simulate only the octant defined by \(x\ge0, y\le0, z\ge0\). To enforce symmetry, the following displacement BCs are applied: \(u_x=0\) at \(x = 0\), \(u_y = 0\) at \(y=0\) and \(u_z = 0\) at \(z = 0\). A uniform compressive pressure of magnitude \(p = 30\) is applied to the top surface of the beam by applying a Neumann BC \(\boldsymbol{\sigma}(\boldsymbol{u}) \cdot \boldsymbol{n} = - p \boldsymbol{n}\). The beam is fixed at \(z=2.5\) via a Dirichlet BC \(u=(0,0,0)\). Here, the bottom of the beam and the top of the inner shell are the mortar surfaces, while the top and bottom of the outer sphere act as the nonmortar surfaces.  The material parameters for all bodies correspond to a Young modulus \(E=1000\) and a Poisson's ratio \(\nu=0.4\). We test solver performance for both linear and nonlinear formulations of the problem. For the nonlinear formulation we use a compressible Neo-Hookean material which is already available in MFEM \cite{mfem1}. Note that in this example, buckling arises under compressive loading. The linear model fails to reproduce this behavior, while the nonlinear formulation is able to capture it (see \Cref{fig:test6lin_final_mesh,fig:test6nl_final_mesh}). 
For both formulations we use the same time-stepping process as in the previous examples, with six time steps. The Neumann BC is applied incrementally, i.e.,
\begin{equation*}
\boldsymbol{\sigma}(\boldsymbol{u}) \cdot \boldsymbol{n}|_{t_i} = - p \frac{i}{6} \, \boldsymbol{n}, \quad i=1,2, \dots, 6.
\end{equation*}
The displacement at each time step \(t_i\) generates an intermediate mesh used to update the gap function and its Jacobian. A new problem is then formulated on the original mesh, incorporating the next boundary values and the updated gap and Jacobian.
\begin{figure}[H]
\setlength{\abovecaptionskip}{-5pt} %
\setlength{\belowcaptionskip}{-15pt} %
    \centering
        \begin{subfigure}{0.295\textwidth}
        \centering
        \includegraphics[width=1.0\linewidth,trim=75 0 200 150, clip]{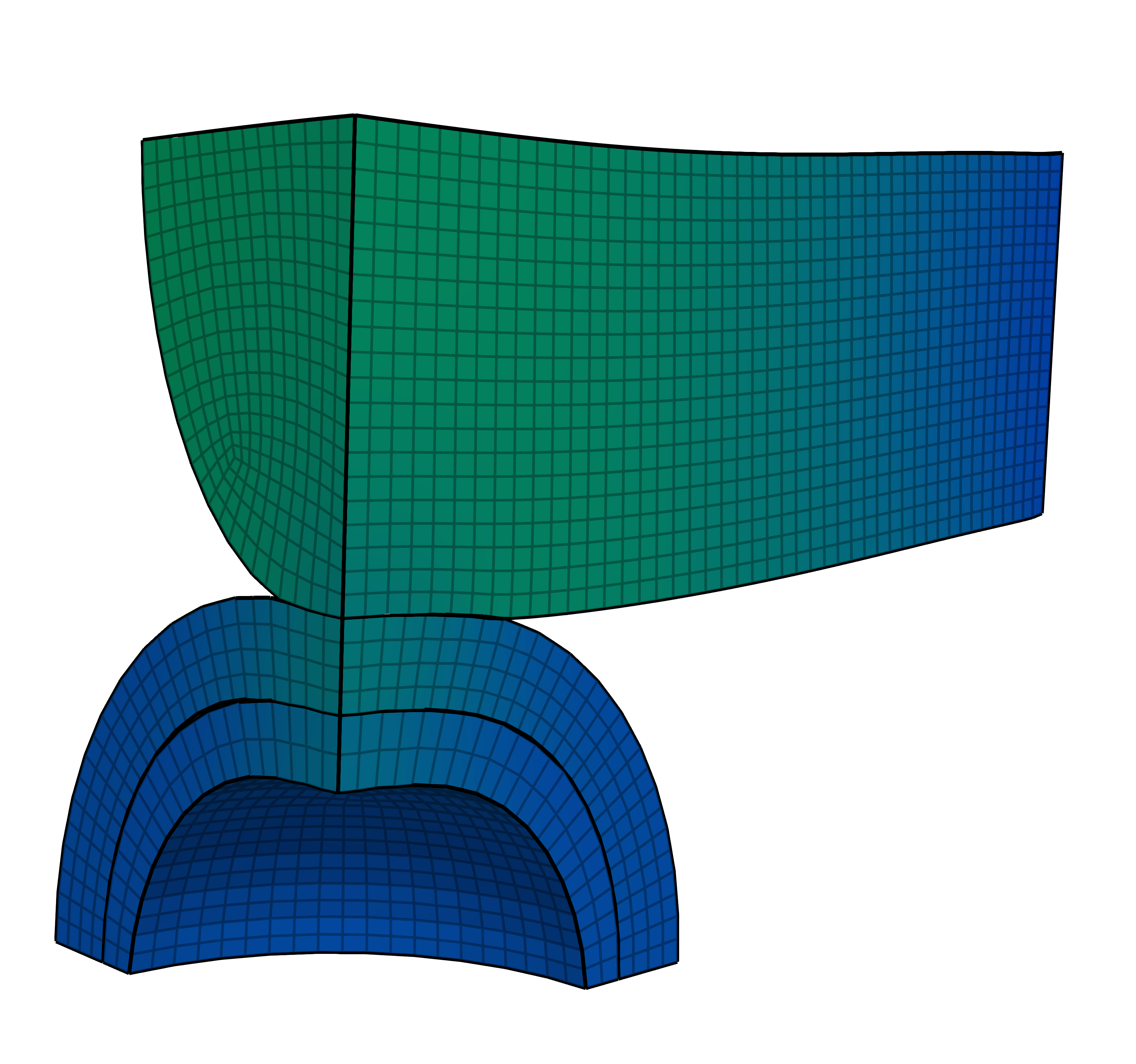}
    \end{subfigure}
        \begin{subfigure}{0.295\textwidth}
        \centering
        \includegraphics[width=1.0\linewidth,trim=75 0 200 150, clip]{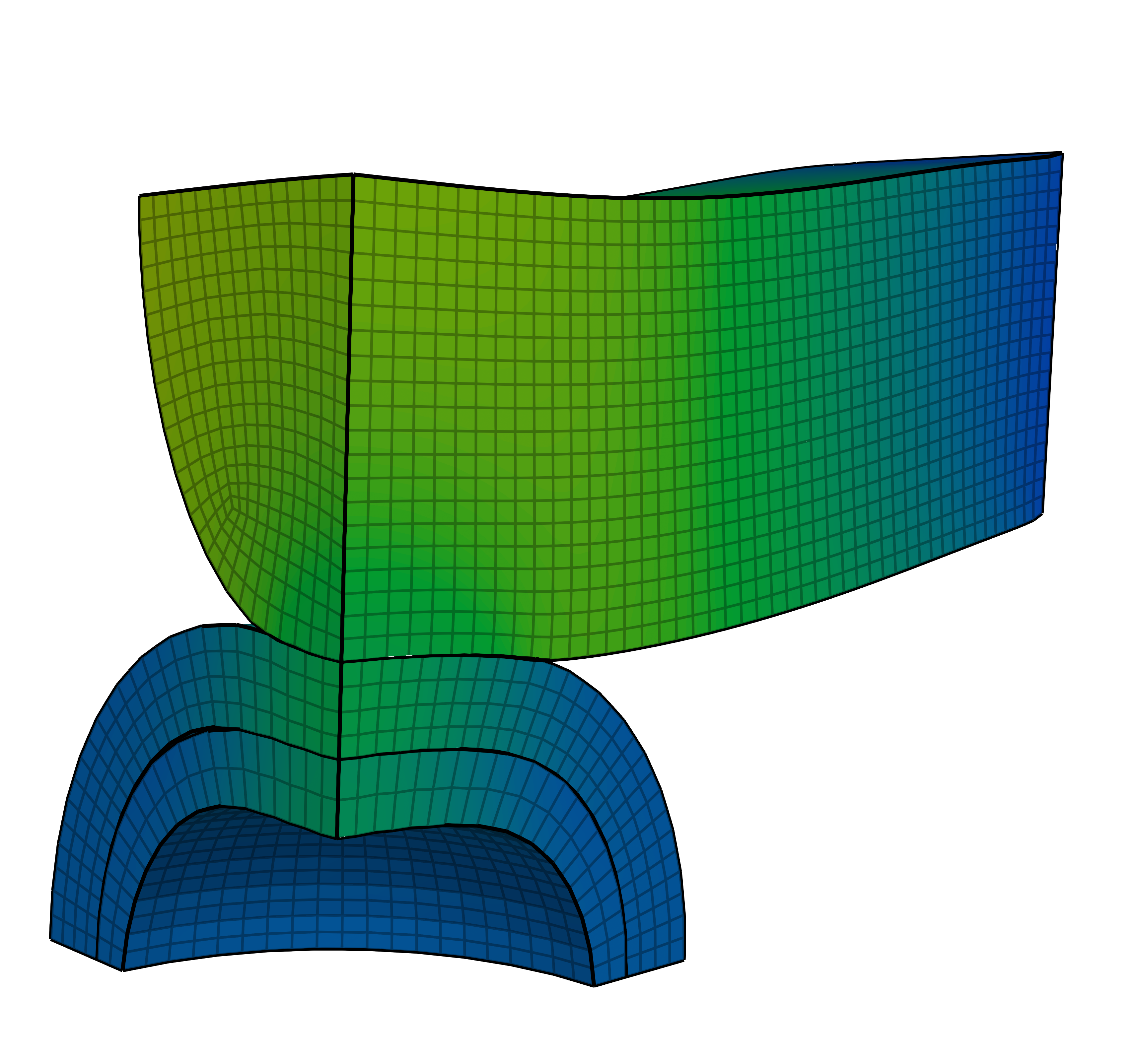}
    \end{subfigure}
    \begin{subfigure}{0.295\textwidth}
        \centering
        \includegraphics[width=1.0\linewidth,trim=75 0 200 150, clip]{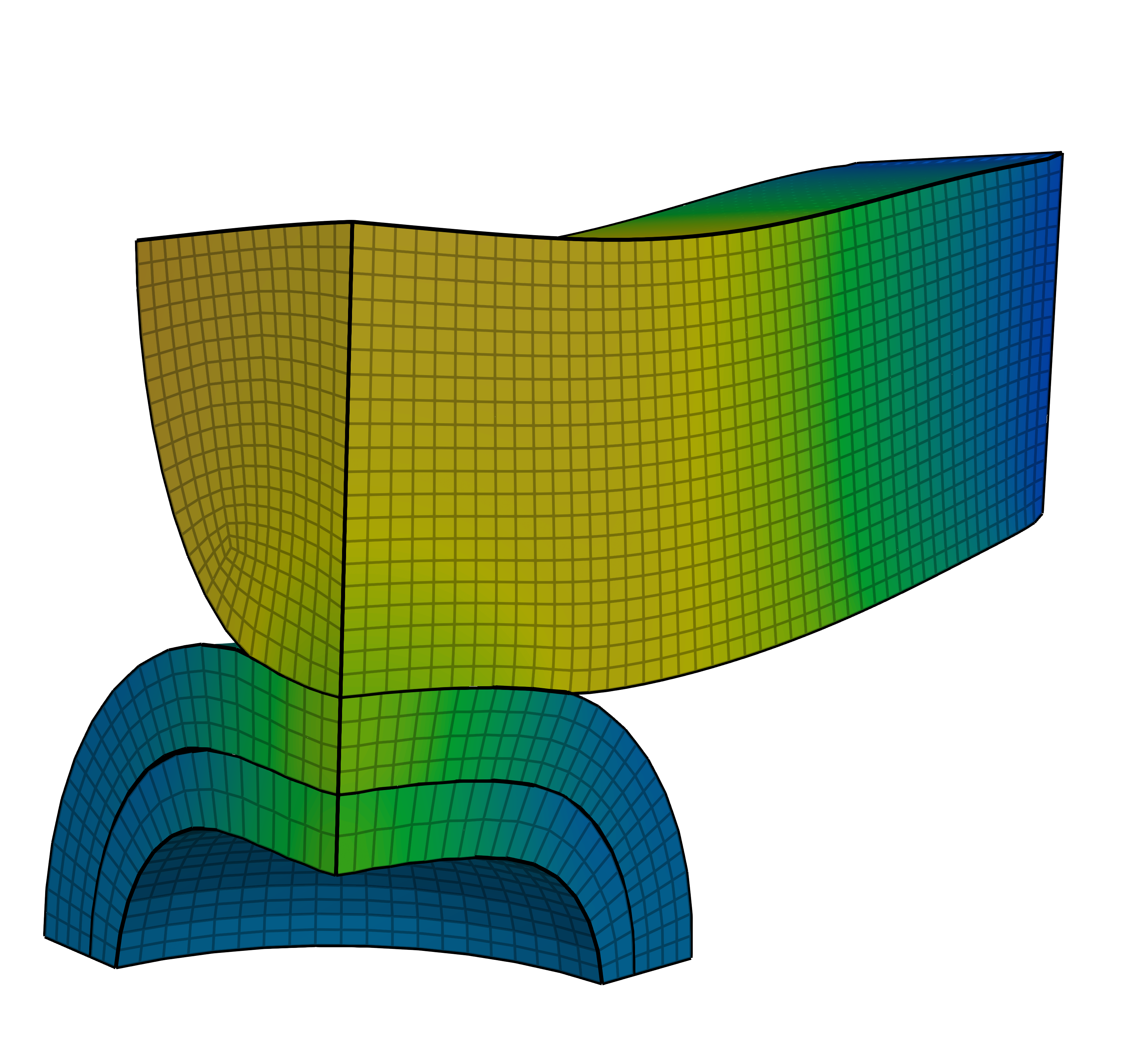}
    \end{subfigure}
    \begin{subfigure}{0.075\textwidth}
        \centering
        \includegraphics[width=1.5\linewidth,trim=600 0 250 300, clip]{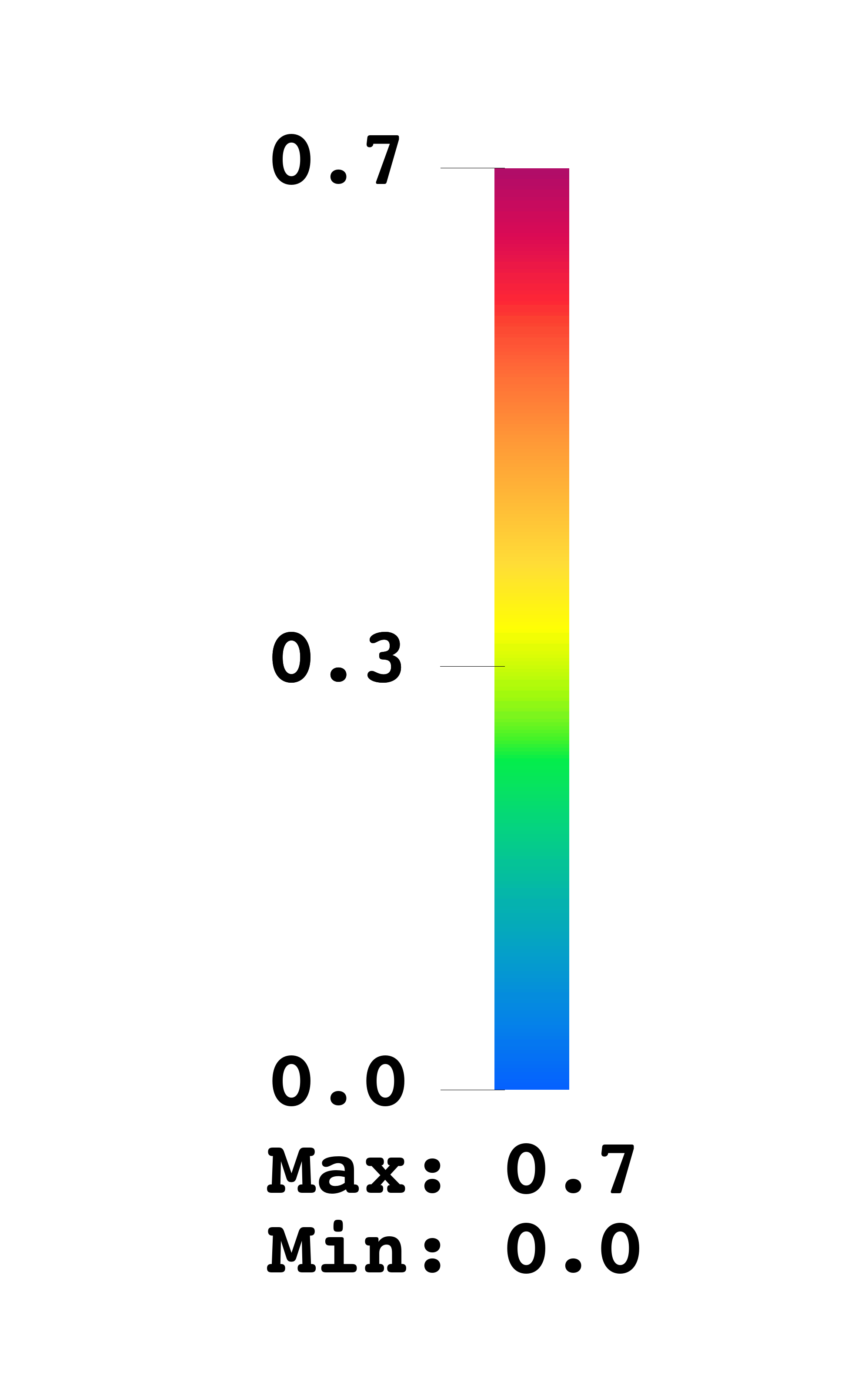}
    \end{subfigure}\hfill
    \caption{Linear model: deformed configurations and displacement magnitudes at time steps \(t_2,t_4,t_6\). Note that the linear model fails to capture material buckling.}
    \label{fig:ex6lin_deformation}
\end{figure}
\begin{figure}[H]
\setlength{\abovecaptionskip}{-2pt} %
    \centering
        \begin{subfigure}{0.295\textwidth}
        \centering
        \includegraphics[width=1.0\linewidth, trim=75 0 200 150, clip]{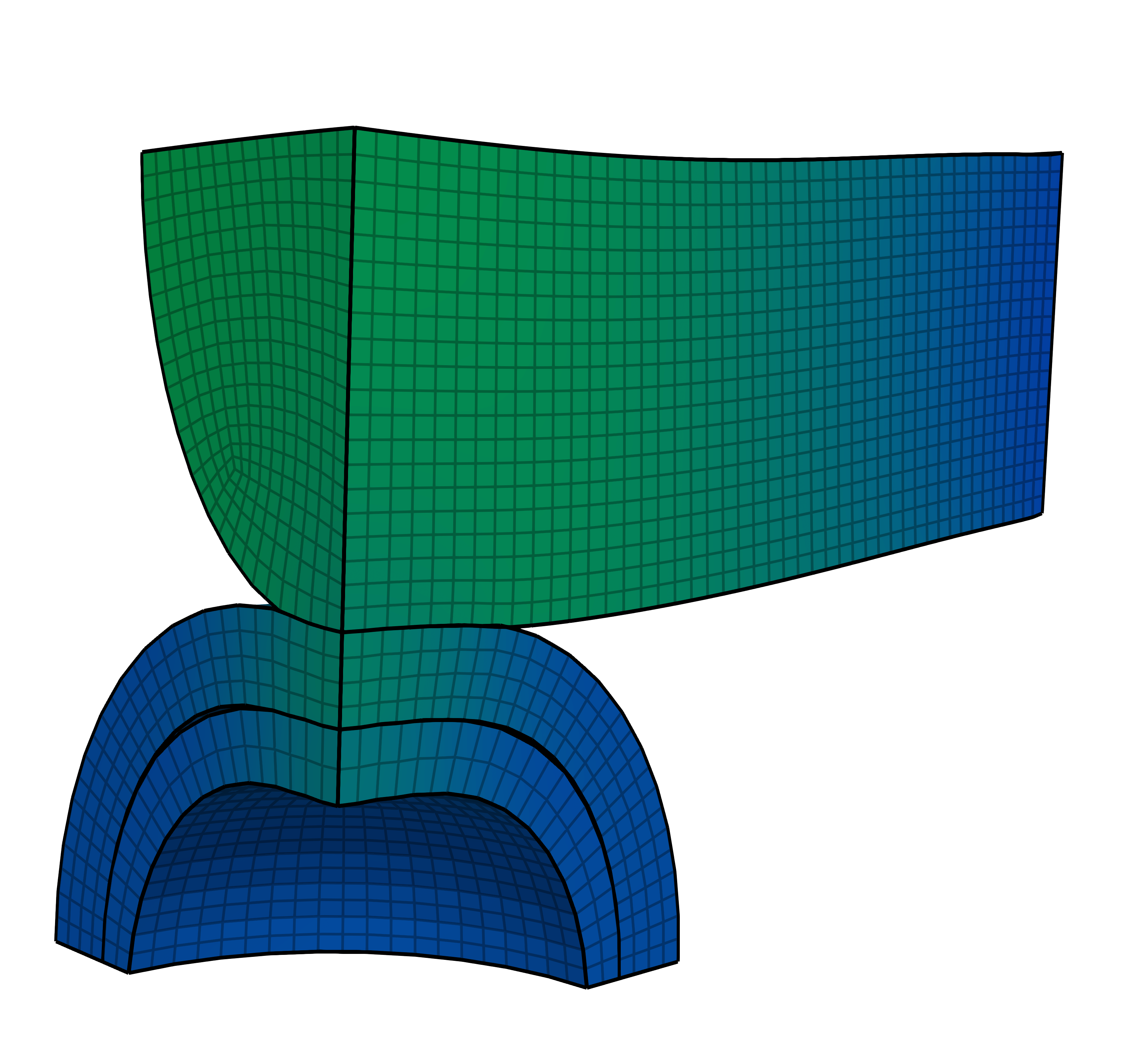}
    \end{subfigure}
        \begin{subfigure}{0.295\textwidth}
        \centering
        \includegraphics[width=1.0\linewidth, trim=75 0 200 150, clip]{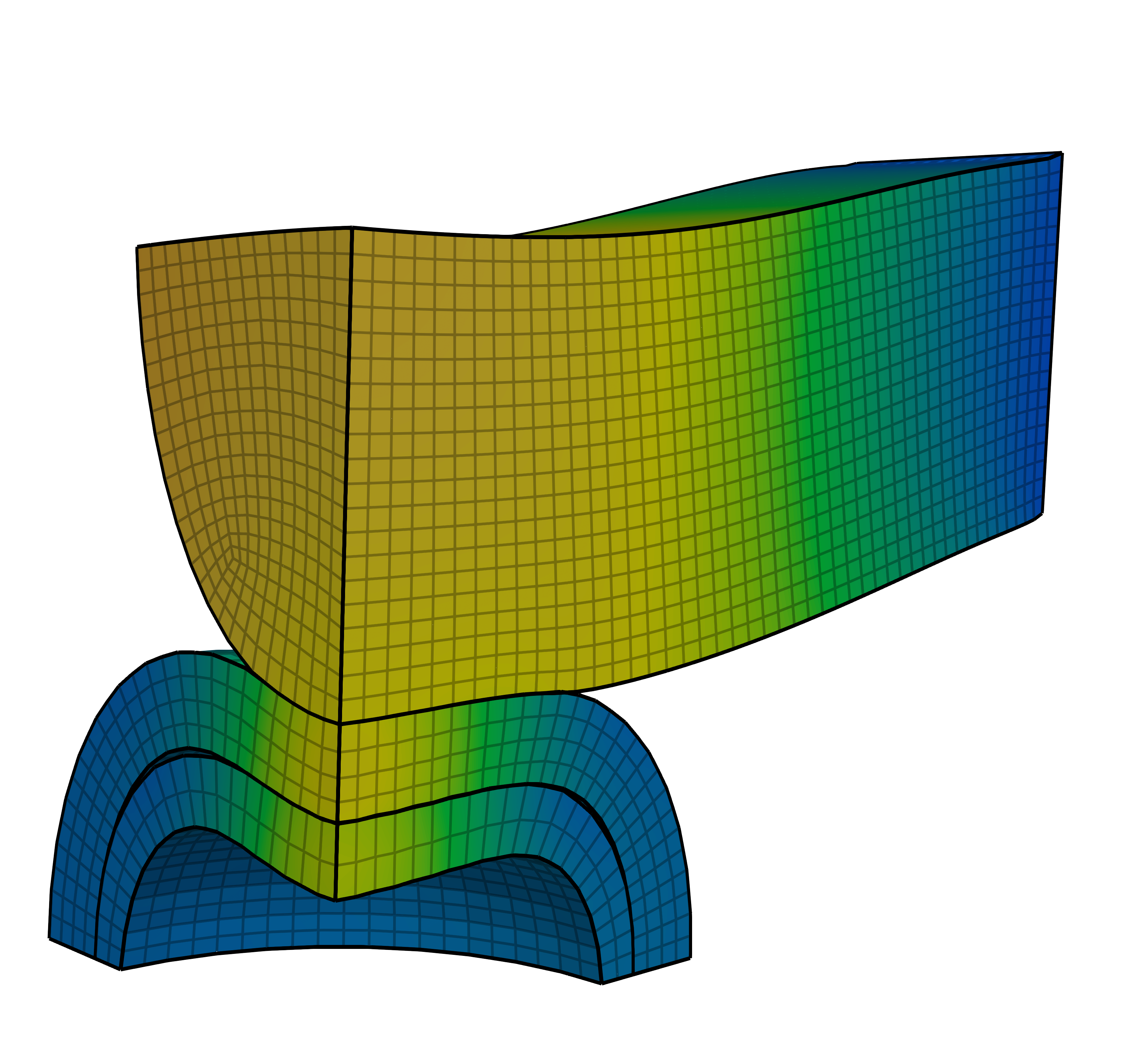}
    \end{subfigure}
    \begin{subfigure}{0.295\textwidth}
        \centering
        \includegraphics[width=1.0\linewidth, trim=75 0 200 150, clip]{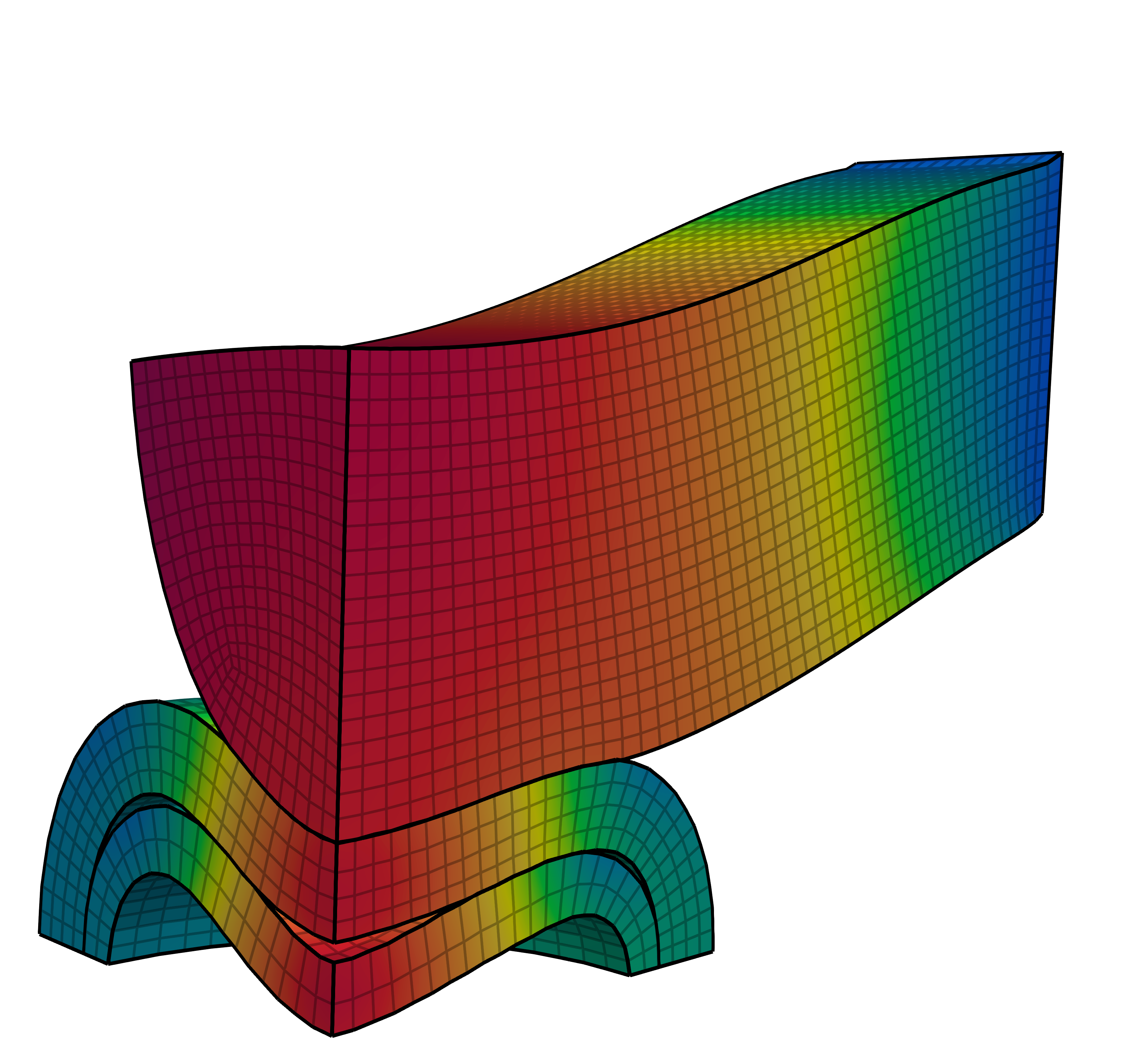}
    \end{subfigure}
    \begin{subfigure}{0.075\textwidth}
        \centering
        \includegraphics[width=1.5\linewidth, trim=600 0 250 300, clip]{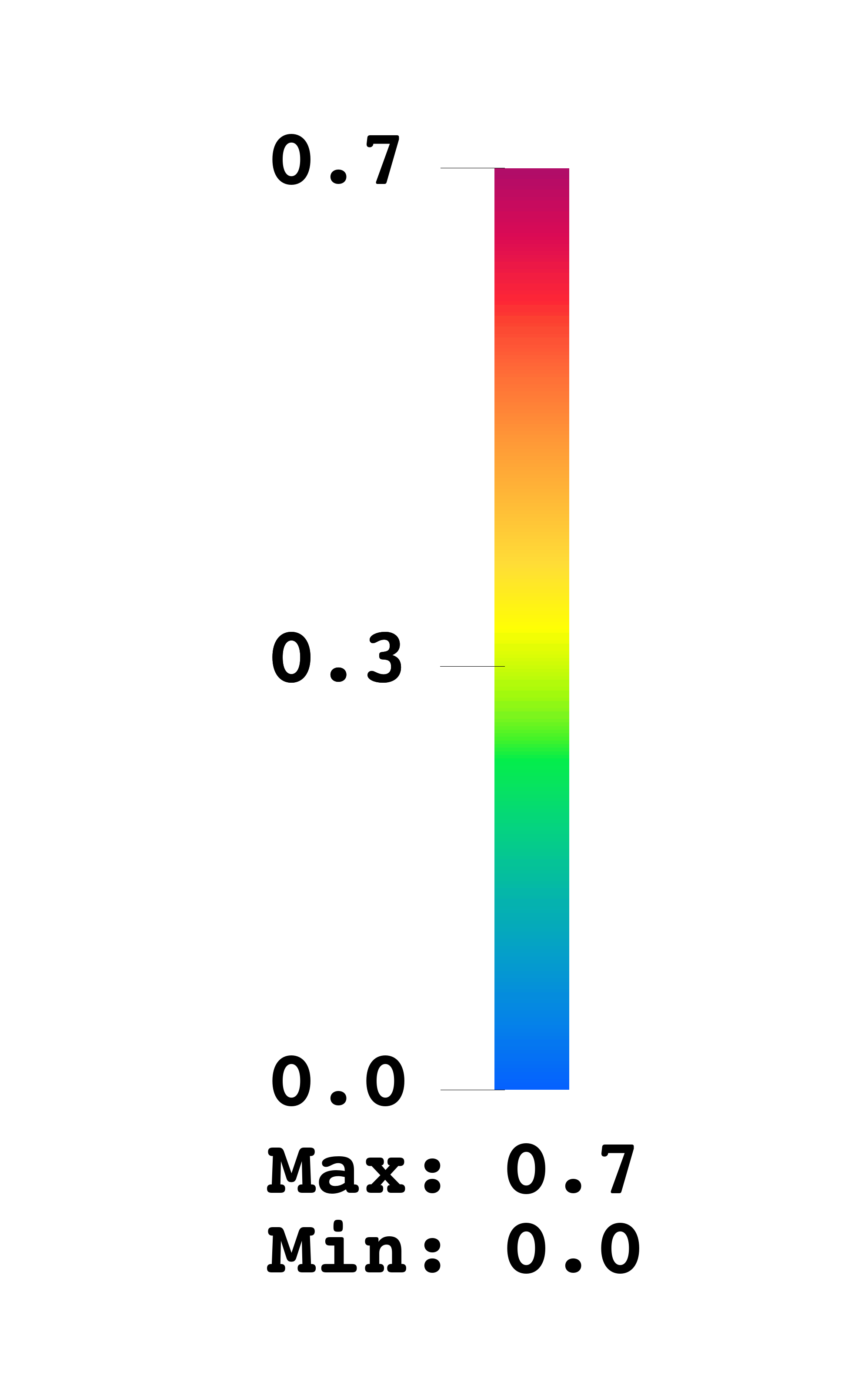}
    \end{subfigure}\hfill
    \caption{Nonlinear model: deformed configurations and displacement magnitudes at time steps \(t_2,t_4,t_6\). Note that the nonlinear model is able to capture material buckling.}
    \label{fig:ex6nl_deformation}
\end{figure}
In \Cref{fig:ex6lin_deformation,fig:ex6nl_deformation}, the computed displacement magnitudes and corresponding deformed configurations are shown for the linear and nonlinear models, respectively. \Cref{tab:problem_size_test6lin,tab:problem_size_test6nl} summarize the solver performance for both models. As in the previous examples, we report \(\tt n\), the problem size for each refinement level, as well as \(\tt n_c^{max}\) and \(\tt m^{max}\), the maximum number of DOFs in the contact subspace and the maximum number of contact constraints across all time steps, respectively. For each refinement level, we report \(\tt{k}_{IP}^{avg}\), the average IP iteration count over all time steps, and \(\tt{k}_{AMGF}^{avg}\) the average AMGF--PCG iterations over all time and optimization steps. For the linear model, values in parentheses indicate the corresponding upper bounds predicted by \cref{eq:practical_bound}. Note that in the nonlinear case, the system matrix changes over time, leading to varying condition numbers and no single representative bound.
\begin{table}[H]
    \centering
    \scalebox{1}{
    \begin{tabular}{lrrrrrrr}
        \hline
        \hline
                                           && Mesh 1 & Mesh 2 & Mesh 3  & Mesh 4 & Mesh 5 & Mesh 6 \\
        \hline
        \(\tt n\)                          && 7,902 & 53,547& 392,661 & 3,004,161 & 23,496,057  & 185,841,897 \\
        \(\tt n_c^{max}\)                  && 1,071 & 3,708 & 13,992  & 54,177    & 212,895     & 841,113 \\
        \(\tt m^{max}\)                    && 197   & 704   & 2,665    & 10,365    & 40,839      & 162,048    \\
        \hdashline
        \(\tt{k}_{IP}^{avg}\)              && 15    & 18    & 20      & 23        & 25          & 28  \\
        \(\tt{k}_{AMGF}^{avg}\)            && 15 (35)    & 19 (42)    & 24 (50)      & 30 (59)        & 41 (79)          & 50 (95)  \\
        \hline
    \end{tabular}}
\caption{{\bf Solver iteration counts for the linear beam-sphere contact problem across mesh refinements.} Starting from the initial coarse mesh (Mesh 1), each subsequent mesh is obtained by uniformly refining the previous one. Here \(\tt n\) denotes the total number of DOFs in the solution space \(\mathbb{U}\), while \(\tt n_c^{max}\) and \(\tt m^{max}\) represent the maximum dimension of the contact subspace \(\mathbb{W}\) and the maximum number of contact constraints encountered across all time steps, respectively. We also list \(\tt{k}_{IP}^{avg}\), the average number of IP iterations over all time steps. The last row reports \(\tt{k}_{AMGF}^{avg}\), the average number of iterations required by AMGF--PCG solver across all optimization and time steps to solve the contact problem. The numbers in parentheses indicate the corresponding computed AMGF--PCG iteration count upper bounds derived by the condition number estimate \cref{eq:practical_bound}.}
    \label{tab:problem_size_test6lin}
\end{table}
\begin{table}[H]
    \centering
    \scalebox{1}{
    \begin{tabular}{lrrrrrrr}
        \hline
        \hline
                                            && Mesh 1 & Mesh 2 & Mesh 3  & Mesh 4 & Mesh 5 & Mesh 6 \\
        \hline
        \(\tt n\)                           && 7,902 & 53,547& 392,661 & 3,004,161 & 23,496,057 & 185,841,897 \\
        \(\tt n_c^{max}\)                   && 1,128      & 3,873     & 14,715  & 56,970    & 223,764    & 883,725 \\
         \(\tt m^{max}\)                    && 210      & 737      & 2,809   & 10,911    & 42,991   & 170,539  \\
        \hdashline
    \(\tt{k}_{IP}^{avg}\)                   && 17     & 19      & 21      & 24        & 27           & 31  \\
    \(\tt{k}_{AMGF}^{avg}\)                 && 17     & 20      & 24      & 31        & 42           & 51  \\
        \hline
    \end{tabular}}
\caption{{\bf Solver iteration counts for the nonlinear beam-sphere problem across mesh refinement levels.} Starting from the initial coarse mesh (Mesh 1), each subsequent mesh is obtained by uniformly refining the previous one. Here \(\tt n\) denotes the total number of DOFs in the solution space \(\mathbb{U}\), while \(\tt n_c^{max}\) and \(\tt m^{max}\) represent the maximum dimension of the contact subspace \(\mathbb{W}\) and the maximum number of contact constraints encountered across all time steps, respectively. We also list \(\tt{k}_{IP}^{avg}\), the average number of IP iterations over all time steps. The last row reports \(\tt{k}_{AMGF}^{avg}\), the average AMGF--PCG iteration count across all optimization and time steps.}
    \label{tab:problem_size_test6nl}
\end{table}
\Cref{fig:test6lin_results,fig:test6nl_results} present the AMGF--PCG iteration counts for the linear and nonlinear beam-sphere problems, respectively, across all mesh refinement levels. As in the previous examples, iteration counts grow moderately with refinement and as the optimization progresses, but remain stable overall. In the linear case, iteration counts remain well below the theoretical upper bound \cref{eq:practical_bound}. For the nonlinear problem, no global upper bound is shown due to the evolving system matrix, but iteration counts remain consistent with expectations and exhibit no signs of instability. Notably, during buckling events, the elastic energy's Hessian may lose positive definiteness, making the objective in \cref{eq:linelastmin_discrete} unbounded. To ensure a well-posed problem, we impose bound constraints on displacement, 
$-\delta\leq \sf{u}-\sf{u}_{\star}\leq \delta$, for $\delta\in\mathbb{R}^{n}_{>0}$ (see \Cref{remark:boundconstr}). These bounds constraints are enabled after the fourth time step and are chosen based on displacements observed in earlier steps. These results further demonstrate the robustness of AMGF in both linear and nonlinear settings effectively handling contact constraints and large deformations while maintaining bounded iteration growth.
\begin{figure}[H]
    \centering
    \begin{subfigure}{0.32\textwidth}
        \centering
        \includegraphics[width=0.99\linewidth]{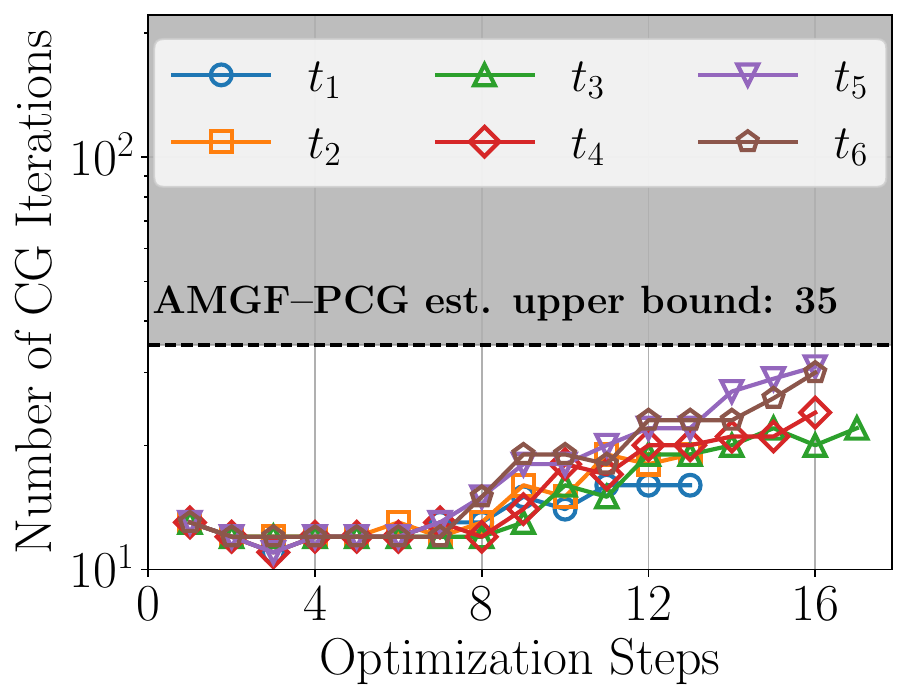}
        \caption{Mesh 1: 7,902 DOFs}
        \label{fig:testNo6lin_ref0}
    \end{subfigure}
    \begin{subfigure}{0.32\textwidth}
        \centering
        \includegraphics[width=0.99\linewidth]{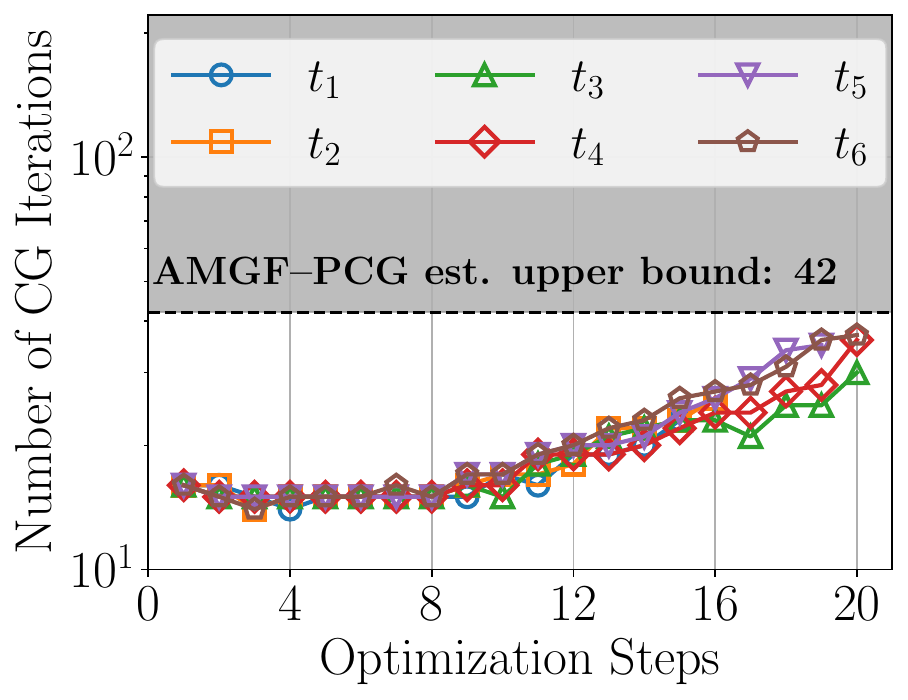}
        \caption{Mesh 2: 53,547 DOFs}
        \label{fig:testNo6lin_ref1}
    \end{subfigure}
        \begin{subfigure}{0.32\textwidth}
        \centering
        \includegraphics[width=0.99\linewidth]{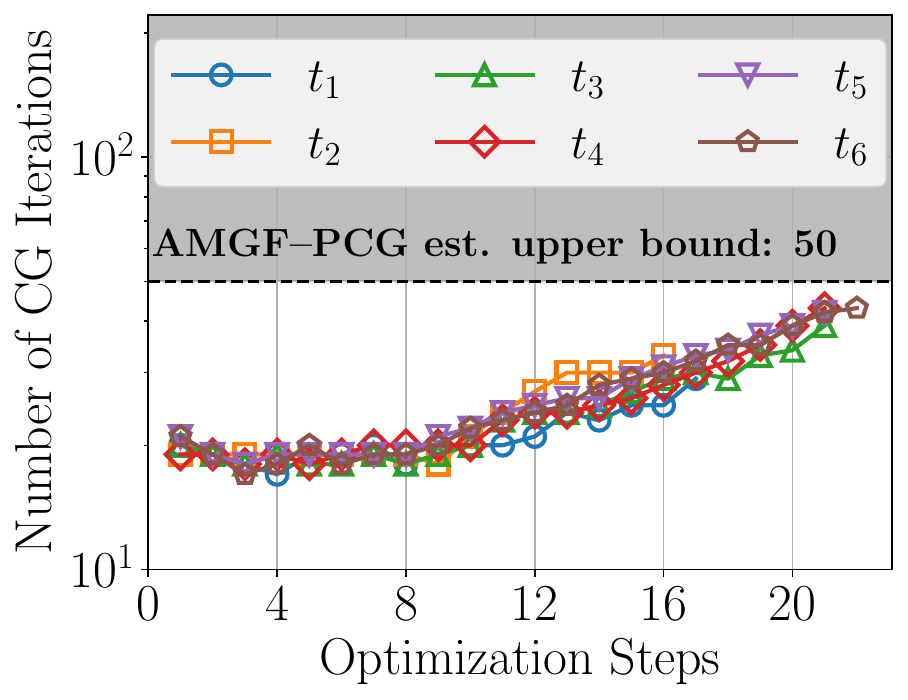}
        \caption{Mesh 3: 392,661 DOFs}
        \label{fig:testNo6lin_ref2}
    \end{subfigure}

    \begin{subfigure}{0.32\textwidth}
        \centering
        \includegraphics[width=0.99\linewidth]{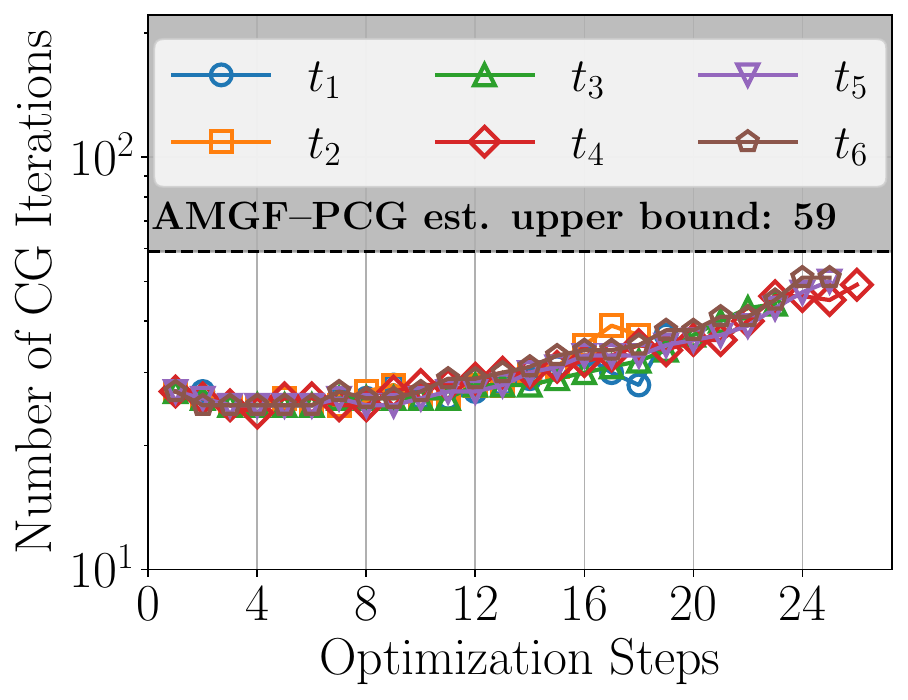}
        \caption{Mesh 4: 3,004,161 DOFs}
        \label{fig:testNo6lin_ref3}
    \end{subfigure}
    \begin{subfigure}{0.32\textwidth}
        \centering
        \includegraphics[width=0.99\linewidth]{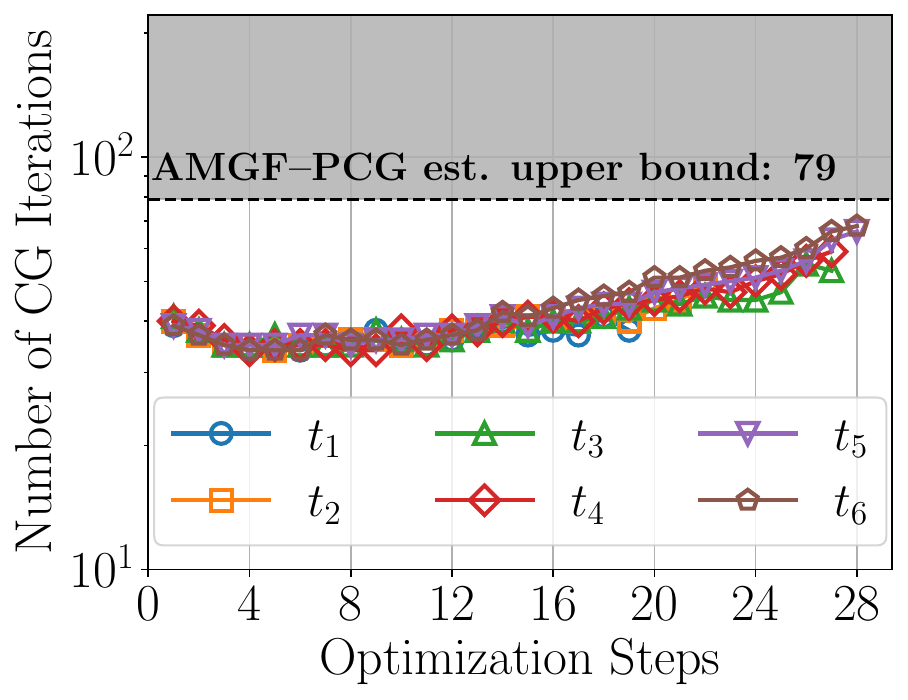}
        \caption{Mesh 5: 23,496,057 DOFs}
        \label{fig:testNo6lin_ref4}
    \end{subfigure}
        \begin{subfigure}{0.32\textwidth}
        \centering
        \includegraphics[width=0.99\linewidth]{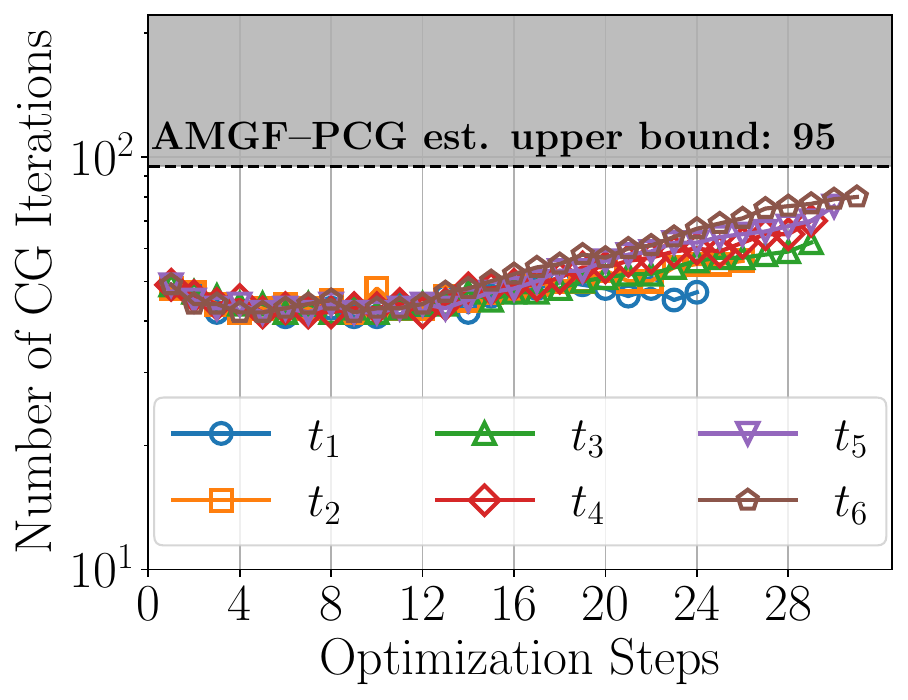}
        \caption{Mesh 6: 185,841,897 DOFs}
        \label{fig:testNo6lin_ref5}
    \end{subfigure}
    \caption{AMGF--PCG convergence for the linear beam-sphere problem. Each curve represents the AMGF--PCG iteration count through the IP optimization method for time steps $t_i$, $i=1,\dots,6$. The horizontal line indicates the estimate of the upper bound derived from the theoretical condition number result \cref{eq:practical_bound}, i.e., the number of iterations of the AMGF--PCG solver when applied to the contact problem is bounded by approximately \(\sqrt{2}\) times the AMG--PCG solver count when applied to the contact-free problem.}
    \label{fig:test6lin_results}
\end{figure}
\begin{figure}[H]
\setlength{\abovecaptionskip}{-1pt} %
    \centering
    \begin{subfigure}{0.32\textwidth}
        \centering
        \includegraphics[width=0.99\linewidth]{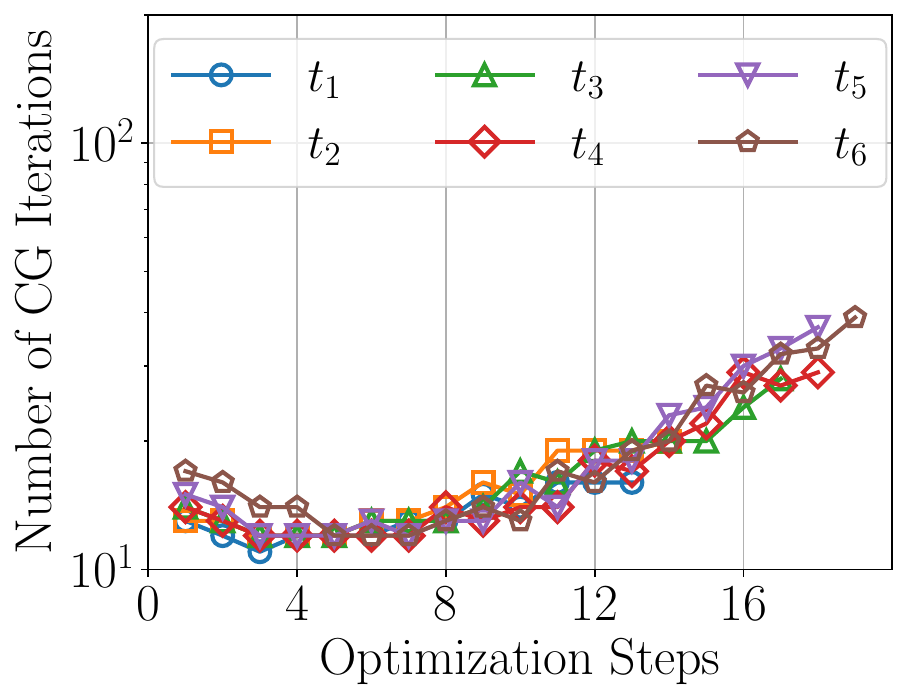}
        \caption{Mesh 1: 7,902 DOFs}
        \label{fig:testNo6nl_ref0}
    \end{subfigure}
    \begin{subfigure}{0.32\textwidth}
        \centering
        \includegraphics[width=0.99\linewidth]{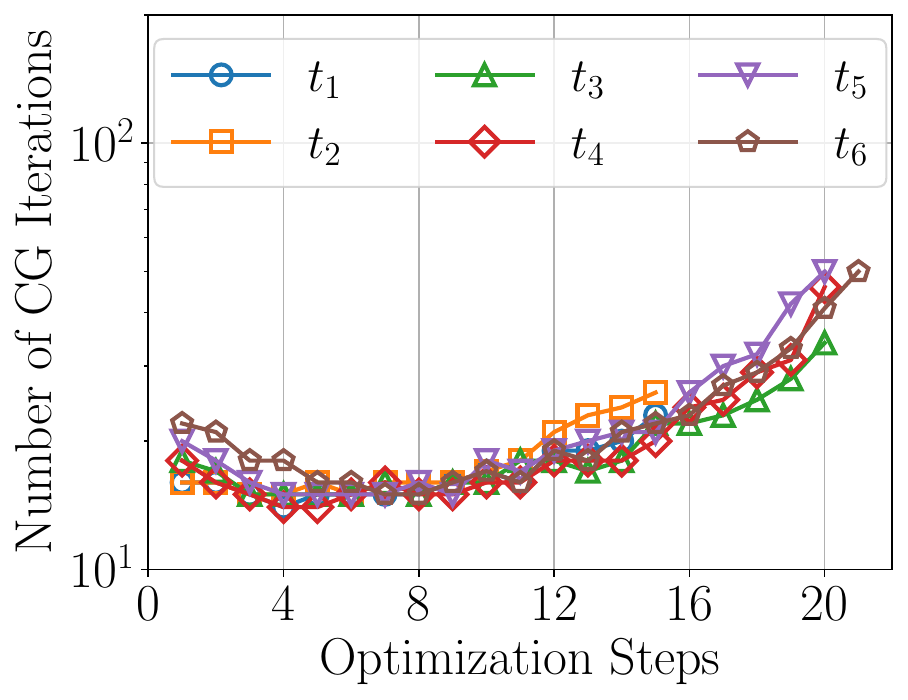}
        \caption{Mesh 2: 53,547 DOFs}
        \label{fig:testNo6nl_ref1}
    \end{subfigure}
        \begin{subfigure}{0.32\textwidth}
        \centering
        \includegraphics[width=0.99\linewidth]{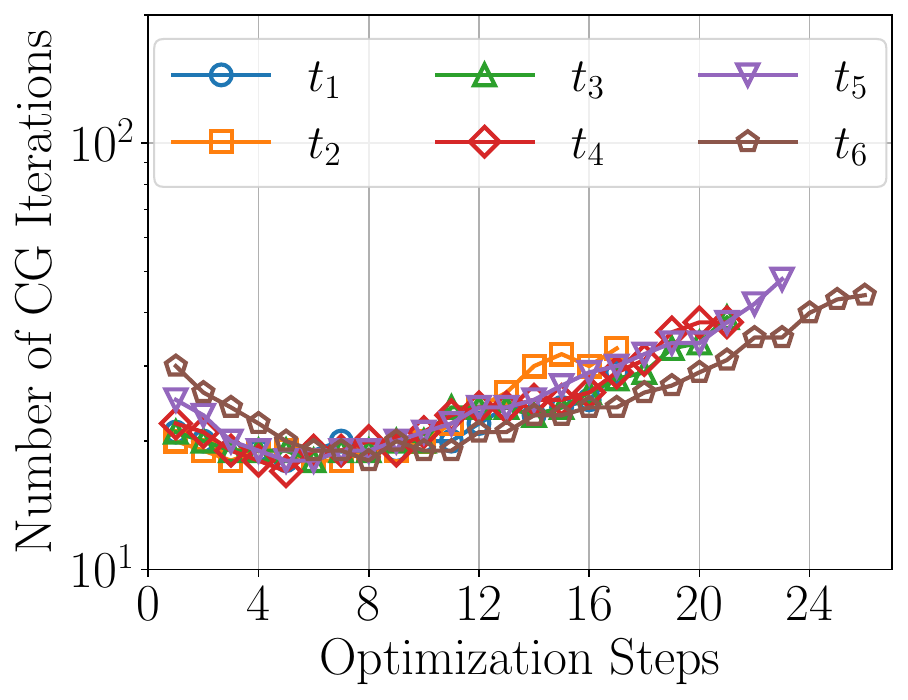}
        \caption{Mesh 3: 392,661 DOFs}
        \label{fig:testNo6nl_ref2}
    \end{subfigure}

    \begin{subfigure}{0.32\textwidth}
        \centering
        \includegraphics[width=0.99\linewidth]{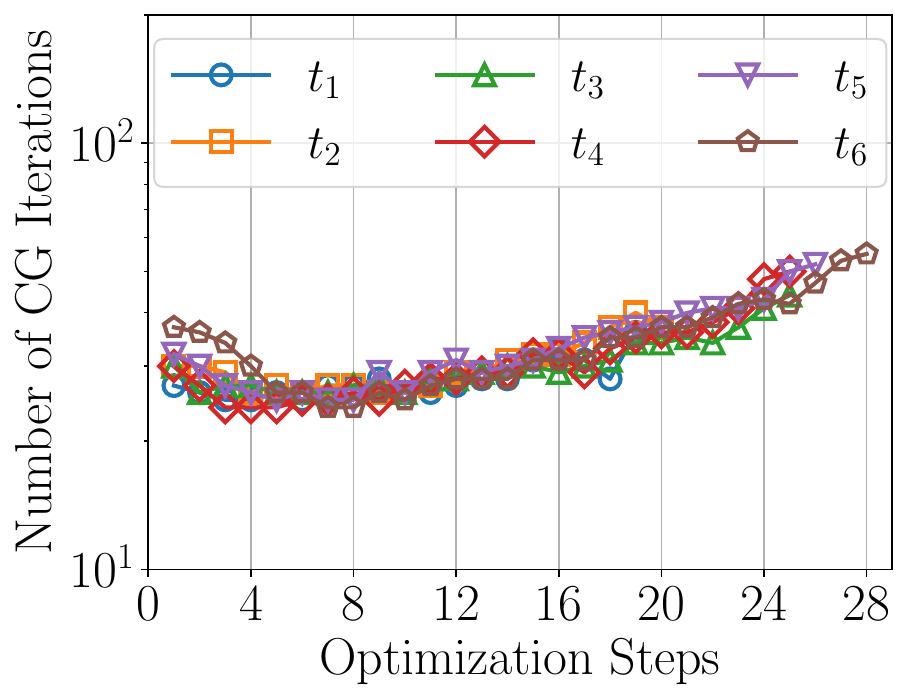}
        \caption{Mesh 4: 3,004,161 DOFs}
        \label{fig:testNo6nl_ref3}
    \end{subfigure}
    \begin{subfigure}{0.32\textwidth}
        \centering
        \includegraphics[width=0.99\linewidth]{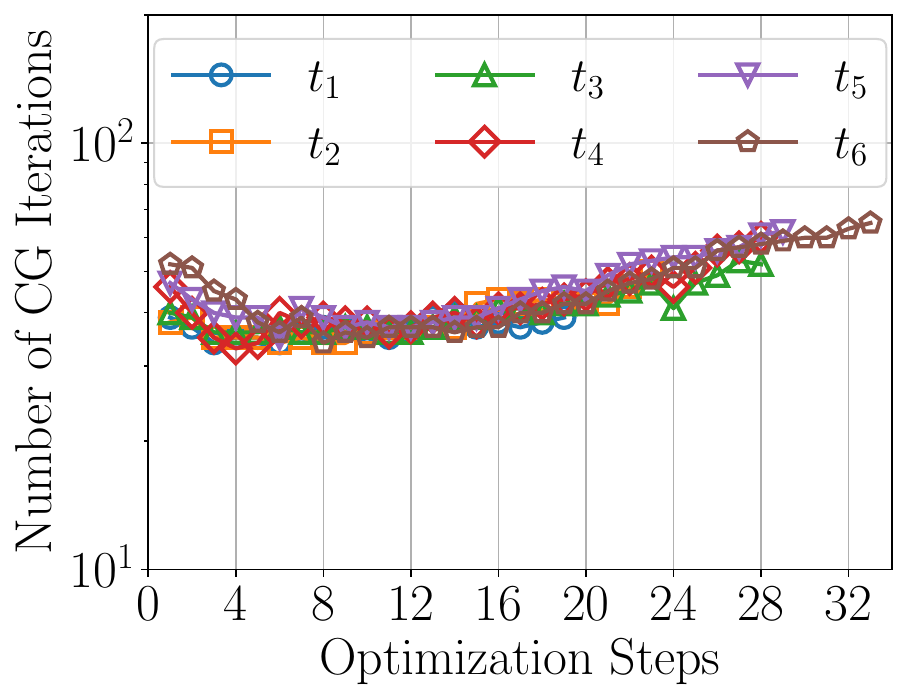}
        \caption{Mesh 5: 23,496,057 DOFs}
        \label{fig:testNo6nl_ref4}
    \end{subfigure}
        \begin{subfigure}{0.32\textwidth}
        \centering
        \includegraphics[width=0.99\linewidth]{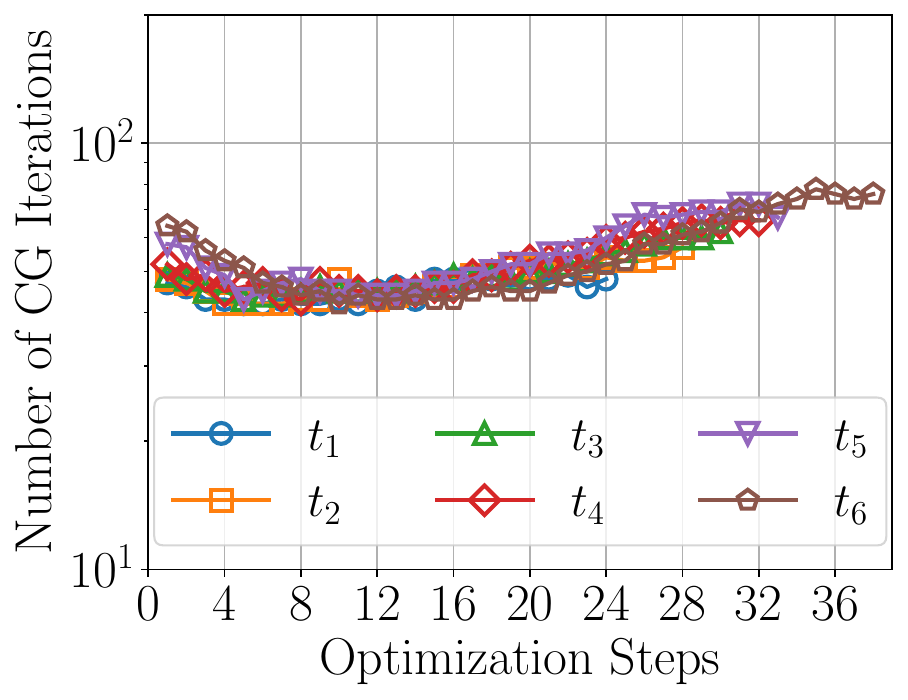}
        \caption{Mesh 6: 185,841,897 DOFs}
        \label{fig:testNo6nl_ref5}
    \end{subfigure}
    \caption{AMGF--PCG convergence for the nonlinear beam-sphere problem. Each curve represents the AMGF--PCG iteration count through the IP optimization method for time steps $t_i$, $i=1,\dots,6$.}
    \label{fig:test6nl_results}
\end{figure}

%% file: cost.tex
{\color{revised}
\section{Computational cost}
\label{section:cost}

To assess efficiency, we examine the computational cost of AMGF relative to AMG on the two-block problem described in \Cref{section:Model_problem}. \Cref{fig:speedup} reports the speedup in total simulation time for the first time step (\(t_1\)) when AMGF is used over AMG as a preconditioner to the linear system at every IP step. The bar chart shows the results for all meshes presented in the previous section, with the corresponding number of MPI processors indicated in parentheses. We also show the speedups restricted to the last 10 optimization steps. In all cases AMGF achieves substantial savings, ranging from factors 4\(\times\) to more than 10\(\times\). These improvements arise from the reduction of PCG iterations as illustrated in \Cref{fig:test4_results}. Note that since PCG--AMG iterations are capped at 5000, the reported gains are conservative and would be even larger without this limit. 

The benefits, however, are not uniform across all steps of the IP solver. In the early IP iterations, AMG converges rapidly and although AMGF often requires fewer iterations, the extra filtering step can outweigh this advantage, so AMGF may even be slightly slower. The improvements become significant in the later stages, when contact constraints are enforced and PCG--AMG performance deteriorates. This explains the larger speedups in the last 10 steps and suggests that a dynamic switching strategy, using AMG in the early IP iterations and AMGF once constraints dominate, could yield further improvements. 
\begin{figure}[H]
        \centering
        \includegraphics[width=0.8\linewidth]{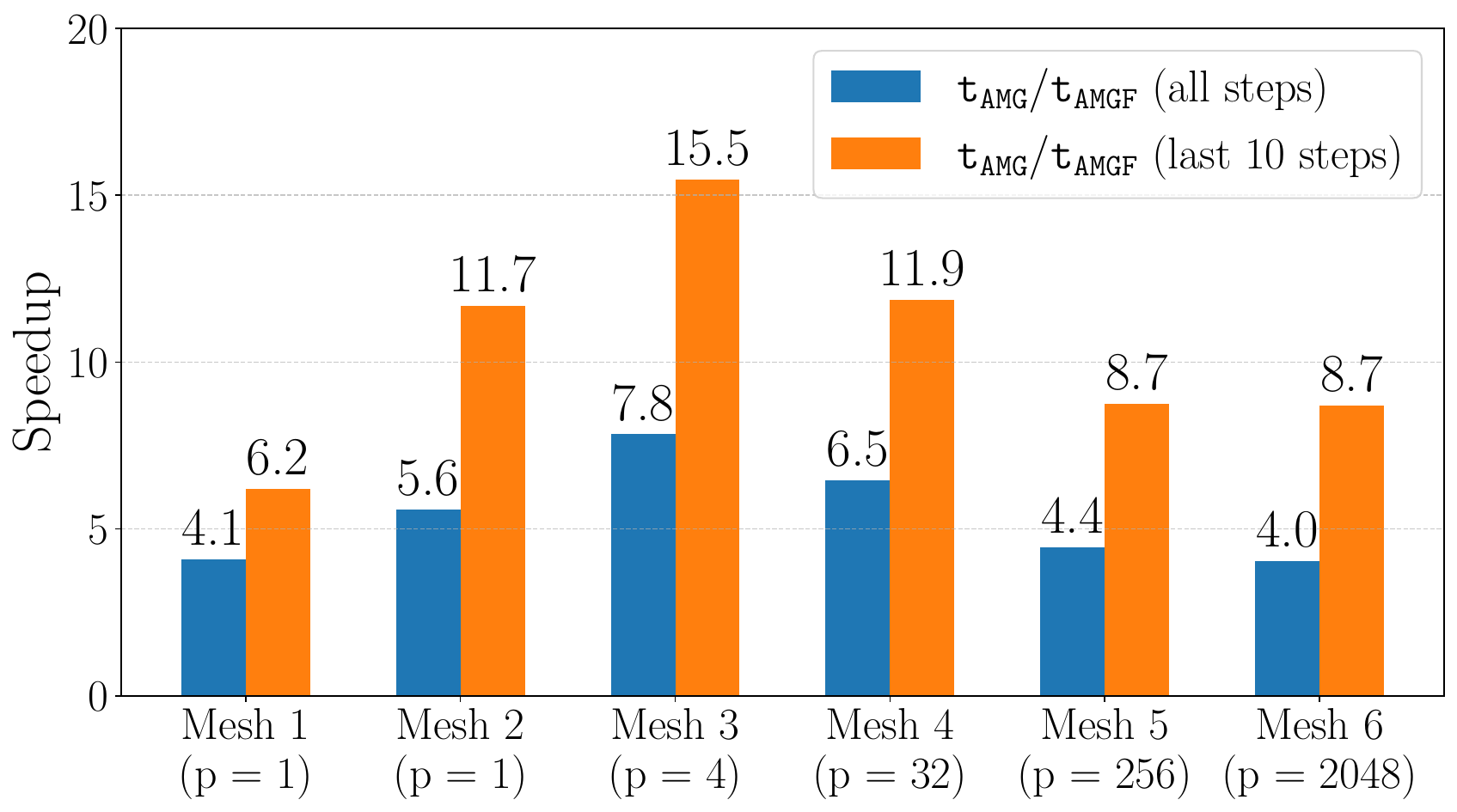}
        \caption{\color{revised}{\bf Time speedup of AMGF compared to AMG for the two-block problem.} Results are shown for all meshes reported in the previous section, with the corresponding number of MPI processors indicated in parentheses. The blue bars represent the ratio of AMG to AMGF total runtime over all optimization steps, while the orange bars correspond to the runtime over the last 10 steps of the IP solver. AMGF consistently achieves substantial savings, from 4\(\times\) to more than 10\(\times\). The higher gains in the last 10 IP steps reflect the deterioration of AMG performance once contact constraints are enforced.}
        \label{fig:speedup}
\end{figure}
\Cref{fig:amgf_pies} shows the runtime breakdown of AMGF into AMG setup and application, filter solver on the contact subspace setup and application, and the remaining CG operations for the three largest  meshes. The AMG setup and application together remain the dominant cost (\(\approx 90\%\)), while the filter solver on the contact subspace contributes only a few percent. Thus, the filtered subspace solver introduces only a minor overhead while enabling the iteration reductions that drive the observed speedups.  Recall that for a 3D problem the total number of unknowns \(\tt n\) scales as \(h^{-3}\) whereas the filtered subspace dimension, \(\tt n_c\) scales as \(h^{-2}\), so the ratio \(\tt n_c /  n \rightarrow 0\) under mesh refinement. A multifrontal direct solver on a 2D subspace of size \(\tt n_c\), has factorization cost \(\mathcal{O}({\tt n_c}^{\frac{3}{2}})\) and memory requirements of \(\mathcal{O}(\tt n_c \log n_c)\). Since \(\tt n_c \ll n\), these costs remain modest in practice and do not alter AMG's near-linear complexity.
\begin{figure}[H]
    \begin{subfigure}{0.32\textwidth}
        \centering
        \includegraphics[width=1.05\linewidth]{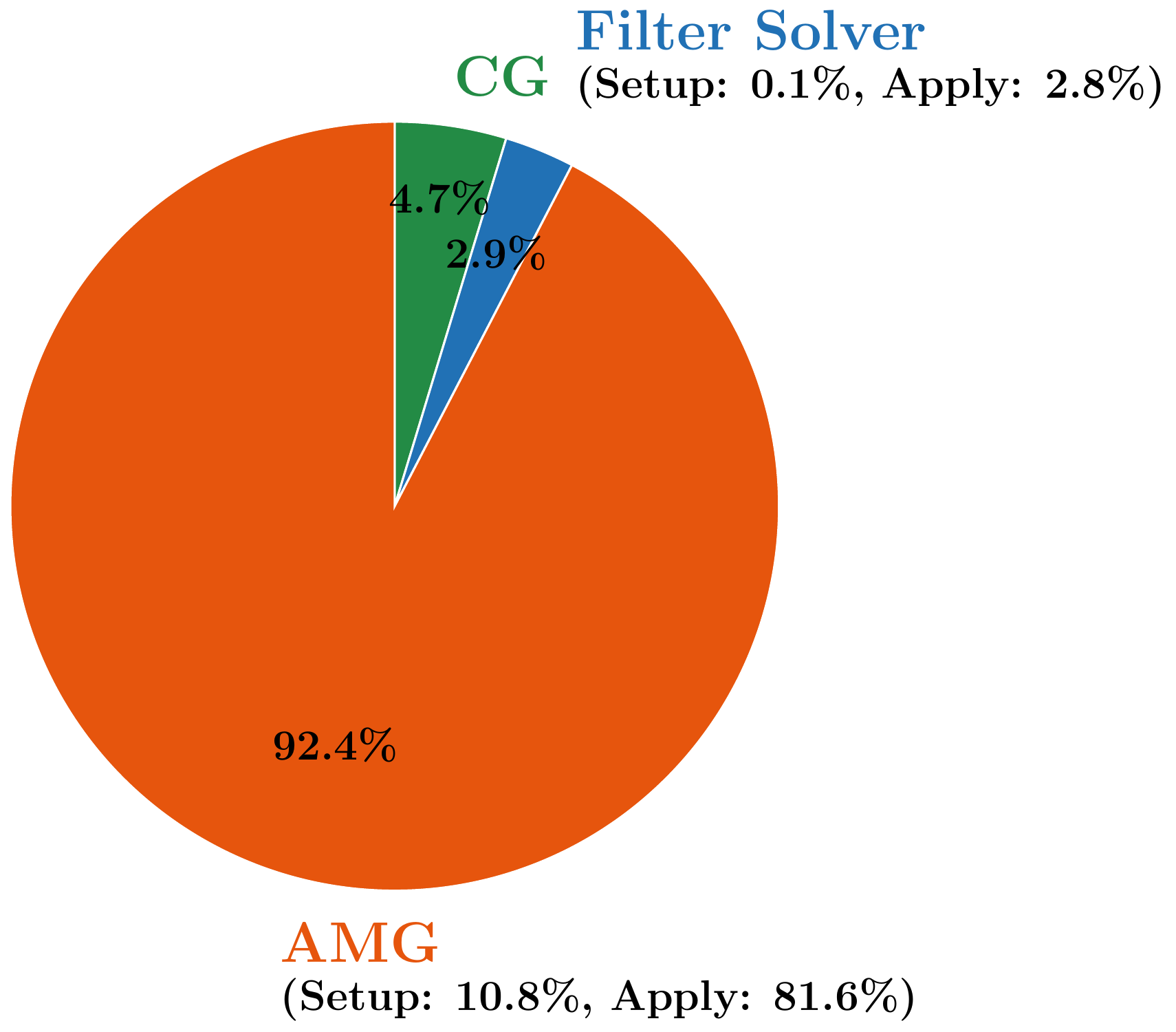}
        \caption{Mesh 4}
        \label{fig:pie_mesh4}
    \end{subfigure}
    \begin{subfigure}{0.32\textwidth}
        \centering
        \includegraphics[width=1.05\linewidth]{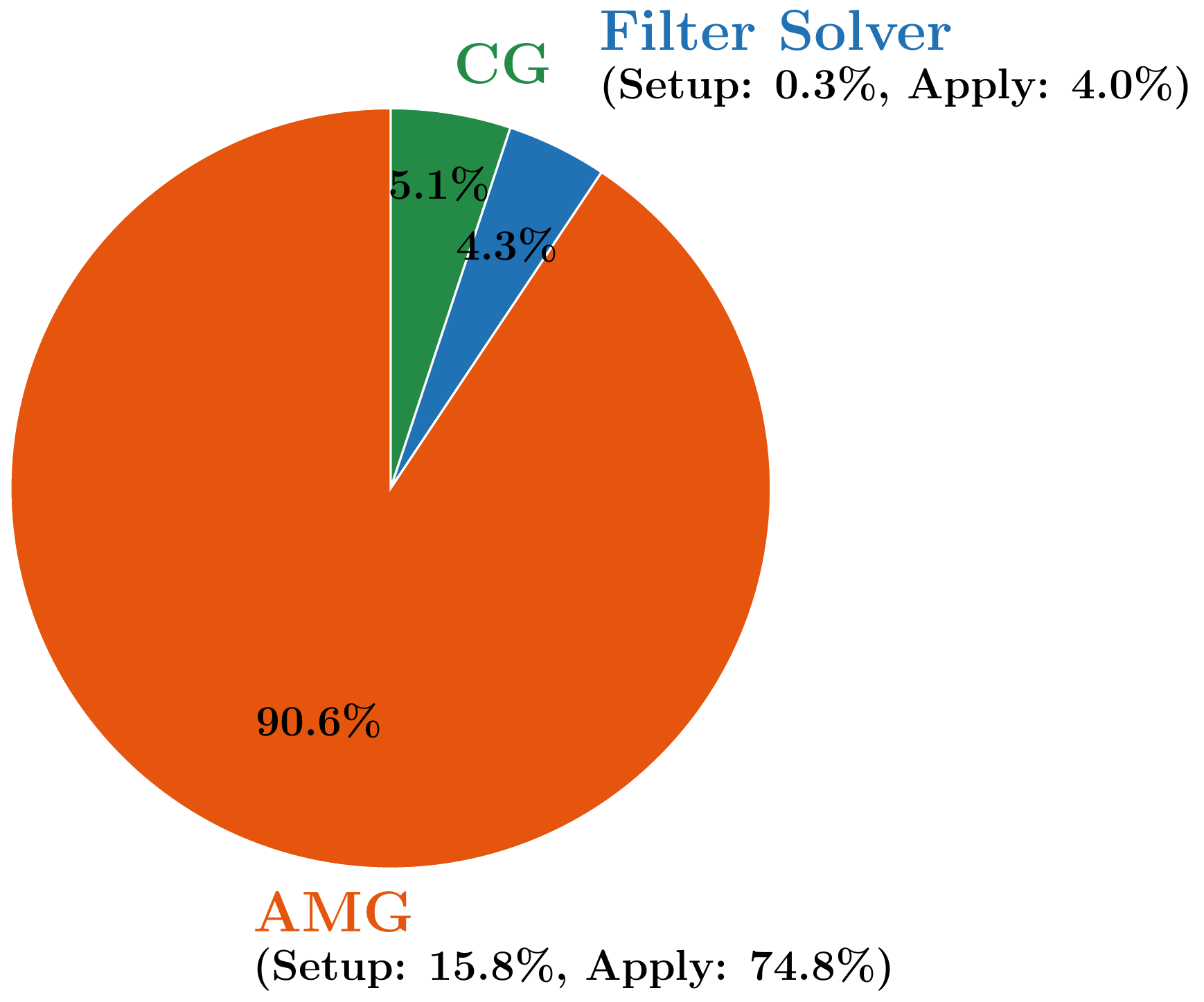}
        \caption{Mesh 5}
        \label{fig:pie_mesh5}
    \end{subfigure}
    \begin{subfigure}{0.32\textwidth}
        \centering
        \includegraphics[width=1.05\linewidth]{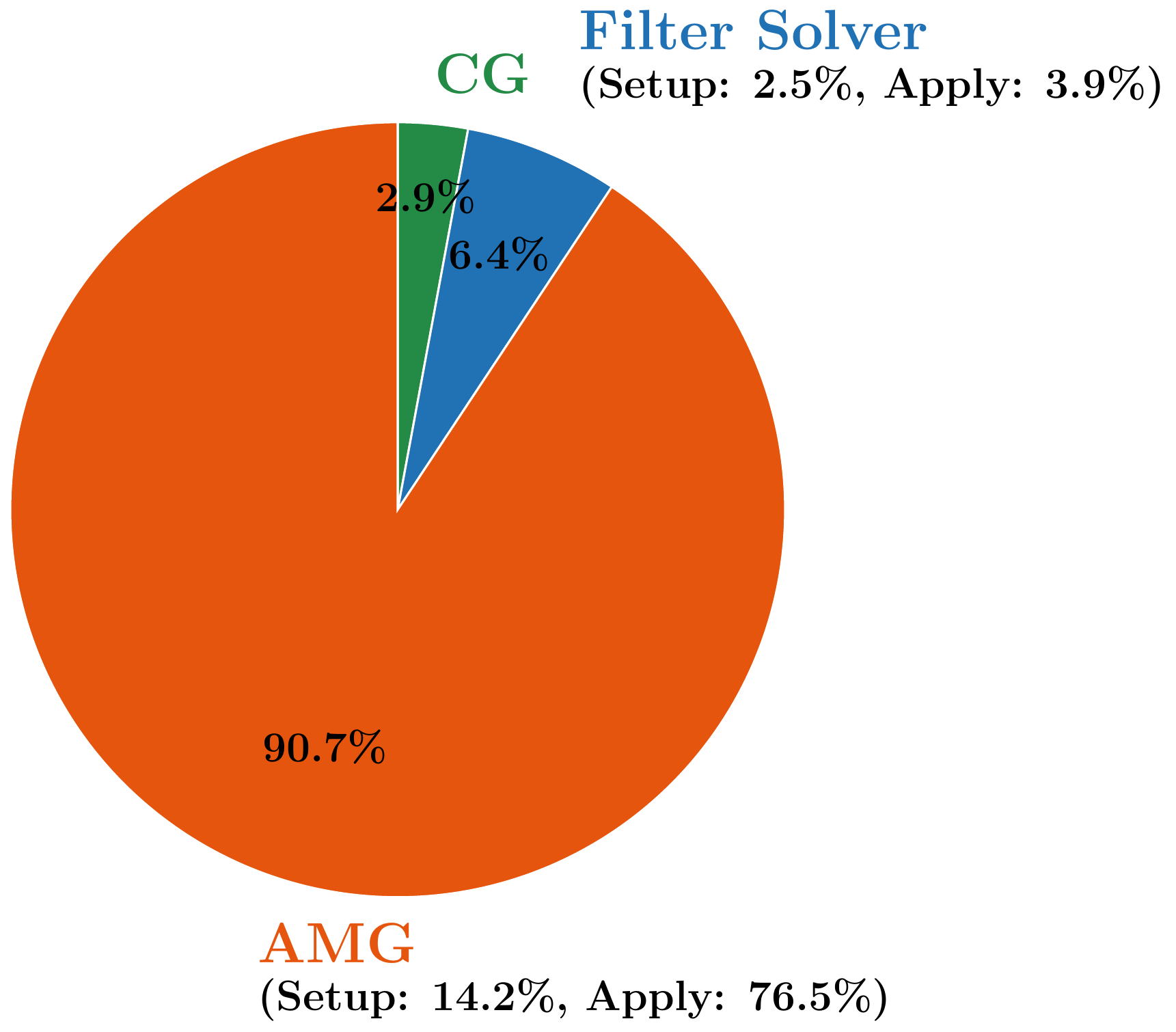}
        \caption{Mesh 6}
        \label{fig:pie_mesh6}
    \end{subfigure}
        \caption{\color{revised}Runtime breakdown of AMGF for representative large-scale runs (Meshes 4–6). The total runtime is partitioned into AMG setup, AMG application, the contact subspace filter solver setup and application, and the remaining CG operations. AMG setup and application dominates with \(\approx 90\%\)  of the total cost. The filter solver on the contact subspace contributes approximately \(3-7\%\), which is comparable with the cost of the remaining CG operations. This confirms that the filtered subspace solve introduces only a modest overhead while enabling the iteration reductions that drive the overall speedups.}
        \label{fig:amgf_pies}
\end{figure}

}

%% file: conclusions.tex
\section{Conclusions}
\label{section:conclusion}

In this paper, we introduced AMGF, a scalable and robust preconditioner designed to enhance the performance of IP methods for large-scale contact mechanics problems. By augmenting the classical AMG solver with a filtering step that targets a small subspace associated with contact constraints, AMGF directly addresses one of the primary sources of ill-conditioning in these systems. This targeted correction is both effective and computationally inexpensive.

Theoretical analysis shows that the condition number of the AMGF--preconditioned system remains bounded and close to that of the AMG--preconditioned system in the absence of contact. Numerical experiments across a range of challenging examples, including nonlinear materials and large deformations, confirm this behavior. In all cases, AMGF achieves reliable convergence, even where classical AMG fails, and exhibits the predicted bounded iteration growth with respect to both mesh refinement and contact constraint enforcement, consistent with the theoretical results. {\color{revised} The computational study further demonstrates that the additional filtering step remains inexpensive and leads to substantial speedups over classical AMG. Moreover, the results suggest that a dynamic strategy that switches between AMG and AMGF depending on the stage of the IP iterations could yield even greater efficiency.} 

AMGF offers a practical and effective solver for scalable contact simulations. More broadly, it applies to a broader class of problems, optimization-based or otherwise, where solver performance deteriorates due to a low-dimensional subspace, such as those arising from localized constraints or interface conditions. Its design is simple, general, and easily adaptable. Future work will explore extensions to more complex contact phenomena and broader integration into multiphysics simulation frameworks.